\documentclass{article}[11pt]
\usepackage{a4}
\usepackage[english]{babel}
\usepackage[T1]{fontenc}
\usepackage[all]{xy}
\usepackage{amssymb}
\usepackage{amsmath}
\usepackage{amsthm}
\usepackage{graphicx}
\usepackage{caption}
\usepackage{rotating}
\usepackage{latexsym}
\usepackage{float}

\newtheorem{thm}{Theorem}[section]
\newtheorem{prop}[thm]{Proposition}
\newtheorem{lem}[thm]{Lemma}

\newtheorem{defn}[thm]{Definition}

\numberwithin{equation}{section}   

\newcommand{\rdbb}{\mathbb{R}^d}
\newcommand{\lra}{\longrightarrow}

\newcommand{\rbb}{\mathbb{R}}

\newcommand{\zbb}{\mathbb{Z}}

\newcommand{\embthm}{\overline{\mathrm{Emb}}_c(\coprod_{i=1}^r \mathbb{R}^{m_i}, \rdbb)}
\newcommand{\emb}{\overline{\mbox{Emb}}_c(\coprod_{i=1}^r \mathbb{R}^{m_i}, \rdbb)}
\newcommand{\rmodthm}{\underset{\Omega}{\mathrm{hRmod}}}
\newcommand{\rmod}{\underset{\Omega}{\mbox{hRmod}}}
\newcommand{\qbb}{\mathbb{Q}}

\newcommand{\codim}{d > 2 \mathrm{max}\{m_i| \ 1 \leq i \leq r\} + 1}
\newcommand{\zotimes}{Z_{V^{\otimes  \bullet}}}
\newcommand{\lie}{{{\mathcal L}ie}((\bullet))}

\newcommand{\empi}{\mathcal{E}_{\pi}^{m_1, \cdots, m_r, d}}

\newcommand{\Det}{\mathrm{Det}}

\newcommand{\MODL}{\mathbf{{Mod}}({\mathcal L}_\infty)}
\newcommand{\MODLdet}{\mathbf{{Mod}}_{\Det}({\mathcal L}_\infty)}

\newcommand{\FCOM}{\mathrm{F}{\mathcal C}om}
\newcommand{\FCOMdet}{\mathrm{F}_\Det{\mathcal C}om}

\newcommand{\Ch}{{\mathbb C}\mathrm{h}}

\newcommand{\nbb}{\mathbb{N}}

\newcommand{\cbb}{\mathbb{C}}

\newcommand{\xvecs}{\vec{x}^{\vec{s}}}

\title{\textbf{Euler characteristics for 
spaces of string links
and the modular envelope of ${\mathcal L}_\infty$}}
\date{}

\author{Paul Arnaud Songhafouo Tsopm\'en\'e \\
 Victor Turchin \footnote{The second author is partially supported by the Simons Foundation \lq\lq{}Collaboration grant for mathematicians,\rq\rq{} award ID:~519474.} \\
}

\begin{document}
\maketitle

\begin{abstract}
We make calculations in graph homology which further understanding of the topology of spaces of string links, in particular calculating the Euler 
characteristics of finite-dimensional summands in their homology and homotopy. In doing so, we also determine the supercharacter of the symmetric group action on the positive arity components of the modular envelope of ${\mathcal L}_\infty$. 
\end{abstract}


\setcounter{section}{-1}

\section{Introduction}\label{s:intro}

\sloppy

Let $\mbox{Emb}_c(\coprod_{i=1}^r \mathbb{R}^{m_i}, \rdbb)$ be the space of smooth embeddings $f \colon \coprod_{i=1}^r \rbb^{m_i} \hookrightarrow \rdbb$ that coincide outside a compact set with a fixed affine embedding $\iota$.  
Such embeddings are called \textit{string links}. Let  $\mbox{Imm}_c(\coprod_{i=1}^r \mathbb{R}^{m_i}, \rdbb)$ be the space of smooth immersions 
with the same behavior at infinity. 
In this paper and in~\cite{songhaf_tur16} we study the homotopy fiber over $\iota$ of the obvious inclusion
$\mbox{Emb}_c(\coprod_{i=1}^r \mathbb{R}^{m_i}, \rdbb) \subset \mbox{Imm}_c(\coprod_{i=1}^r \mathbb{R}^{m_i}, \rdbb)$, which we denote $\emb$.
In~\cite{songhaf_tur16}
we built on Goodwillie-Weiss calculus  
\cite{good_weiss99} to define two types of complexes computing the rational homology and homotopy groups of $\emb$, in the case when $\codim$, 
which split into a direct sum  of  finite complexes. 
 In Theorem~\ref{gen_function_thm} and Theorem~\ref{gen_function_htpy_thm} below we compute the generating functions for the Euler characteristic of these summands, giving to our knowledge the first known \lq\lq{}lower-bounds\rq\rq{} for the homology and homotopy of these spaces.

The first type of complexes we study, which we call {\it Koszul complexes}, are built up from the (co)homology of configuration
spaces of points in $\rbb^{m_i}$ and in $\rdbb$. These complexes are more amenable to computation. 
The second type of complexes, which we call {\it hairy graph-complexes}, are similar to
 Kontsevich\rq{}s commutative operad graph-complexes~\cite{kont93} except that our graphs
 are allowed to have external univalent vertices, called {\it hairs}, and thus are more similar to graph-complexes
 studied  by Conant, Kassabov and Vogtmann \cite{con_kas_vogt13}. We show that these complexes are built up 
from the positive arity  components of the modular envelope  $\MODL$ of the ${\mathcal L}_\infty$ operad in case of odd $d$,  
and from their twisted version $\MODLdet$ for even $d$. 
As a consequence we  determine in Theorem~\ref{t:Z_L_infty} the \textit{supercharacter} 
 of the symmetric group action on the positive arity components of these modular envelopes, which up to
 regrading are the Feynman transforms $\FCOMdet$ and $\FCOM$, respectively, defined by Getzler and Kapranov in~\cite{getzler_kapranov98}.
 
 
 Spaces of embeddings are central objects of study in differential topology, with for example their components corresponding to isotopy classes and 
 their higher homotopy groups useful for generating interesting phenomena~\cite{bud08}. 
The main applications of modular operads are in
mathematical physics, in particular in Chern-Simons theory, open and closed string field theories, Batalin-Vilkovisky formalism,  
and formal geometry~\cite{barannikov13,chuang13,hamilton09,kauffmann17,kont93,markl01,markl16}. 
Graph-complexes and modular operads  also have applications in topology and geometry, in 
particular in the study of  the moduli spaces of curves with marked points, automorphisms of free groups, and Vassiliev invariants of knots and string links~\cite{con_kas_vogt13,getzler_kapranov98,kondo17,kont93,smile85, hin_vaintrob02}.
 In fact  $\mathbf{{Mod}}({\mathcal L}ie)=H_0\MODL$, the modular envelope of ${\mathcal L}ie$, is
 related to  Bar-Natan\rq{}s space of unitravalent graphs modulo AS and IHX relations from Vassiliev theory
 in exactly the same
 way as the modular envelope of ${\mathcal L}_\infty$ is related to our hairy graph-complexes.

 Direct computations  of the supercharacter of the symmetric group
 action on $\MODL$ and $\MODLdet$ using graph complexes have previously been much less tractable than we find here.
 Such were first attempted in case of arity zero, where 
 these correspond to 
 the well known Kontsevich graph-complexes associated to the commutative operad~\cite{kont93}, by T.~Willwacher and M.~\v{Z}ivkovi\'c in~\cite{wil_zh}.
 The original paper of Getzler and Kapranov~\cite{getzler_kapranov98}, where the notion of a modular operad is introduced, gives
 a universal method to compute the supercharacter of a Feynman transform of any modular operad, but explicit calculations were done 
 in~\cite{getzler_kapranov98}  only for the associative operad.
 Thus our work deepens  understanding of modular operads and graph-complexes as well as embedding spaces.

 \subsection{Basic definitions and previous results}\label{definitions}
We recall the main results of \cite{songhaf_tur16}.  
To define complexes   arising from Goodwilliie-Weiss  calculus and computing the rational homology and homotopy groups of $\emb$, 
we first define two categories: ${\mbox{Rmod}_\Omega}$ and  ${\mbox{Rmod}_\Gamma}$. Let $\Omega$ denote the category of finite unpointed sets $\underline{n}=\{1\ldots n\}$, $n\geq 0$, and surjections. Define a \textit{right $\Omega$-module} as a contravariant functor from $\Omega$ to any given category. The category of right $\Omega$-modules in chain complexes over the rationals is denoted ${\mbox{Rmod}_\Omega}$. This category can be endowed with several model structures. We choose the one called \textit{projective model structure}. In that model, weak equivalences are quasi-isomorphisms, fibrations are degreewise surjective maps~\cite{hovey}. Given two objects $P$ and $Q$ in ${\mbox{Rmod}_\Omega}$, we write ${\mbox{Rmod}_\Omega}(P,Q)$ for the space (chain complex) of morphisms between them, and we write $h{\mbox{Rmod}_\Omega}(P,Q)$ for the derived mapping space. For specific computations we will need to apply this construction only to $\Omega$-modules with zero differential, in which case $\rmod(-,-)$  can be expressed as a product of $\mathrm{Ext}$ groups.
%

 For a pointed topological space $X$, define the functor $X^{\wedge \bullet} \colon \Omega \lra \mbox{Top}$ from $\Omega$  to topological spaces by $X^{\wedge \bullet}(\underline{n}) = X^{\wedge n}$. 
Here  \lq\lq$\wedge n$\rq\rq{}  is the $n$-fold smash product operation, with $X^{\wedge 0} = S^0$  the two-point  space. For a morphism $f$ in $\Omega$, $X^{\wedge \bullet}(f)$ is induced by the diagonal maps. 
Let $\widetilde{C}_*(-)$ denote the reduced singular chain complex functor, so the functors $\widetilde{C}_*(X^{\wedge \bullet})$ and $\widetilde{H}_*(X^{\wedge \bullet})$ (with zero differential) are objects of ${\mbox{Rmod}_\Omega}$. 
Note that in case $X$ is a suspension, any \textit{strict surjection} (that is, a surjection which is not a bijection) acts on $\widetilde{H}_*(X^{\wedge \bullet})$
as a zero map.

Let $\Gamma$ be the category whose objects are finite pointed sets $n_+ = \{0, 1, \cdots, n\}$, with $0$ as the basepoint, and whose morphisms are pointed maps.\footnote{We follow Pirashvili\rq{}s notation for 
$\Omega$ and $\Gamma$, see~\cite{pira00,pirash00}. In the literature following Segal~\cite{segal_74}, one often denotes by $\Gamma$ the
opposite category. 
We choose Pirashvili\rq{}s notation as in the sequel we use the higher order Hochschild-Pirashvili homology
 defined in~\cite{pirash00} by means of these categories $\Omega$ and $\Gamma$.}   
The category of contravariant functors from $\Gamma$ to chain complexes is denoted ${\mbox{Rmod}_\Gamma}$. Objects of that category are called \textit{right $\Gamma$-modules}. As an example of a right $\Gamma$-module, we have the homology $H_*(C (\bullet, \rdbb), \qbb)$, $d\geq 2$, where $C(k, \rdbb)$ denotes the configuration space of $k$ labeled points in $\rdbb$. One can see that $H_*(C (\bullet, \rdbb), \qbb)$ is indeed a right $\Gamma$-module as follows. 
 First, since there is a morphism $\mbox{Com} \lra H_*(C(\bullet, \rdbb), \qbb)$ of operads from the commutative operad $\mbox{Com} = H_0(C(\bullet, \rdbb), \qbb)$ to the homology $H_*(C(\bullet, \rdbb), \qbb) =  H_*(B_d(\bullet), \qbb)$ of the little $d$-disks operad, it follows that $H_*(C(\bullet, \rdbb), \qbb)$ is an \textit{infinitesimal bimodule} (see \cite[Definition 3.8]{aro_tur12} or \cite[Definition 4.1]{turchin10}) over $\mbox{Com}$. Secondly,  the category of infinitesimal bimodules over $\mbox{Com}$ is equivalent to the category of right $\Gamma$-modules (\cite[Corollary~4.10]{aro_tur12} or \cite[Lemma~4.3]{turchin10}). One can also show that the sequence  $\qbb \otimes \pi_*C(\bullet,\rdbb)$, $d\geq 3$,  has a natural structure of a right $\Gamma$-module. 

The two categories ${\mbox{Rmod}_\Omega}$ and  ${\mbox{Rmod}_\Gamma}$ we just defined are equivalent. To prove it, Pirashvili \cite{pira00} constructed a functor $\mbox{cr} \colon {\mbox{Rmod}_\Gamma} \lra {\mbox{Rmod}_\Omega}$, called \textit{cross effect}, and showed that it is actually an equivalence of categories.  Let 
$\widehat{H}_*(C(\bullet, \rdbb), \qbb)$  (or $\mbox{respectively } \qbb \otimes \widehat{\pi}_*C(\bullet, \rdbb)$) 
denote the cross effect of $H_*(C(\bullet, \rdbb), \qbb)$ (respectively of $\qbb \otimes \pi_*C(\bullet, \rdbb)$).

Theorems~0.1 and~0.2 in \cite{songhaf_tur16} express the rational homology and homotopy
groups of $\mathcal{L}:=\emb$  
 as the homology groups of derived mapping complexes of right $\Omega$-modules.  More precisely, for $d > 2 \mathrm{max}\{m_i| \ 1 \leq i \leq r\} + 1$, there are isomorphisms:\footnote{In other words,
 these formulae express the rational homology and homotopy of $\emb$ as the higher order
 Hochschild homology over the space $\vee_{i=1}^r S^{m_i}$ with coefficients in the $\Gamma$-modules
 $H_*(C(\bullet, \rdbb), \qbb)$ and $\qbb \otimes \pi_*C(\bullet, \rdbb)$, respectively, see~\cite{songhaf_tur16}.}


\begin{equation} \label{rational_homology_iso}
H_*(\mathcal{L}, \qbb) \cong H\left(\rmodthm\left(\widetilde{H}_*((\vee_{i=1}^r S^{m_i})^{\wedge \bullet}, \qbb), \widehat{H}_*(C(\bullet, \rdbb), \qbb)\right)\right).
\end{equation}
\begin{equation} \label{rational_homotopy_iso}
\qbb \otimes \pi_*\mathcal{L} \cong H\left(\rmodthm\left(\widetilde{H}_*((\vee_{i=1}^r S^{m_i})^{\wedge \bullet}, \qbb), \qbb \otimes \widehat{\pi}_*C(\bullet, \rdbb)\right)\right).
\end{equation}

\sloppy


For our purposes we need to split the right-hand sides of (\ref{rational_homology_iso}) and (\ref{rational_homotopy_iso}). In the sequel a sequence of $r$ integers $s_1, \cdots, s_r$ will be written as $\vec{s}$.  Also we will write  $|\vec{s}|$ for $s_1 + \cdots +s_r$, and $\Sigma_{\vec{s}}$ for $\Sigma_{s_1} \times  \cdots \times \Sigma_{s_r}$. If $x_1, \cdots, x_r$ is another sequence, we will write $\vec{s} \cdot \vec{x}$ for $s_1x_1 + \cdots + s_rx_r$, and $\vec{x}^{\vec{s}}$ for $\prod_{i} x_i^{s_i}$. 
We will also write $\vec{s}\geq 0$ if $s_i\geq 0$, $i=1\ldots r$.

Since $\vee_{i=1}^r S^{m_i}$ is a suspension, any strict surjection acts on  $\widetilde{H}_*((\vee_{i=1}^r S^{m_i})^{\wedge \bullet}, \qbb)$ as zero. This implies that the $\Omega$-module  $\widetilde{H}_*((\vee_{i=1}^r S^{m_i})^{\wedge \bullet}, \qbb)$ splits as follows:
\begin{equation} \label{splitting_htilde}
\widetilde{H}_*((\vee_{i=1}^r S^{m_i})^{\wedge \bullet}, \qbb) \cong \underset{\vec{s} \geq 0}{\bigoplus} Q^{\vec{m}}_{\vec{s}}, 
\end{equation}
where $Q^{\vec{m}}_{\vec{s}}$ is the right $\Omega$-module defined by 
  \begin{equation} \label{qs1sr}
Q_{\vec{s}}^{\vec{m}}(k) = \left\{ \begin{array}{lll}
                                 0 & \mbox{if} & k \neq |\vec{s}|; \\
																\mbox{Ind}^{\Sigma_k}_{\Sigma_{\vec{s}}} \widetilde{H}_*(S^{\vec{s} \cdot \vec{m}} ; \mathbb{Q}) & \mbox{if} & k = |\vec{s}|.
                                \end{array} \right.
   \end{equation}
		
Consider now $\widehat{H}_*(C(\bullet, \rdbb), \qbb)$ and $\qbb \otimes \widehat{\pi}_*C(\bullet, \rdbb)$ that appear in (\ref{rational_homology_iso}) and (\ref{rational_homotopy_iso}) respectively. One has the splittings of $\Omega$-modules:
\begin{equation} \label{splitting_hhat_H}
\widehat{H}_*(C(\bullet, \rdbb), \qbb) = \underset{t \geq 0}{\prod} \widehat{H}_{t(d-1)}(C(\bullet, \rdbb), \qbb),
\end{equation}
and
\begin{equation} \label{splitting_hhat_Pi}
\qbb \otimes \widehat{\pi}_*C(\bullet, \rdbb) = \underset{t \geq 0}{\prod} \qbb \otimes \widehat{\pi}_{t(d-2)+1}C(\bullet, \rdbb). 
\end{equation}
Combining (\ref{rational_homology_iso}), (\ref{rational_homotopy_iso}), (\ref{splitting_htilde}), (\ref{splitting_hhat_H}), and (\ref{splitting_hhat_Pi}), we get the following splittings 
\begin{multline} \label{final_splitting_H}
H_*(\emb, \qbb) \cong \underset{\vec{s}, t}{\prod} \rmod \left(Q^{\vec{m}}_{\vec{s}},  \widehat{H}_{t(d-1)}(C(\bullet, \rdbb), \qbb) \right)\\
\cong \underset{\vec{s}, t}{\bigoplus} \rmod \left(Q^{\vec{m}}_{\vec{s}},  \widehat{H}_{t(d-1)}(C(\bullet, \rdbb), \qbb) \right).
\end{multline}
\begin{multline} \label{final_splitting_pi}
\qbb \otimes \pi_*(\emb) \cong \underset{\vec{s}, t}{\prod} \rmod \left(Q^{\vec{m}}_{\vec{s}},  \qbb \otimes \widehat{\pi}_{t(d-2)+1}C(\bullet, \rdbb) \right)\\
\cong \underset{\vec{s}, t}{\bigoplus} \rmod \left(Q^{\vec{m}}_{\vec{s}},  \qbb \otimes \widehat{\pi}_{t(d-2)+1}C(\bullet, \rdbb) \right).
\end{multline}
The product is replaced by the direct sum because only finitely many factors contribute for any given degree, as 
$d > 2 \mathrm{max}\{m_i| \ 1 \leq i \leq r\} + 1$, using the graph-complexes we explicitly described in \cite[Remark 2.4]{songhaf_tur16}. (For (\ref{final_splitting_pi}) this is true even for a weaker constraint $d > \mathrm{max}\{m_i| \ 1 \leq i \leq r\} + 2$. Moreover, we conjecture in~\cite[Section~3]{songhaf_tur16} that~\eqref{final_splitting_pi} 
holds always in that  range for $*\geq 0$.)

\subsection{Statements of main results}
We can now state the main results of this paper.  For $\vec{s}\geq 0$ and $t \geq 0$, let $\mathcal{X}_{\vec{s}, t}$ be the Euler characteristic of the  summand of (\ref{final_splitting_H}) indexed by 
$\vec{s}, t$. The associated generating function is 
$F^H_{\vec{m}, d}(x_1, \cdots, x_r, u) = \underset{\vec{s}, t \geq 0}{\sum} \mathcal{X}_{\vec{s}, t} \cdot u^t \vec{x}^{\vec{s}}$.

Let $\Gamma(-)$ denote the gamma function, and let $\mu(-)$ denote the standard M\"obius function. Given a variable $x$ and an integer $l \geq 1$, let $E_l(x)$ denote the sum 
\begin{equation} \label{eq:E_l}
E_l(x) = \frac{1}{l} \sum_{p|l} \mu(p)x^{\frac{l}{p}}. 
\end{equation}
The following result computes the generating function above.

\begin{thm} \label{gen_function_thm}
Assume that $d > 2 \mathrm{max}\{m_i| \ 1 \leq i \leq r\} + 1$. The generating function $F^H_{\vec{m}, d}(x_1, \cdots, x_r, u)$  is given by the formula
\begin{equation} \label{gen_function_formula}
F^H_{\vec{m}, d}(x_1, \cdots, x_r, u) = \prod_{l = 1}^{+\infty} \frac{\Gamma((-1)^{d-1}E_l(\frac{1}{u}) - \sum_{i=1}^r (-1)^{m_i-1}E_l(x_i))}{((-1)^{d-1}lu^l)^{\sum_{i=1}^r (-1)^{m_i-1}E_l(x_i)} \Gamma((-1)^{d-1} E_l(\frac{1}{u}))}, 
\end{equation}
where each factor is understood as the asymptotic expansion of the underlying function when $u$ is complex and $(-1)^{d-1} u^l \lra +0$ and $x_1, \cdots, x_r$ are considered as fixed parameters. 
\end{thm}
When $r=1$, the formula (\ref{gen_function_formula}) coincides with that of \cite[Theorem 6.1]{aro_tur13}.  

Using \eqref{gen_function_formula} we also compute the generating function of the Hodge splitting in the rational homotopy, which is our second main result. First we need a couple more definitions. 
Let $B_p$ denote the $p$th Bernoulli number,  
so that
 $
 \sum_{p\geq 0} \frac{B_px^p}{p!} = \frac x{e^x-1}.
 $
 Recall that $B_{2n+1}=0$, $n \geq 1$. Bernoulli's summation formula equates $1^j + 2^j + \cdots + n^j$ with $S_j(n)$ where
 \begin{equation} \label{sjx_defn}
S_j(x) = \frac{1}{j+1} \sum_{p = 0}^j (-1)^p\binom{j+1}{p} B_p x^{j+1-p}, \, j \geq 1. 
\end{equation}
Define also  $F_l(u)$ by 
\begin{equation}\label{eq:F_l}
F_l(u) = lu^l E_l(\frac{1}{u}) = \sum_{t|l} \mu (t) u^{l-\frac{l}{t}} = 1- u^{l-l/p_1}-u^{l-l/p_2}+u^{l-l/p_1p_2}+\ldots
+\mu(l)u^{l-1},
\end{equation}
where $p_1$ and $p_2$ are the first prime factors of $l$. Notice that one always has $F_l(0)=1$.

Similarly, consider $F^{\pi}_{\vec{m}, d}(x_1, \cdots, x_r, u)$ associated to the splitting (\ref{final_splitting_pi}). 

\begin{thm} \label{gen_function_htpy_thm}
The generating function $F^{\pi}_{\vec{m}, d}(x_1, \cdots, x_r, u)$ is given by the formula 
\begin{align*}
F^{\pi}_{\vec{m}, d}(x_1, \cdots, x_r, u) =  & \sum_{k, l, j \geq 1} \frac{\mu(k)}{kj} S_j \left( \sum_{i=1}^r (-1)^{m_i-1} E_l(x_i^k) \right) 
\left(\frac{(-1)^{d-1} l u^{kl}}{F_l(u^k)} \right)^j - \\
 &\sum_{k, l \geq 1} \sum_{i=1}^r \frac{\mu(k)}{k} (-1)^{m_i-1} E_l(x_i^k) \ln (F_l(u^k)), 
\end{align*}
where the polynomials $E_l$, $F_l$, $S_j$ are respectively defined by \eqref{eq:E_l}, \eqref{eq:F_l}, \eqref{sjx_defn}.
\end{thm}

The latter result essentially encodes the same information as the \textit{supercharacter} (which will be defined in Subsection~\ref{ss:31}) 
 of the symmetric group action on the positive arity components of $\MODL$ and $\MODLdet$.

i

\begin{thm}\label{t:Z_L_infty}
\sloppy
The supercharacters of the symmetric group action on the modular envelope  $\{\MODL((k))\}_{k\geq 0}$   of ${\mathcal L}_\infty$ and on the $\Det$-twisted modular envelope $\{\MODLdet((k))\}_{k\geq 0}$ of ${\mathcal L}_\infty$ are  described by the cycle index sums
as follows
\begin{multline}\label{eq:ch_odd}
\hbar^{-1}\left(\frac{p_1^2+p_2}2\right)+\Ch(\MODL)=
\hbar^{-1} Z_{\mathcal{X}\MODL}(\hbar;p_1,p_2,p_3,\ldots) =  \\
 \Psi^{odd}(\hbar)+\sum_{k, l, j \geq 1} \frac{\mu(k)}{kj} S_j \left( \frac 1l\sum_{a|l}\mu\left(\frac la\right)
\frac{p_{ak}}{\hbar^{ak}} \right) 
\left(\frac{l \hbar^{kl}}{F_l(\hbar^k)} \right)^j -  \\
 \sum_{k, l \geq 1}  \frac{\mu(k)}{kl}\left( \sum_{a|l} \mu\left(\frac la\right)\frac{p_{ak}}{\hbar^{ak}}\right)  \ln (F_l(\hbar^k)), 
\end{multline}
\begin{multline}\label{eq:ch_even}
\hbar^{-1}\left(\frac{p_1^2+p_2}2\right)+\Ch(\MODLdet)=
\hbar^{-1} Z_{\mathcal{X}\MODLdet}(\hbar;p_1,p_2,p_3,\ldots) =  \\
 \Psi^{even}(\hbar)-\sum_{k, l, j \geq 1} \frac{\mu(k)}{kj} S_j \left( - \frac 1l\sum_{a|l}\mu\left(\frac la\right)
\frac{p_{ak}}{\hbar^{ak}} \right) 
\left(\frac{-l \hbar^{kl}}{F_l(\hbar^k)} \right)^j -   \\
 \sum_{k, l \geq 1}  \frac{\mu(k)}{kl}\left( \sum_{a|l} \mu\left(\frac la\right)\frac{p_{ak}}{\hbar^{ak}}\right)  \ln (F_l(\hbar^k)), 
\end{multline}
where the variable $\hbar$ is responsible for the genus. The functions $\Psi^{odd}(\hbar)$ and
$\Psi^{even}(\hbar)$ are expressed in terms of Willwacher-\v{Z}ivkovi\'c\rq{} functions $P^{odd}(s,t)$, 
$P^{even}(s,t)$ from~\cite[Theorem~1]{wil_zh} as follows
\begin{equation}\label{eq:psi_odd}
\Psi^{odd/even}(\hbar)=\sum_{l\geq 1}\frac{\mu(l)}{l}\ln\left(1+P^{odd/even}(\mp\hbar^{-l},\pm\hbar^l)\right),
\end{equation}
\end{thm}

The first line in~\eqref{eq:ch_odd} (and also in~\eqref{eq:ch_even}) is to account the difference 
of the Getzler-Kapranov notation~\cite{getzler_kapranov98} (left-hand side) with ours (right-hand side). 
The factor $\hbar^{-1}$ in front of $Z_{\mathcal{X}\MODL}$ appears because we consider the grading by
the genus (first Betti number) of the graphs, while in~\cite{getzler_kapranov98} the additional grading is  
minus the Euler characteristics. The summand $\hbar^{-1}\left(\frac{p_1^2+p_2}2\right)$ in the left-hand side is because Getzler-Kapranov consider non-unital modular operads, while we follow 
Hinich-Vaintrob\rq{}s conventions~\cite{hin_vaintrob02}, and allow operads have a unit.

Some readers might be more familiar with the Feynman transform rather than with the modular envelope. They are related one to another through a regrading $\FCOMdet=\Sigma\mathfrak{s}\MODL$,
$\FCOM=\Sigma\mathfrak{s}\MODLdet$, see Lemma~\ref{l:fcom}, which implies
\[
\Ch(\FCOMdet)(\hbar;p_1,p_2,p_3,\ldots) = - \Ch(\MODL)(\hbar;-p_1,-p_2,-p_3,\ldots),
\]
\[
\Ch(\FCOM)(\hbar;p_1,p_2,p_3,\ldots) = - \Ch(\MODLdet)(\hbar;-p_1,-p_2,-p_3,\ldots).
\]

Comparing our answer with $\Ch(\mathrm{F}_\Det\mathcal{A}ss)$ computed 
in~\cite[Theorem~9.18]{getzler_kapranov98}, we see that the positive arity part (which we computed) is the 
\lq\lq{}easy part\rq\rq{} of the expression. The arity zero part, expressed by $\Psi^{odd}(\hbar)$,
 $\Psi^{even}(\hbar)$, is the difficult one, which is still not explicitly computed. To recall Willwacher and 
 \v{Z}ivkovi\'c~\cite{wil_zh} expressed the generating functions $P^{odd}(s,t)$, 
$P^{even}(s,t)$ of the dimensions of the graph-complexes  $\mathrm{fGC}_d$ 
(where $d$ is either {\it odd} or {\it even})\footnote{The graph-complexes $\mathrm{fGC}_d$ are defined similarly as our hairy graph-complexes except that
their graphs do not have hairs (univalent vertices) and can be disconnected.} as an infinite sum of 
rather complicated expressions, which do not make sense when one passes to the 
generating function of the Euler characteistics by taking $s=\mp\hbar^{-1}$, $t=\pm\hbar$. 
The right-hand side of~\eqref{eq:psi_odd}  is the plethystic logarithm
\cite[Equations (8.6), (8.1.3\rq{})]{getzler_kapranov98}, which works in any arity and in particular in arity zero,
and which determines the Euler characters (supertraces for higher arities) of the connected part of the 
graph-complexes.\footnote{Compare with our Lemma~\ref{gen_function_htpy-lemma} and also
with \cite[Lemma~3]{wil_zh}.} We added~1 inside the logarithm to take into account the empty graph.

\subsection{Outline of the paper.}

In Section~\ref{generating_function_homology} we prove Theorem~\ref{gen_function_thm}, which presents a formula for the generating function of Euler characteristics of summands in the homological splitting (\ref{final_splitting_H}).  
 We end with Subsection~\ref{some_special_cases_gen_function_homology}, which gives formulas obtained from (\ref{gen_function_formula}) by taking $m_i=1$, $x_i=\pm 1$, $1 \leq i \leq r$.

In Section~\ref{generating_function_homotopy} we prove three results: Theorem~\ref{gen_function_htpy_thm},
Theorem~\ref{genius_zero_thm} and Theorem~\ref{genius_one_thm}. The first one 
computes the generating function of  Euler characteristics of summands in the homotopical splitting (\ref{final_splitting_pi}). The key point is Lemma~\ref{gen_function_htpy-lemma}, which presents the generating function in homotopy as a sum of  logarithms  of the generating function in homology. 
 The second  result  (respectively the third result) computes the generating function for the homology ranks  of the summands in~\eqref{final_splitting_pi} of genus zero (respectively of genus one).  Both from  Theorem~\ref{genius_zero_thm} and  Theorem~\ref{genius_one_thm} one can see the exponential growth of the Betti numbers of the rational homotopy of $\emb$, $r\ge 2$. 

In Section~\ref{ss:modular} we prove Theorem~\ref{t:Z_L_infty}, which determines the supercharacter of the symmetric group action on the positive arity components of the modular envelope of ${\mathcal L}_\infty$.  This result is an interesting byproduct  for the theory of operads, obtained from the computations of the previous section. 

 In the appendix we produce, using the generating function from Theorem~\ref{gen_function_htpy_thm}, tables of Euler characteristics of the summands in the homotopical splitting (\ref{final_splitting_pi}).

\textbf{Acknowledgements:} 
This work has been supported by Fonds de la Recherche Scientifique-FNRS (F.R.S.-FNRS), that the authors acknowledge. It has been also supported by the Kansas State University (KSU), where this paper was partially written during the stay of the first author, and which he thanks for hospitality. The authors are also grateful
to Dev Sinha for his  thoughtful editing work -- his comments and numerous suggestions to
improve the manuscript.

\section{Generating functions of Euler characteristics in homology} \label{generating_function_homology}

The goal of this section is to prove Theorem~\ref{gen_function_thm}, which gives a formula for the generating function $F^H_{\vec{m}, d}(x_1, \cdots, x_r, u) $ of the Euler characteristics of the summands in the homological splitting of  the space of string links.


We first recall the definition of Koszul complexes
computing~\eqref{final_splitting_H}. In short these complexes were obtained in~\cite{songhaf_tur16} by taking the projective resolution of the source $\Omega$-modules. 

\subsection{Koszul complexes}


\begin{defn} \label{otimes_hat_defn}
Let $V = \{V(n)\}_{n \geq 0}$ and $W = \{W(n)\}_{n \geq 0}$ be two symmetric sequences. Define a new symmetric sequence $V \widehat{\otimes} W$ by 
$V \widehat{\otimes} W (n) = \bigoplus_{p+q =n} \mathrm{Ind}^{\Sigma_n}_{\Sigma_p \times \Sigma_q} V(p) \otimes W(q).$
\end{defn}

In practice we need to deal with multigraded vector spaces. Besides the usual homological degree, they will have the 
Hodge multi-grading $(s_1, \ldots, s_r)$ and grading by complexity $t$. As usual when we take tensor product all the degrees
get added.


Recalling the notation $\widehat{H}_*(-)$ from the introduction, we have the following result. 

\begin{prop} \cite[Proposition 4.5]{songhaf_tur16} \label{total_complex_homology}
For $d > 2 \mathrm{max}\{m_i| \ 1 \leq i \leq r\} + 1$, there is a quasi-isomorphism 
\begin{multline} \label{total_complex_quasi-iso}
C_*(\embthm) \otimes \mathbb{Q} \simeq  \\ 
 \left( \underset{k \geq 0}{\bigoplus} \mathrm{hom}_{\Sigma_k} \left( \underset{1 \leq i \leq r}{\widehat{\otimes}} \overline{H}_*(C(\bullet, \rbb^{m_i}), \qbb)(k), \widehat{H}_*(C(k, \rdbb), \qbb) \right), \partial \right).
\end{multline}
Here $\overline{H}_*(-)$ is the Borel-Moore homology functor. 
\end{prop}

The differential $\partial$ here is defined similarly to the $r=1$ case \cite[Section 5]{aro_tur13}. We won\rq{}t describe it explicitly here. However, in \cite[Subsection 4.3]{songhaf_tur16} we  explicitly describe the dual complex computing the cohomology of the space of links.


From now on let us assume that all $m_i>1$. We can do so because the summands in~\eqref{final_splitting_H} up to a regrading depend only on the parities of $m_i$, $i=1\ldots r$, and that 
of~$d$. The right hand side of (\ref{total_complex_quasi-iso}) can be rewritten as 

\[\underset{\vec{s}, t \geq 0}{\bigoplus} \left( \underset{k \geq 0}{\bigoplus} \mbox{hom}_{\Sigma_k} \left(A^k_{\vec{s}}, \widehat{H}_{t(d-1)}(C(k, \rdbb), \qbb)\right) ,\, \partial \right),\] 
where  
\[
A^k_{\vec{s}} =  \underset{k = |\vec{k}|}{\bigoplus} \left( \mbox{Ind}^{\Sigma_k}_{\Sigma_{\vec{k}}} \bigotimes_{i=1}^r \overline{H}_{s_i(m_i-1)+k_i}(C(k_i, \rbb^{m_i}), \qbb) \right).
\]

The $\vec{s}, t $ summand in the complex above computes exactly the corresponding 
summand in~\eqref{final_splitting_H}. Thus  the Euler characteristic $\mathcal{X}_{\vec{s}, t}$
used in $F^H_{\vec{m}, d}(x_1, \cdots, x_r, u) = \underset{\vec{s}, t \geq 0}{\sum} \mathcal{X}_{\vec{s}, t} \cdot u^t \vec{x}^{\vec{s}}$ can
be defined as 
\begin{equation} \label{Euler_char_defn}
\mathcal{X}_{\vec{s}, t}= \underset{k \geq 0}{\sum} (-1)^{t(d-1)-\sum_{i=1}^r s_i(m_i-1)-k} \mbox{dim}
 X^k_{\vec{s},t},
\end{equation} 
where 
\begin{equation} \label{xk_eqn}
X^k_{\vec{s},t} = \mbox{hom}_{\Sigma_k} \left(A^k_{\vec{s}}, \widehat{H}_{t(d-1)}(C(k, \rdbb), \qbb) \right).
\end{equation}  

%
%
%

\subsection{Cycle index sum and proof of  Theorem~\ref{gen_function_thm}}

Before starting the proof of Theorem~\ref{gen_function_thm} we will state two lemmas. 
For a permutation $\sigma \in \Sigma_k$, let $j_l(\sigma)$ denote the number of its cycles of length $l$. Let $\mbox{tr}(-)$ denote the trace function from the space of matrices to the ground field. 

\begin{defn} \label{cycle_index_sum_defn}
Let $\{p_1, p_2, \cdots\}$ be a family of commuting variables.
\begin{enumerate}
\item[$\bullet$] For a representation $\rho^V \colon \Sigma_k \lra GL(V)$ of the symmetric group $\Sigma_k$,  the \emph{cycle index sum} of $V$, denoted  $Z_V(p_1, p_2, \cdots)$, is defined  by 
\begin{equation} \label{z_defn}
Z_{V}(p_1, p_2, \cdots) = \frac{1}{|\Sigma_k|} \sum_{\sigma \in \Sigma_k} \mbox{tr}(\rho^V(\sigma)) \prod_l p_l^{j_l(\sigma)}.
\end{equation} 
\item[$\bullet$]  If $V = \{V(k)\}_{k \geq 0}$ is a symmetric sequence, we define $Z_V(p_1, p_2, \cdots) = \sum_{k \geq 0} Z_{V(k)} (p_1, p_2, \cdots)$. 
\end{enumerate}
\end{defn}

The vector spaces in the symmetric sequences below will be always multigraded. The trace in~\eqref{z_defn} will be a graded trace,
i.e. it will be a generating function of traces on each component. For example, if $V=\bigoplus_{i,\vec{s}}V_{i,\vec{s}},$
where $i$ is the homological degree, and $\vec{s}$ is the Hodge multigrading, then 
\[
\mbox{tr}(\rho^V(\sigma)) = \sum_{i,\vec{s}} \mbox{tr}(\rho^{V_{i,\vec{s}}}(\sigma)) z^i\vec{x}^{\vec{s}}.
\]
With such definition, the cycle index sum $Z_V$ of a symmetric sequence will also depend on $z,x_1,\ldots,x_r$, and also
possibly on $u$, where the last variable will be responsible for the grading by complexity. 

Given two families $\{p_1, p_2, \cdots\}$ and $\{b_1, b_2, \cdots\}$  of commuting variables, we will write $Z_V(p_l \leftarrow b_l; l \in \mathbb{N})$ for $Z_V(b_1, b_2, \cdots)$. This notation means substituting the value $p_l$ with the value $b_l$ in $Z_V$. For a function $f = f(p_1, p_2, \cdots)$ on variables $p_1, p_2, \cdots$, the notation  $\left. \left\{Z_V(\frac{\partial}{\partial p_l}; l \in \mathbb{N})f(p_1, p_2, \cdots) \right\} \right|_{p_l=0}$ means that we apply the differential operator $Z_V(\frac{\partial}{\partial p_l}; l \in \mathbb{N})$ to $f$, and at the end we take $p_l = 0$ for all $l \geq 0$.

 The following lemma is well known, see for example \cite[Corollary 15.5]{turchin10}.

\begin{lem} \label{dim_hom}
Let $V = \oplus_i V_i$ and $W = \oplus_j W_j$ be two graded vector spaces admitting an action of the symmetric group $\Sigma_k$, that respects the grading, on each of them. 
Consider the  series 
\[\mathrm{dim\, hom}_{\Sigma_k} (V, W) = \sum_{i,j} \mathrm{dim\, hom}_{\Sigma_k}(V_i, W_j) z^{j-i},\]
\small
\[Z_V(z; p_1, p_2, \cdots) = \sum_{i} Z_{V_i}(p_1, p_2, \cdots) z^i, \text{ and }  Z_W(z, p_1, p_2, \cdots) = \sum_{j} Z_{W_j}(p_1, p_2, \cdots) z^j.\]
\normalsize
Then 
\begin{multline}  \label{dim_hom_formula}
\mathrm{dim\, hom}_{\Sigma_k}(V, W) = \left. \left\{ Z_V(\frac{1}{z}; p_l \leftarrow \frac{\partial}{\partial p_l}, l \in \mathbb{N}) Z_{W}(z; p_l \leftarrow lp_l, l \in \mathbb{N}) \right\} \right|_{p_l =0}=\\
\left. \left\{ Z_V(\frac{1}{z}; p_l \leftarrow l\frac{\partial}{\partial p_l}, l \in \mathbb{N}) Z_{W}(z; p_1,p_2,\cdots) \right\} \right|_{p_l =0}.  
\end{multline}
\end{lem}

For the application below $V$ will be in addition Hodge multigraded and $W$ will be in addition graded by complexity. This
will add up variables in the expression. 

The  second result  we need to prove Theorem~\ref{gen_function_thm} is the following well known lemma, see for example \cite[Proposition 15.3]{turchin10} or \cite{feit67}.

\begin{lem}  \label{z_otimes}
Let $V=\{V(k)\}_{k \geq 0}$ and $W = \{W(k)\}_{k \geq 0}$ be two finite symmetric sequences (that is, for all $k \geq 0$, $V(k)$ and $W(k)$ are of finite dimensions). Then 
\[Z_{V \widehat{\otimes} W}(p_1, p_2, \cdots) = Z_V(p_1, p_2, \cdots) Z_W(p_1, p_2, \cdots).\] 
\end{lem}

Again we will be using a multigraded version of this lemma.

We are now ready to prove the main result of this section. 

\begin{proof}[Proof of Theorem~\ref{gen_function_thm}]
  Consider the following generating function 
\[
\Psi_{\vec{m}, d}(x_1, \cdots, x_r, u, z)=  \underset{\vec{s}, t,k }{\sum}  (\mbox{dim} X^k_{\vec{s},t})  \vec{x}^{\vec{s}} u^t z^{t(d-1)-\sum_{i=1}^r s_i(m_i-1)-k}, 
\]
where $X^k_{\vec{s},t}$ is the space from (\ref{xk_eqn}). Define two symmetric sequences $V$ and $W$ by 
\[ V(k) = \underset{k = |\vec{k}|}{\bigoplus} \left( \mbox{Ind}^{\Sigma_k}_{\Sigma_{\vec{k}}} \otimes_{i=1}^r \overline{H}_*(C(k_i, \rbb^{m_i}), \qbb) \right) \ \mbox{and} \ W(k) = \widehat{H}_*(C(k, \rdbb), \qbb).\]
Then, by Lemma~\ref{dim_hom}, the generating function $\Psi_{\vec{m}, d}(x_1, \cdots, x_r, u, z) =$ 
\[
\left. \left\{Z_V(\frac{1}{z}, x_1, \cdots, x_r; p_l \leftarrow \frac{\partial}{\partial p_l}, l \geq 1 ) Z_W(z, u; p_l \leftarrow lp_l, l \geq 1) \right\} \right|_{p_l=0}.
\]
Since (recalling the definition of $E_l(-)$ from the introduction, just before Theorem~\ref{gen_function_thm}) 
\begin{enumerate}
\item[$\bullet$]  \small $Z_W(z, u; p_l \leftarrow lp_l, l \geq 1) = \prod_{l=1}^{+\infty} e^{-p_l}(1+(-1)^dl((-z)^{d-1}u)^lp_l)^{(-1)^d E_l(\frac{1}{(-z)^{d-1}u})}$ \normalsize by \cite[Proposition 6.4]{aro_tur13}, 
\item[$\bullet$] \small $Z_V(\frac{1}{z}, x_1, \cdots, x_r; p_l \leftarrow \frac{\partial}{\partial p_l}, l \geq 1 ) = \prod_{i=1}^r Z_{\overline{H}_*(C(\bullet, \rbb^{m_i}), \qbb)}(\frac{1}{z}, x_i; p_l \leftarrow \frac{\partial}{\partial p_l}, l \geq 1 )$ \normalsize by Lemma~\ref{z_otimes} and the fact that $V = \widehat{\bigotimes}_{i=1}^r \overline{H}_*(C(\bullet, \rbb^{m_i}), \qbb)$, 
\end{enumerate}
and since 
\[ 
Z_{\overline{H}_*(C(\bullet, \rbb^{m_i}), \qbb)}(\frac{1}{z}, x_i; p_l \leftarrow \frac{\partial}{\partial p_l}, l \geq 1 ) = \prod_{l=1}^{+\infty} \left(1+(-\frac{1}{z})^l \frac{\partial}{\partial p_l} \right)^{(-1)^{m_i}E_l(\frac{x_i}{(-z)^{m_i-1}})} , 
\]
by \cite[Proposition  6.5]{aro_tur13},  it follows that  $\Psi_{\vec{m}, d}(x_1, \cdots, x_r, u, z)=$
 \begin{align*} 
  & \left. \left\{ \prod_{i=1}^r \prod_{l=1}^{+\infty} \left(1+(-\frac{1}{z})^l \frac{\partial}{\partial p_l} \right)^{(-1)^{m_i}E_l(\frac{x_i}{(-z)^{m_i-1}})} Y_l\right\} \right|_{p_l=0} = & & \\
  & \left. \left\{ \prod_{l=1}^{+\infty} \left(1+(-\frac{1}{z})^l \frac{\partial}{\partial p_l} \right)^{\sum_{i=1}^r(-1)^{m_i}E_l(\frac{x_i}{(-z)^{m_i-1}})} Y_l\right\} \right|_{p_l =0},
\end{align*}
where 
$
Y_l = e^{-p_l} \left(1+(-1)^dl((-z)^{d-1}u)^lp_l \right)^{(-1)^d E_l(\frac{1}{(-z)^{d-1}u})} . 
$
Notice that the $l$-th factor uses $p_l$ and not any other $p_i$, $i\neq l$, which is then taken to be zero. For this reason in the formula below we replace each $p_l$ by $a$. By looking at 
(\ref{Euler_char_defn}), and the definition of $\Psi_{\vec{m}, d}(x_1, \cdots, x_r, u, z)$,  we have the equality 
\[F^H_{\vec{m}, d}(x_1, \cdots, x_r, u) = \Psi_{\vec{m}, d}(x_1, \cdots, x_r, u, -1), \] 
and this implies
\begin{align*} 
F^H_{\vec{m}, d}(x_1, \cdots, x_r, u) = &  \left. \left\{\prod_{l=1}^{+\infty} \left(1+\frac{\partial}{\partial a} \right)^{-\sum_{i=1}^r (-1)^{m_i-1} E_l(x_i)} Z_l\right\} \right|_{a=0}\\
 = & \prod_{l = 1}^{+\infty} \frac{\Gamma\left((-1)^{d-1}E_l(\frac{1}{u}) - \sum_{i=1}^r (-1)^{m_i-1}E_l(x_i)\right)}{\left((-1)^{d-1}lu^l\right)^{\sum_{i=1}^r (-1)^{m_i-1}E_l(x_i)} \Gamma\left((-1)^{d-1} E_l(\frac{1}{u})\right)},
\end{align*}
where 
$
Z_l = e^{-a}\left(1+(-1)^dlu^la \right)^{(-1)^dE_l(\frac{1}{u})}. 
$
The last equality follows from the identity
\[
\left.\left(1+\frac{\partial}{\partial a} \right)^{-X} e^{-a}\left(1+(-1)^dlu^la \right)^{(-1)^dE_l(\frac{1}{u})}\right|_{a=0}=
\frac{\Gamma\left((-1)^{d-1}E_l(\frac{1}{u}) - X\right)}{\left((-1)^{d-1}lu^l\right)^X \Gamma\left((-1)^{d-1} E_l(\frac{1}{u})\right)},
\]
whose proof is identical to that  of \cite[Proposition 15.7]{turchin10}.

We thus obtain the desired result. 
\end{proof}

\subsection{Understanding generating function of the Euler characteristics of the homology splitting}\label{ss:understanding}

The formula~\eqref{gen_function_formula} might appear confusing at first as it is defined in terms of assymptotic expansions that still need to be deciphered. 
Let us first fix some notation. For two variables $x, u$,  let $\Gamma(x,u)$ denote the asymptotic expansion of $\frac{\Gamma(\frac{1}{u}-x)}{u^x \Gamma(\frac{1}{u})}$ when $u$ goes to $+0$. For instance, in the case $x=n$ is a positive integer, one has 
\begin{equation} \label{gamma_nu}
\Gamma(n, u) = \frac{1}{(1-u)(1-2u) \cdots (1-nu)}.
\end{equation}
In the case $x$ is a negative integer: $x=-n$, one has
\begin{equation}\label{gamma_nu_neg}
\Gamma(-n,u) = (1+u)(1+2u)\cdots (1+(n-1)u).
\end{equation}
(In particular $\Gamma(-1,u)=\Gamma(0,u)=1$.) 
The series $\Gamma(x,u)$ can be seen as a generating function in $u$ of a sequence of polynomials in $x$, 
see \cite[Subsection 14.1]{turchin10}. It has rather bad convergency properties: for any $x\in \cbb \setminus \zbb$ its radius of convergence in $u$ is zero~\cite[Proposition~14.5]{turchin10}. 
%
%
%

%


\begin{prop}\label{pr:gen_function_homol}
\begin{equation}\label{eq:gen_func_homol}
F^H_{\vec{m}, d}(x_1, \cdots, x_r, u)  = \frac{\prod_{l \geq 1} \Gamma \left(\sum_{i=1}^r (-1)^{m_i-1} E_l(x_i), (-1)^{d-1} \frac{lu^l}{F_l(u)} \right)}{\prod_{l \geq 1} \left(F_l(u) \right)^{\sum_{i=1}^r (-1)^{m_i-1} E_l(x_i)}}, 
\end{equation}
where polynomials $F_l(u)$ are defined in (\ref{eq:F_l}).
\end{prop}
\begin{proof}
This  is  deduced from Theorem~\ref{gen_function_thm} and from the formula
\[
\frac{\Gamma( (-1)^{d-1} E_l(\frac 1u)-X)}{((-1)^{d-1}lu^l)^X \Gamma((-1)^{d-1} E_l(\frac 1u))}=
\frac{\Gamma\left(X,(-1)^{d-1}\frac{lu^l}{F_l(u)}\right)}{F_l(u)^X},
 \]
   see \cite[Lemma 14.6]{turchin10} for similar computations for the case of long knots.  
\end{proof}
%

\subsection{Some special cases of the generating function in homology} \label{some_special_cases_gen_function_homology}

The generating function $F^H_{\vec{m}, d}(x_1, \cdots, x_r, u)$ from (\ref{gen_function_formula}) can be used
 to estimate the ranks of the homology and homotopy groups of $\emb$. However, the formula itself is hard to
 use to estimate the asymptotics of the growth since it has infinitely many rather complicated factors. However, sometimes 
 it is possible to choose  specific values for $x_i$ that would kill all but finitely many factors. This has been done in the case 
 of long knots $r=1$, $m_1=1$ in~\cite{turchin10}, where choosing $x_1=-1$ allows one to prove the exponential growth
 of the ranks of the rational  homology and homotopy groups of the space of long knots. Notice that choosing $x_1=1$ (all $x_i=1$ 
 in the general case) corresponds to forgetting the Hodge splitting. In the case of long knots the latter choice does not prove exponential growth, see~\cite{kom_lambr,turchin10} or computations below, which showed that the Hodge splitting was essentital.

In this subsection we produce similar computations for the spaces of classical string links: all $m_i=1$, $i \in \{1, \cdots, r\}$.
Recall that it has been shown that the homology ranks of $\overline{\mbox{Emb}}_c(\coprod_{i=1}^r \mathbb{R}^1, \rdbb)$ grow exponentially for $r\geq 2$ without using the Hodge splitting (this being obtained as a combination of the main results of \cite{kom_lambr} and \cite{songhaf13}).\footnote{In general the homology ranks of $\emb$ grow exponentially for $r \geq 2$.  One can see that both from Theorem~\ref{genius_zero_thm} and   Theorem~\ref{genius_one_thm}. For $r=1$ and $m_1$ odd one still get the exponential growth \cite{aro_tur12}. But for $r=1$ and $m$ even the question  about  the exponential growth  is open.} The computations that we produce below are rather disappointing: we compare the choice of all $x_i=-1$ versus all $x_i=1$ and see that only for $r\leq 2$ 
the first choice may give a better estimate.\footnote{As we said earlier it does give a better estimate for $r=1$. In case $r=2$ and $d$ odd it also gives a better estimate, but the case of
$r=2$ and $d$ even is more subtle as both generating functions happen to have the same radius of convergence.}
%

%
%
%

\begin{enumerate}

\item[$\bullet$] Assume $d$ odd, and consider the function $F_{\mbox{odd}}(1, u)$ defined by $F_{\mbox{odd}}(1, u) = F^H_{\vec{1}, d}(1, \cdots, 1, u)$.  We want to compute $F_{\mbox{odd}}(1, u)$. By substituting $m_i$ by $1$, and $x_i$ by $1$, $1 \leq i \leq r$, in the formula (\ref{eq:gen_func_homol}), and by noticing that $E_l(1) = \left\{ \begin{array}{ccc} 
                                                                                      1 & \mbox{if} & l =1 \\
																																											0 & \mbox{if} & l \geq 2
                                                                                    \end{array} \right.$, we obtain
\[ 
F_{\mbox{odd}}(1,u) = \Gamma(r,u) =  \frac{1}{(1-u)(1-2u) \cdots (1-ru)},
\]
where the last equality  comes from~\eqref{gamma_nu}. 
 The radius of convergence in this case is $\frac 1r$ which implies exponential growth of rational homology ranks in case $r\geq 2$.

\item[$\bullet$] For $d$ even, similar computations give 
\begin{equation}
F_{\mbox{even}}(1, u) = F^H_{\vec{1},d}(1, \cdots, 1, u) = \frac{1}{(1+u)(1+2u) \cdots (1+ru)}. 
\end{equation}
The radius of convergence is $\frac 1r$.

\item[$\bullet$] Assume $d$ odd, and consider the function $F_{\mbox{odd}}(-1, u)$ defined by $F_{\mbox{odd}}(-1,u) = F^H_{\vec{1}, d}(-1, \cdots, -1,u)$.  By substituting $m_i$ by $1$, and $x_i$ by $-1$, $1 \leq i \leq r$, in the formula (\ref{eq:gen_func_homol}), and by noticing that $E_l(-1) = \left\{ \begin{array}{ccc} 
                                                                                      (-1)^l & \mbox{if} & l \in \{1,2\} \\
																																											0 & \mbox{if} & l \geq 3                         \end{array} \right.$, 
																																											we obtain
\begin{equation} \label{fodd_-1u-0}
 F_{\mbox{odd}}(-1,u) = \Gamma(-r,u)\times \frac{\Gamma\left(r,\frac{2u^2}{F_2(u)}\right)}{F_2(u)^2}.
\end{equation}
Since $F_2(u) = 1-u$ and applying~\eqref{gamma_nu} and~\eqref{gamma_nu_neg}, the formula (\ref{fodd_-1u-0}) produces 
\begin{equation} \label{fodd_-1u}
F_{\mbox{odd}}(-1, u) = \frac{(1+u)(1+2u) \cdots (1+(r-1)u)}{(1-u-2u^2)(1-u-4u^2) \cdots (1-u-2ru^2)}
\end{equation}

It is easy to see that the smallest root (in absolute value) of the denominator of (\ref{fodd_-1u}) is $\frac{-1+\sqrt{1+8r}}{4r}$. Therefore the radius of convergence is  equal to $\frac{-1+\sqrt{1+8r}}{4r}$.  This gives a beter estimate of the homology ranks growth for $r\leq 2$. In case $r\geq 4$, the produced
estimate is weaker than the one obtained from $F_{\mbox{odd}}(1,u)$.

\item[$\bullet$] For $d$ even, similar computations  give 
\begin{equation} \label{feven_-1u}
F_{\mbox{even}}(-1, u) = \frac{(1-u)(1-2u) \cdots (1-(r-1)u)}{(1-u+2u^2)(1-u+4u^2) \cdots (1-u+2ru^2)}
\end{equation}
As before, the radius of convergence of (\ref{feven_-1u})  is  equal to $\left|\frac{1+\sqrt{1-8r}}{4r} \right|=\frac 1{\sqrt{2r}}$.  In particular for $r=1$ this implies exponential growth of the homology ranks. In case $r\geq 3$, the produced estimate is weaker than the one obtained from $F_{\mbox{even}}(1, u)$.

\end{enumerate}

\section{Generating functions of Euler characteristics in homotopy} \label{generating_function_homotopy}

In this section we prove Theorem~\ref{gen_function_htpy_thm}, which computes the generating function of Euler characteristics  of  summands in the homotopical splitting  (\ref{final_splitting_pi}). 
We also compute, using a different complex, the generating function of the ranks of
the summands in~\eqref{final_splitting_pi} of genus zero and one.\footnote{Genus $g$ is defined as the complexity minus total Hodge degree plus one $g=t{-}(s_1{+}\ldots{+}s_r){+}1$.}

\subsection{The generating function for $\qbb \otimes \pi_*(\emb)$} \label{generating_function_homotopy_subsection}

In Section~\ref{generating_function_homology} we have computed the generating function $F^H_{\vec{m}, u}(x_1, \cdots, x_r, u)$ of Euler characteristics of summands in the homological splitting of $H_*(\emb, \qbb)$. 
The aim of this section is to compute  the generating function $F^\pi_{\vec{m}, u}(x_1, \cdots, x_r, u)$
of the Euler characteristics of summands, but now in the homotopy $\qbb \otimes \pi_* (\emb)$. Proposition~\ref{pr:gen_function_homol} will be a key ingredient in our computations. 
  If there is no confusion, the first generating function will be denoted  by $F^{H}(x_1, \cdots, x_r, u)$,
 and the second one just by $F^{\pi}(x_1, \cdots, x_r, u)$.  Recalling the equation (\ref{final_splitting_pi}) from the introduction, for a sequence $\vec{s}, t \geq 0$, let $\mathcal{X}^{\pi}_{\vec{s}, t}$ denote the Euler characteristic of the summand $\rmod (Q^{\vec{m}}_{\vec{s}}, \qbb \otimes \widehat{\pi}_{t(d-2)+1} (C(\bullet, \rdbb)))$. We will also write $\xvecs$ for $\prod_{i=1}^r x_i^{s_i}$.  The generating function we look at here is defined by
\begin{equation}  \label{gen_function_htpy_defn}
F^{\pi}_{\vec{m}, d}(x_1, \cdots, x_r, u) = \sum_{\vec{s}, t} \mathcal{X}^{\pi}_{\vec{s}, t} \xvecs u^t.
\end{equation}

 In order to compute $F^{\pi}_{\vec{m}, d}(x_1, \cdots, x_r, u)$,  we need the following lemma, which connects $F^H(x_1, \cdots, x_r, u)$ and $F^{\pi}(x_1, \cdots, x_r, u)$. 

\begin{lem} \label{gen_function_htpy-lemma}
We have 
\begin{equation} \label{gen_function_htpy-formula}
F^{\pi}_{\vec{m}, d}(x_1, \cdots, x_r, u) = \sum_{l =1}^{+\infty} \frac{\mu(l)}{l} \ln (F^H (x_i \leftarrow x_i^l, u \leftarrow u^l)).
\end{equation}
\end{lem}
Here $\mu(-)$ is the standard M\"obius function. The notation  $F^H (x_i \leftarrow x_i^l, u \leftarrow u^l)$ means that in the expression $F^H (x_1, \cdots, x_r, u)$ the variable $u$ is replaced by $u^l$, and $x_i$ is replaced by $x_i^l, 1 \leq i \leq r$.

\begin{proof}
The plan of the proof is to compute the right hand side of (\ref{gen_function_htpy-formula}), and compare the obtained result with the right hand side of (\ref{gen_function_htpy_defn}). From \cite[Lemma 16.1]{turchin10}, it is straightforward to get the following
\begin{equation} \label{gen_function_htpy-lemma1}
F^{H}(x_1, \cdots, x_r, u) = \prod_{\vec{s}, t} \frac{1}{(1-\xvecs u^t)^{\mathcal{X}^{\pi}_{\vec{s}, t}}},
\end{equation}
which gives (by replacing $u$ by $u^l$ and $x_i$ by $x_i^l$)
\begin{equation} \label{gen_function_htpy-lemma2}
F^{H}(x_i \leftarrow x_i^l, u \leftarrow u^l) = \prod_{\vec{s}, t} \frac{1}{\left(1-(\xvecs u^t)^l \right)^{\mathcal{X}^{\pi}_{\vec{s}, t}}}.
\end{equation}
Taking now the logarithm of (\ref{gen_function_htpy-lemma2}), and using the well known series expansion $\ln (1-x) = -\sum_{p \geq 1} \frac{x^p}{p}$, we obtain 
\begin{equation} \label{gen_function_htpy-lemma3}
\ln (F^H(x_i \leftarrow x_i^l, u \leftarrow u^l)) = \sum_{\vec{s}, t} \sum_{p \geq 1} \mathcal{X}^{\pi}_{\vec{s}, t} \frac{(\xvecs u^t)^{pl}}{p}.
\end{equation}
By putting (\ref{gen_function_htpy-lemma3}) in the right hand side of (\ref{gen_function_htpy-formula}), we get
\begin{align*}
\sum_{l =1}^{+\infty} \frac{\mu(l)}{l} \ln (F^H (x_i \leftarrow x_i^l, u \leftarrow u^l)) = &  \sum_{l \geq 1} \sum_{\vec{s}, t} \sum_{p \geq 1} \mathcal{X}^{\pi}_{\vec{s}, t} \mu(l) \frac{(\xvecs u^t)^{pl}}{pl} \\
 =& \sum_{\vec{s}, t} \sum_{q \geq 1} \mathcal{X}^{\pi}_{\vec{s}, t} \frac{1}{q} (\xvecs u^t)^q \left(\sum_{l|q} \mu (l) \right).
\end{align*}
Since $\sum_{l|q} \mu (l) = \left\{ \begin{array}{lll}
                                    1 & \mbox{if} & q=1 \\
																		0 & \mbox{if} & q \geq 2
                                    \end{array} \right.$, it follows that 
\begin{align*}																		
\sum_{l =1}^{+\infty} \frac{\mu(l)}{l} \ln (F^H (x_i \leftarrow x_i^l, u \leftarrow u^l)) = &   \sum_{\vec{s}, t}  \mathcal{X}^{\pi}_{\vec{s}, t} \xvecs u^t  = &  F^{\pi} (x_1, \cdots, x_r, u) \ \mbox{ by (\ref{gen_function_htpy_defn})}. 
\end{align*}																	
\end{proof}
The last thing we need is an explicit expansion of $\ln(\Gamma(x,u))$, where $\Gamma(x,u)$ is defined at the beginning 
of Subsection~\ref{ss:understanding}.  
%
%
%


\begin{lem}\label{l:ln_gamma}
$\ln (\Gamma (x, u)) = \sum_{j \geq 1} S_j(x) \frac{u^j}{j}.$
\end{lem}
\begin{proof}
This identity easily follows from  (\ref{gamma_nu}) and from Bernoulli's summation formula $S_j(n) = 1^j + \cdots + n^j$  when $x=n$ is a positive integer. For other values of $x$, it follows from the fact that $\ln(\Gamma(x, u))$ is a generating function of $u$ of a sequence of polynomials in $x$  \cite[Lemma 14.2]{turchin10}. 
\end{proof}

Now we are ready to prove Theorem~\ref{gen_function_htpy_thm}, which is the main result of this subsection. 

\begin{proof}[Proof of Theorem~\ref{gen_function_htpy_thm}]
The proof is a direct application of Proposition~\ref{pr:gen_function_homol} and Lemmas~\ref{gen_function_htpy-lemma}-\ref{l:ln_gamma}.
\end{proof}



\subsection{The generating function in homotopy for genus $\leq 1$} \label{gen_function_htpy_genus_less1}

In \cite[Subsection 2.1]{songhaf_tur16} we explicitly described a complex $\empi$ of hairy graphs that computes the rational homotopy of the space $\emb$ of high dimensional string links. 
In short, this complex  is obtained in~\cite{songhaf_tur16} by taking an injective resolution of the
target $\Omega$-modules in~\eqref{final_splitting_pi}.  
The graph-complex $\empi$ is a differential graded vector space spanned by so called {\it hairy graphs}, i.e. connected graphs having a finite non-empty set of external vertices (of valence $1$ and called {\it hairs}), a finite set of non-labeled internal vertices (of valence $\geq 3$). The edges are oriented. One allows both multiple edges and
tadpoles.  The external vertices are colored with $\{1, \cdots, r\}$ as the set of colours. As a part of the data,
each graph comes with the ordering of its {\it orientation set} that consists of the following elements:
edges (of degree $d-1$), internal vertices (of degree $-d$), external vertices (of degree $-m_i$ if it is colored
by~$i$). Changing orientation of an edge in a graph produces sign $(-1)^d$. 
Changing the order of the orientation set produces the Koszul sign of permutation that takes into account
the degree of the elements. The differential is defined as the sum of expansion of  internal vertices.  

%

For a hairy graph $G\in\empi$, its homological degree is the sum of the degrees of the elements in its orientation 
set. Its corresponding \emph{Hodge multi-degree}  is the tuple $(\vec{s})$, where $s_i$ is the number of external vertices colored by $i$. Its  \emph{complexity}  is the first Betti number of  the graph obtained from $G$ by gluing together all its external vertices.  For hairy graphs it is also natural to define
 \emph{genus} $g$ as their first Betti number.

%
As an example of the complexity, the graph in Figure~\ref{graph_genus0} is of complexity $2$, while the complexity is three in Figure~\ref{graph_genus1}, and five in Figure~\ref{graph_genus2}. Concerning the genus, intuitively, a graph is of genus $g$ if it contains $g$ \lq\lq loops\rq\rq{}. See Figure~\ref{graph_genus0}, Figure~\ref{graph_genus1} and Figure~\ref{graph_genus2} for some examples of graphs of genus $0$, $1$ and $2$ respectively. 


\begin{figure}[!ht]
\centering
\begin{minipage}[t]{3cm}
\includegraphics[scale=0.5]{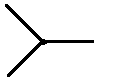}
\caption{Graph of genus $0$} \label{graph_genus0}
\end{minipage}
\begin{minipage}[t]{1cm}
\includegraphics[scale=0.1]{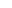}
\end{minipage}
\begin{minipage}[t]{3cm}
\includegraphics[scale=0.5]{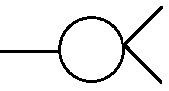}
\caption{Graph of genus $1$} \label{graph_genus1}
\end{minipage}
\begin{minipage}[t]{1cm}
\includegraphics[scale=0.1]{blanc}
\end{minipage}
\begin{minipage}[t]{3cm}
\includegraphics[scale=0.5]{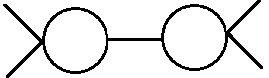}
\caption{Graph of genus $2$} \label{graph_genus2}
\end{minipage}

\end{figure}

The graph-complex $\empi$ can be split  as follows. For $g \geq 0$, let $\mathcal{E}_{\pi g}^{\vec{m}, d}$ denote the subcomplex of $\empi$ generated by graphs of genus $g$. For integers $\vec{s}$, $t$, let $\mathcal{E}_{\pi g, \vec{s}, t}^{\vec{m}, d}$ be the subcomplex of $\mathcal{E}_{\pi g}^{\vec{m}, d}$ generated by graphs of Hodge multi-degree $(\vec{s})$ and  complexity $t$. All these subcomplexes are well defined since the differential in $\empi$ preserves the Hodge multi-degree, the complexity, and the genus. In particular, one has the splitting
\begin{equation} \label{splitting_epi}
\empi=\bigoplus_{g\geq 0}\mathcal{E}_{\pi g}^{\vec{m}, d}.
\end{equation} 
\begin{defn} Let $g \geq 0$.
\begin{enumerate} 
\item[$\bullet$] The \emph{Hodge splitting} of the complex $\mathcal{E}_{\pi g}^{\vec{m}, d}$ is the splitting
\begin{equation} \label{hodge_splitting}
\mathcal{E}_{\pi g}^{\vec{m}, d} = \bigoplus_{\vec{s}, t \geq 0} \mathcal{E}_{\pi g, \vec{s}, t}^{\vec{m}, d}.
\end{equation}
Each term of the right-hand side of (\ref{hodge_splitting}) is a finite dimensional graph complex. 
\item[$\bullet$] The \emph{generating function} associated to (\ref{hodge_splitting}), and denoted $F^{\pi g}_{\vec{m}, d}(x_1, \cdots, x_r, u)$, is defined as 
\begin{equation} \label{gen_function_genus}
F^{\pi g}_{\vec{m}, d}(x_1, \cdots, x_r, u) = \underset{\vec{s}, t \geq 0}{\sum} \mathcal{X}^\pi_{g, \vec{s}, t} \xvecs u^t,
\end{equation}
where $\mathcal{X}^\pi_{g,\vec{s},t}$ is the Euler characteristic of $\mathcal{E}_{\pi g,\vec{s}, t}^{\vec{m}, d}$. 
\end{enumerate}
\end{defn}

It is easy to see the following relation between the genus and the complexity of a graph. 
\begin{equation}\label{eq:compl_gen}
\text{genus $=$ complexity $-$ number of external vertices $+$ 1}.
\end{equation}
Thanks to (\ref{eq:compl_gen}), we can easily express the generating function of the Euler characteristics that takes into account the grading $g$ as well:
\begin{equation}\label{eq:F_pi_genus}
F^\pi(x_1,\ldots,x_r,u,\hbar)=\hbar F^\pi(\frac{x_1}\hbar,\ldots,\frac{x_r}\hbar,u\hbar),
\end{equation}
where $\hbar$ is the variable responsible for the genus, and  the right-hand side  uses the generating function defined earlier~\eqref{gen_function_htpy_defn}. (Abusing notation we denote this function by $F^\pi$ as well). Note that
$
F^{\pi} = \sum_{g \geq 0} F^{\pi g} \hbar^g. 
$
Applying the result from 
Theorem~\ref{gen_function_htpy_thm} one can get an explicit formula for $F^\pi(x_1,\ldots,x_r,u,\hbar)$. However, even though this formula is very simple for explicit computer calculations it\rq{}s not at all obvious that it produces zero for the negative genus. Indeed, there will be summands in which $\hbar$ appears with a negative exponent, that must somehow cancel out with each other.

The aim of this subsection is to compute the generating function from (\ref{gen_function_genus}) with $g=0$ and~$1$ using hairy graph-complexes, which is done by Theorems~\ref{genius_zero_thm} and~\ref{genius_one_thm} below. We need first to define a symmetric sequence of graph-complexes $\{M(P_d^k)\}_{k\geq 1}$ that will be used in our computations.

\subsubsection{Graph-complexes $M(P_d^k)$} \label{graph_complex_mpdk}

Recall that $d$ denotes the dimension of the ambient space. For $k\geq 0$, $M(P_d^k)$ is the complex
of graphs defined essentially in the same way as the hairy graph-complex with the only difference that
its graphs have exactly $k$ external vertices labeled bijectively by $1,\ldots,k$. Also we exclude external vertices from the orientation set of such graphs. Thus the orientation set of any graph $G\in M(P_d^k)$ is the union
of the set $E_G$ of edges of $G$ (each edge being of degree $d-1$) and the set $I_G$ of  its internal vertices
(each internal vertex being of degree $-d$). So that the total degree of $G$ is $(d-1)|E_G|-d|I_G|$. 
The differential on $M(P_D^k)$ is defined in the same way as the sum of expansions of internal vertices. Since the differential  preserves the genus (first Betti number of a graph), each of $M(P_d^k)$ can be split into a direct sum by genus $g$. That is, one has $M(P_d^k)=\bigoplus_{g\geq 0} M_g(P_d^k).$  Let $V$ be the multi-graded $r$-dimensional vector space of colors\footnote{A \textit{color} for us is a component of links in $\emb$}, whose basis is the set   $\{v_1, \ldots, v_r\}$,  with each $v_i$ being of degree $m_i$ and of Hodge multi-degree $(\underbrace{0\ldots 0}_{i-1},1,\underbrace{0\ldots 0}_{r-i})$.
 Consider  the symmetric sequence $V^{\otimes \bullet} = \{V^{\otimes n}\}_{n \geq 0}$. One has an isomorphism
\begin{equation} \label{eq_empi_as_sigma_map-2}
\empi \cong \mbox{hom}_{\Sigma}(V^{\otimes \bullet}, M(P_d^{\bullet})). 
\end{equation}

This implies 
\begin{equation}\label{eq_empi_as_sigma_map_g}
\mathcal{E}_{\pi g}^{\vec{m}, d}\cong \mbox{hom}_\Sigma(V^{\otimes \bullet}, M_g(P_d^\bullet)).
\end{equation}

Therefore the homology of each summand $\mathcal{E}_{\pi g}^{\vec{m}, d}$ in (\ref{splitting_epi}) is completely determined by the symmetric sequence of the homology of $M_g(P_d^\bullet)$. It turns out that for $g=0$ and~$1$, the corresponding symmetric sequences can be easily described. Thus the plan to get Theorem~\ref{genius_zero_thm} and Theorem~\ref{genius_one_thm} will be to compute the cycle index sums    
 $Z_{H_*M_0(P_d^\bullet)}$, $Z_{H_*M_1(P_d^\bullet)}$ first, and then apply~\eqref{observation_proof_genius0}. The latter equation is a general formula that we obtain in the next section by first computing the cycle index sum of $V^{\otimes\bullet}$~\eqref{zotimes_formula} and then applying~\eqref{dim_hom_formula}. We mention also at this point that for every given $k$ both groups $H_*M_0(P_d^k)$ and
$H_*M_1(P_d^k)$ are concentrated in a single homological degree. Thus computation of the homology ranks or of Euler characteristics carry essentially the same information.

\subsubsection{Genus zero} \label{graph_complexes_genius_zero}

By~\eqref{eq_empi_as_sigma_map_g} the complex $\mathcal{E}_{\pi 0}$ of colored hairy graphs of genus zero has the form (take $g=0$)
\begin{equation}\label{eq_empi_as_sigma_map_0}
\mathcal{E}_{\pi 0}\cong \mbox{hom}_\Sigma(V^{\otimes \bullet}, M_0(P_d^\bullet)). 
\end{equation}
On the other hand the  complexes $M_0(P_d^\bullet)=\{M_0(P_d^k)\}_{k\geq 1}$ are complexes of trees, whose homology is well known
to be isomorphic up to a shift of degrees to the components ${\mathcal L}ie((k)):={\mathcal L}ie(k-1)$ of the cyclic Lie operad~\cite{white}. More precisely one has an isomorphism of $\Sigma_k$-modules
\begin{equation}\label{eq:H_tree=Lie}
H_*(M_0(P_d^k))\cong_{\Sigma_k}\Sigma^{k(d-2)-d+3}{{\mathcal L}ie}((k))\otimes (sign)^{\otimes d}.
\end{equation}
The homology is concentrated in the smallest possible degree of 
\textit{uni-trivalent trees}.\footnote{A tree is called \textit{uni-trivalent} if all its internal vertices are trivalent} Such a tree with $k$ external vertices must have $k-2$ internal vertices and $2k-3$ edges. Thus its degree is $(2k-3)(d-1)-(k-2)d=k(d-2)-d+3$. Our  hairy graph-complex specialized to the gradings when $d=3$ and all $m_i=1$, contains in the bottom degree homology the space of unitrivalent graphs modulo $AS$ and $IHX$ relations, which encodes the finite type invariants of string links in $\mathbb{R}^3$~\cite{barnatan95}. The tree-part of this space is well studied~\cite{habeg_masb00,
levine02,habeg_pitsch03,moskovich06}. We failed to find in the literature formulas similar  to those given by Theorems~\ref{genius_zero_thm_h}-\ref{genius_zero_thm}, but they could be easily derived from the results of aforementioned papers and are known to specialists.

 Let $\mathbf{1}(\bullet)$ be the symmetric sequence defined by $\mathbf{1}(1) = \mathbb{Q}$, and $\mathbf{1}(n) =0$ for all $n \neq 1$. One has
\[{{\mathcal L}ie}((n)) \cong_{\Sigma_n} \mbox{Ind}^{\Sigma_n}_{\Sigma_1 \times \Sigma_{n-1}} (\mathbf{1}(1) \otimes {{\mathcal L}ie}(n-1)) 
- {{\mathcal L}ie} (n), \, \, n\geq 2.\]
Rationally this was proved in~\cite{levine02,habeg_pitsch03} and integrally in~\cite{con_schn_teich12}. Therefore $\lie \cong_{\Sigma} \mathbf{1} (\bullet) \widehat{\otimes} {{\mathcal L}ie}(\bullet) - {{\mathcal L}ie}(\bullet)+\mathbf{1} (\bullet)$.
 We add $\mathbf{1} (\bullet)$ to compensate subtraction of ${{\mathcal L}ie}(1)$, i.e. to have zero in arity one. This implies that (since  $Z_{\mathbf{1}(\bullet)} (p_1, p_2, \cdots) = p_1$) 
\[Z_{\lie}  = (p_1-1) Z_{{{\mathcal L}ie}(\bullet)}+p_1. \]
 Using now the following well known result (see for example~\cite{brandt44}; another short and elegant proof is given in~\cite[Section 5.1]{dots_khoro06})
\[Z_{{{\mathcal L}ie}(\bullet)}(p_1, p_2, \cdots) =  \sum_{l=1}^{+\infty} \frac{-\mu (l) \ln (1-p_l)}{l},\]
we get 
\begin{equation} \label{zlie_formula}
Z_{\lie} (p_1, p_2, \cdots) = (1-p_1) \sum_{l=1}^{+\infty} \frac{ \mu (l) \ln (1-p_l)}{l}+p_1.
\end{equation}

From \eqref{eq:H_tree=Lie} and the fact that for genus zero graphs the complexity is the number of external vertices minus one, we get
\begin{equation}\label{eq:Z_M_0}
Z_{H_*M_0(P_d^\bullet)}(u,z;p_1,p_2,\ldots)=\frac{1}{z^{d-3}u}Z_{\lie}\left(p_l\leftarrow (-1)^{(l-1)d}(z^{d-2}u)^lp_l,\, l\in\nbb\right).
\end{equation}

To recall, the \textit{Hodge splitting} is defined in (\ref{hodge_splitting}).  
\begin{thm} \label{genius_zero_thm_h}
The generating function of the dimensions of the Hodge summands of the complex 
$\mathcal{E}_{\pi 0}^{\vec{m},d}$ of hairy graphs of genus zero is
\begin{align*}
 & R^{\pi 0}_{\vec{m}, d}(x_1, \cdots, x_r, z,u)  \\ 
& = z\alpha_1(\frac 1z) +\frac{ 1-z^{d-2}u\alpha_1(\frac 1z)}{z^{d-3}u}\sum_{l=1}^{+\infty} \frac{\mu(l)\ln\left(
1-(-1)^{(l-1)d}(z^{d-2}u)^l\alpha_l(\frac 1z)\right)}{l},  
\end{align*}
where $\alpha_l(\frac 1z) = \sum_{i=1}^r (-1)^{m_i(l-1)} x_i^l (\frac 1z)^{m_il}$.
\end{thm}

  \begin{proof}
  This formula is obtained from~\eqref{zlie_formula} by change of variables: first using~\eqref{eq:Z_M_0} and then~\eqref{observation_proof_genius0}.
  \end{proof}

  \begin{thm} \label{genius_zero_thm}
The generating function of the Euler characteristics of the Hodge summands of the complex 
$\mathcal{E}_{\pi 0}^{\vec{m},d}$ of hairy graphs of genus zero is
\begin{multline*}
  F^{\pi 0}_{\vec{m},d}(x_1, \cdots, x_r, u)   =   \\ 
    -\sum_{i=1}^r(-1)^{m_i}x_i +\left(\sum_{i=1}^r(-1)^{m_i}x_i - \frac {(-1)^d}{u}\right)\sum_{l=1}^{+\infty} \frac{\mu(l)\ln\left(
1-(-1)^{d}u^l\sum_{i=1}^r(-1)^{m_i}x_i^l\right)}{l}.   
\end{multline*}
\end{thm}

\begin{proof}
Take $z=-1$ in the previous theorem.
\end{proof}

\subsubsection{Genus one}

In the previous section we have computed the generating function for graph-complexes of genus zero. Here we will make computations in the genus one case. 
%
Figure~\ref{graph1_genius1} and Figure~\ref{graph2_genius1} are examples of graphs of genus one.

\begin{figure}[!ht]
\centering
\begin{minipage}[t]{3cm}
\includegraphics[scale=0.5]{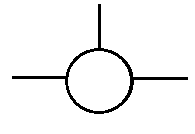}
\caption{A graph of genus one} \label{graph1_genius1}
\end{minipage}
\begin{minipage}[t]{1cm}
\includegraphics[scale=0.3]{blanc}
\end{minipage}
\begin{minipage}[t]{3cm}
\includegraphics[scale=0.5]{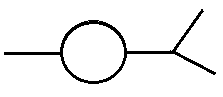}
\caption{Another graph of genus one} \label{graph2_genius1}
\end{minipage}
\begin{minipage}[t]{1cm}
\includegraphics[scale=0.3]{blanc}
\end{minipage}
\begin{minipage}[t]{3cm}
\includegraphics[scale = 0.5]{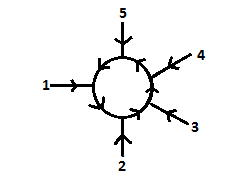}
\caption{A graph in $H_*M_1(P_d^5)$} \label{graph3_genius1}
\end{minipage}
\end{figure}

%

The first goal will be to understand the $\Sigma_n$ acton on  $H_*M_1(P_d^n)$. 
Notice that in this homology we can consider only hedgehogs: graphs whose external vertices are directly connected to the loop. 
This is because the other graphs are killed by the differential in homology (see \cite{con_cos_tur_weed13}). 
A typical graph (or generator) in $H_*M_1(P_d^n)$ is  the one of Figure~\ref{graph3_genius1}.

%

A hedgehog with $n$ external vertices has $2n$ edges and $n$ internal vertices, thus the homology $H_*M_1(P_d^n)$ is concentrated 
in the only degree $2n(d-1)-nd=n(d-2)$. Let $D_n$ denote the dihedral group of symmetries of the unit circle with $n$ points marked 
$e^{\frac{2k\pi i}n}$, $k=0\ldots n-1$. One has an obvious homomorphism $D_n\to \Sigma_n$ corresponding to the permutation of the marked points. This map is  an inclusion for $n\geq 3$. The hedgehog  whose external vertices are marked in the cyclic order $1,\ldots,n$ is sent to itself times certain sign when acted on by  elements of $D_n$. Denote by $\lambda_n$ the character of $D_n$ corresponding to this sign. (This character depends on $d$ as well.) As a $\Sigma_n$ module in graded vector spaces 
$H_*M_1(P_d^n)$ is the $n(d-2)$-suspended induced representation
\begin{equation}\label{eq_hom_gen1_sigma}
H_*M_1(P_d^n)\cong_{\Sigma_n} \Sigma^{n(d-2)}\mathrm{Ind}^{\Sigma_n}_{D_n}\lambda_n.
\end{equation}
Let us describe the character $\lambda_n$. Since $D_n$ is generated by two elements: a $\frac{2\pi}n$ rotation and a reflection, it is enough to compute $\lambda_n$ on those elements. For an element $\sigma \in D_n$, we will explicitly determine the sign 
$\lambda_n(\sigma)$ (using the Koszul sign of permutation taking into account the degrees of the elements in the orientation set) that appears in the equality $\sigma \cdot G = \lambda_n(\sigma)  G$. Here $G$ stands for the  hedgehog  whose external vertices are marked in the cyclic order $1,\ldots,n$ (Figure~\ref{graph3_genius1} is an example of a hedgehog with $5$ external vertices.) To recall the orientation set is the union of edges, of degree $d-1$, and internal vertices, of degree $-d$. 
  There are three possibilities:

\begin{figure}[!ht]
\centering
\begin{minipage}[t]{3cm}
\includegraphics[scale=0.5]{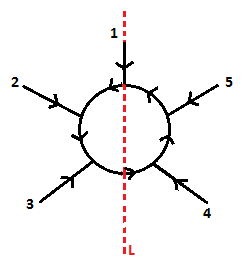}
\caption{A reflection with respect to $L$ (n=5)} \label{reflection_nodd}
\end{minipage}
\begin{minipage}[t]{3cm}
\includegraphics[scale=0.2]{blanc}
\end{minipage}
\begin{minipage}[t]{3cm}
\includegraphics[scale=0.5]{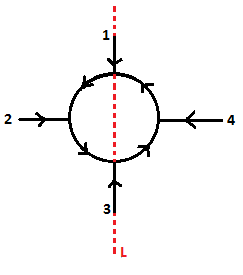}
\caption{A reflection with respect to $L$ (n=4)} \label{reflection_neven1}
\end{minipage}
\end{figure}
%

\begin{enumerate}
\item[-] If $\sigma$ is a rotation by $\frac{2\pi}{n}$, then one can easily see that $\sigma \cdot G = (-1)^{d(n-1)} G$;
\item[-] If $n$ is odd and  $\sigma$ is a reflection whose one vertex is fixed (see Figure~\ref{reflection_nodd}), then $\sigma 
\cdot G = (-1)^{\frac{n+1}{2}d}G$;  
\item[-] If $n$ is even and $\sigma$ is a reflection whose two vertices are fixed (see Figure~\ref{reflection_neven1}), then one has $\sigma \cdot G = (-1)^{\frac{nd}{2}-1} G$. 
\end{enumerate}

%

In order to simplify computations, we will write $\lambda_n$ in another form. Let $\mbox{sign} \colon D_n  \lra \{-1,1\}$ be the signature representation restricted from $\Sigma_n$, and let $\mbox{or} \colon D_n \lra \{-1,1\}$ be the orientation representation. Notice that the latter representation concerns only reflections in $D_n$, that is, \lq\lq{}$\mbox{or}$\rq\rq{}
 is equal to $-1$ on reflections, and to $1$ on rotations. From our computations we  obtain 
\begin{equation} \label{sign_representation2}
\lambda_n = \mbox{sign}^{\otimes d} \otimes \mbox{or}^{\otimes (n+d+1)}.
\end{equation}

We will  need the following well known fact~\cite{feit67}. To recall the cycle index sum is defined in Definition~\ref{cycle_index_sum_defn}. 

\begin{lem}\label{l_induction}
Let $f\colon H\to \Sigma_n$ be a group homomorphism and let $\rho^V\colon H\to GL(V)$ be a representation of $H$, then the cycle index sum of the induced representation can be expressed as
\begin{equation}\label{eq:eq_induction}
Z_{\mathrm{Ind}^{\Sigma_n}_HV}(p_1,p_2,\dots)=\frac 1{|H|}\sum_{h\in H} \mathrm{tr}(\rho^V(h))\prod_l p_l^{j_l(f(h))}.
\end{equation}
\end{lem}

Let $RO_n \subseteq D_n$ denote the subset of $D_n$ formed by rotations, and let $RE_n$ denote the subset formed by reflections. 
Using the above lemma we get
\begin{align}
Z_{\mathrm{Ind}^{\Sigma_n}_{D_n}\lambda_n}(p_1, p_2, \cdots) = &   \frac{1}{|D_n|} \sum_{\sigma \in D_n} \lambda_n(\sigma) \prod_l p_l^{j_l(\sigma)} \\
 = & \frac{1}{2n} \sum_{\sigma \in RO_n} \lambda_n(\sigma) \prod_l p_l^{j_l(\sigma)} + \frac{1}{2n} \sum_{\sigma \in RE_n} \lambda_n(\sigma) \prod_l p_l^{j_l(\sigma)} \label{zhn_with_lambdan2}. 
\end{align}
We will denote by $\overline{A}_n(p_1,p_2,\ldots)$ and  $\overline{B}_n(p_1,p_2,\ldots)$ respectively the first and the second summands 
in~\eqref{zhn_with_lambdan2}.

Let $\mathbf{1}$ denote the trivial character of $D_n$. We compute first  the cycle index sum  $Z_{\mathrm{Ind}^{\Sigma_n}_{D_n}\mathbf{1}}$. Then we will add some signs to describe $Z_{\mathrm{Ind}^{\Sigma_n}_{D_n}\lambda_n}$.

\[Z_{\mathrm{Ind}^{\Sigma_n}_{D_n}\mathbf{1}}(p_1, p_2, \cdots) = A_n(p_1,p_2, \cdots)+ B_n(p_1, p_2, \cdots),\]
where
\[A_n = \frac{1}{2n} \sum_{l|n} \varphi (l) p_l^{\frac{n}{l}} \qquad \mbox{and} \qquad B_n = \left\{ \begin{array}{lll}
                                                                                             \frac{1}{2} p_1p_2^{\frac{n-1}{2}} & \mbox{if} & n \ \mbox{odd,}\\
																																														  \frac{1}{4} (p_1^2p_2^{\frac{n-2}{2}} + p_2^{\frac{n}{2}}) & \mbox{if} & n \ \mbox{even,}
                                                                                            \end{array} \right.\]
where $\varphi(l)$ is the Euler\rq{}s totient function that produces the number of positive integers  less than or equal to $l$ that are relatively prime to $l$. Explicitly $\varphi(l)=\sum_{a|l}\mu(a)\frac la$.

Summing over $n\geq 1$, we get
\begin{equation}\label{Z_D}
Z_{\mathrm{Ind}^{\Sigma_\bullet}_{D_\bullet}\mathbf{1}}(p_1, p_2, \cdots) =
-\frac 12\sum_{l\geq 1}\frac{\varphi(l)\ln(1-p_l)}{l}+\frac{p_1^2+p_2+2p_1}{4(1-p_2)}.
\end{equation}

On the other hand, from~\eqref{sign_representation2} we get
\[
\overline{A}_n(p_1,p_2,\ldots)= A_n(p_l \leftarrow (-1)^{d(l-1)}p_l),
\]
\[
 \overline{B}_n(p_1,p_2,\ldots)=  (-1)^{n+d+1} B_n(p_l \leftarrow (-1)^{d(l-1)}p_l)=
 (-1)^{d+1}B_n(p_l \leftarrow (-1)^{d(l-1)+l}p_l).
 \]
 Here the second equality is deduced from the explicit formula for $B_n$.   Since $\sum_{n\geq 1}A_n$ and $\sum_{n\geq 1}B_n$ are respectively the first and second summands in~\eqref{Z_D}, we get
\begin{multline} \label{Z_D_lambda}
Z_{\mathrm{Ind}^{\Sigma_\bullet}_{D_\bullet}\lambda_\bullet}(p_1, p_2, \cdots) = \\
-\frac 12\sum_{l\geq 1}\frac{\varphi(l)\ln(1-(-1)^{d(l-1)}p_l)}{l}+(-1)^{d+1}\,\frac{p_1^2+(-1)^d p_2-2p_1}{4(1-(-1)^d p_2)}.
\end{multline}

From~\eqref{eq_hom_gen1_sigma} and the fact that any genus 1 graph with $k$ external vertices has complexity $k$,
\begin{equation}\label{S_H_M1}
Z_{H_*M_1(P_d^\bullet)}(z,u;p_1,p_2,\ldots) =
Z_{\mathrm{Ind}^{\Sigma_\bullet}_{D_\bullet}\lambda_\bullet}\left(p_l\leftarrow z^{(d-2)l}u^lp_l,\, l\in\nbb\right).
\end{equation}

Recall the \textit{Hodge splitting} from (\ref{hodge_splitting}). 

\begin{thm} \label{genius_one_thm_h}
The generating function of the dimensions  of the Hodge summands of the complex 
$\mathcal{E}_{\pi 1}^{\vec{m},d}$ of hairy graphs of genus one is
\begin{align*}
  R^{\pi 1}_{\vec{m}, d}(x_1, \cdots, x_r, z,u)  =  -\frac 12\sum_{l\geq 1}\frac{\varphi(l)\ln(1-(-1)^{d(l-1)}z^{l(d-2)}u^l\alpha_l(\frac 1z))}{l}+ \\ 
 (-1)^{d+1}\,\frac{z^{2d-4}u^2 \alpha_1(\frac 1z)^2+(-1)^d z^{2d-4}u^2 \alpha_2(\frac 1z)-2z^{d-2}u\alpha_1(\frac 1z)}{4(1-(-1)^d z^{2d-4}u^2\alpha_2(\frac 1z))},  
\end{align*}
where $\alpha_l(\frac 1z) = \sum_{i=1}^r (-1)^{m_i(l-1)} x_i^l (\frac 1z)^{m_il}$.
\end{thm}

  \begin{proof} 
  This formula is obtained from~\eqref{Z_D_lambda} by change of variables: first using~\eqref{S_H_M1} and then~\eqref{observation_proof_genius0} below. 
  \end{proof}

  \begin{thm} \label{genius_one_thm}
The generating function of the Euler characteristics of the Hodge splitting of the complex 
$\mathcal{E}_{\pi 1}^{\vec{m},d}$ of hairy graphs of genus one is
\begin{align*}
  F^{\pi 1}_{\vec{m}, d}(x_1, \cdots, x_r, u)  =  -\frac 12\sum_{l\geq 1}
  \frac{\varphi(l)
  \ln\left(1-(-1)^{d}u^l  \sum_{i=1}^r(-1)^{m_i}x_i^l\right)}{l}+\\ 
  (-1)^{d+1}\,\frac{u^2 \left( \sum_{i=1}^r(-1)^{m_i}x_i   \right)^2+(-1)^d u^2 \sum_{i=1}^r(-1)^{m_i}x_i^2-
2(-1)^du\sum_{i=1}^r(-1)^{m_i}x_i}{4\left(1-(-1)^d u^2\sum_{i=1}^r(-1)^{m_i}x_i^2\right)}.  
\end{align*}
\end{thm}

\begin{proof}
Take $z=-1$ in the previous theorem.
\end{proof}

\section{Supercharacter of the symmetric group action on $\MODL$ and $\MODLdet$}  \label{ss:modular}

%
%

In Subsection~\ref{ss:31} we compute the cycle index sum of the supercharacter of the 
symmetric group action on  the sequence $M(P_d^\bullet)$ introduced in Subsection~\ref{graph_complex_mpdk}.  In Subsection~\ref{ss:proofLinfty} we briefly recall
basic facts about cyclic and modular operads and we explain how $M(P_d^\bullet)$  is related 
to $\MODL$ and $\MODLdet$. At the end we prove Theorem~\ref{t:Z_L_infty}.

\subsection{Supercharacter for $M(P_d^\bullet)$}\label{ss:31}


  Let $M = (\oplus_i M_i,\partial)$ be a finite dimensional chain complex of $\Sigma_k$-modules  over a ground field $\mathbb{K}$ of characteristic $0$.  By the {\it 
supercharacter} we understand the character of the $\Sigma_k$ action on the virtual representation
$\mathcal{X}M$ defined as
$
\mathcal{X}M:=\sum_i (-1)^i M_i.
$
The latter virtual representation is similar to the Euler characteristic in the sense that
$
\mathcal{X}M\simeq \mathcal{X}(H_*M),
$
that\rq{}s why we use this notation. The cycle index sum encoding the supercharacter of the $\Sigma_k$ action on $M$ can be defined as
$
Z_{\mathcal{X}M}=\sum_i(-1)^i Z_{M_i},
$
or equivalently as $Z_M|_{z=-1}$. 

For a symmetric sequence of chain complexes $M=\{M(k)\}_{k\geq 0}$, we similarly define
$
Z_{\mathcal{X}M}:=\sum_{k\geq 0} Z_{\mathcal{X}M(k)}.
$

%
%
%

For the rest of this section, $V$ will denote the $r$-dimensional vector space whose basis is the set of colours  $v_1, \cdots, v_r$,  with each $v_i$ being of degree $m_i$ and of Hodge multi-degree $(\underbrace{0\ldots 0}_{i-1},1,\underbrace{0\ldots 0}_{r-i})$. Consider  the symmetric sequence $V^{\otimes \bullet} = \{V^{\otimes n}\}_{n \geq 0}$. We will need to know the cycle index sum $\zotimes$.  For each $1 \leq i \leq r$, consider the one dimensional vector space $V_i$ spanned by $v_i$, and consider the symmetric sequence $V_i^{\otimes \bullet} = \{V_i^{\otimes n}\}_{n \geq 0}$. One can rewrite the vector space $V$ in the form $V = V_1 \oplus \cdots \oplus V_r$. Therefore, 
\[ V^{\otimes n} = \underset{|\vec{k}| =k}{\bigoplus} \mbox{Ind}^{\Sigma_k}_{\Sigma_{\vec k}} V_1^{\otimes k_1} \otimes \cdots \otimes V_r^{\otimes k_r}.
\]
Recalling Definition~\ref{otimes_hat_defn}   we have $V^{\otimes \bullet} = V_1^{\otimes \bullet} \widehat{\otimes} \cdots \widehat{\otimes} V_r^{\otimes \bullet}$, and  by Lemma~\ref{z_otimes} one has
\begin{equation} \label{zotimes_decomposition}
\zotimes = \prod_{i=1}^r Z_{V_i^{\otimes \bullet}}.
\end{equation} 
For $1 \leq i \leq r$ we will compute $Z_{V_i^{\otimes \bullet}}$. By noticing that the action of $\Sigma_n$ on $V_i^{\otimes n}$ depends on the parity of $m_i$ (for odd $m_i$ the action is the sign representation, which means that if $\sigma \in \Sigma_n, x \in V_i^{\otimes n}$, then $\sigma x= \pm x,$ and for even $m_i$ it is the identity), by also noticing that $V_i$ is a one dimensional vector space, it follows that $V_i^{\otimes \bullet}$ is the commutative unital operad "up to sign".  One has
$
Z_{{{\mathcal C}om}}(p_1,p_2,\ldots)=\mbox{exp}\Bigl(\sum_{l\geq 1}\frac {p_l}l\Bigr),
$
see for example \cite[Section 5]{dots_khoro06}.  We deduce that 
\begin{multline} \label{zidots_formula}
Z_{V_i^{\otimes \bullet}} (z, x_i; p_1, p_2, \cdots) = Z_{{{\mathcal C}om}}(p_l\leftarrow (-1)^{m_i(l-1)}x_i^lz^{m_il}p_l)= \\
\mbox{exp} \left(\sum_{l \geq 1} (-1)^{m_i(l-1)} x_i^l z^{m_il} \frac{p_l}{l} \right),
\end{multline}
where the variable $z$ is responsible for the usual homological degree and $x_i$ is responsible for the $i$-th Hodge grading.
The sign $(-1)^{m_i(l-1)}$ appears because a cycle of length $l$ is an odd representation if and only if $l$
is even. The factors
$x_i^l$  and $z^{m_il}$ encode the fact that $V_i^{\otimes l}$ is concentrated in the Hodge multi-degree 
$(\underbrace{0\ldots 0}_{i-1},l,\underbrace{0\ldots 0}_{r-i})$ and homological degree $m_il.$

  Combining (\ref{zotimes_decomposition}) and (\ref{zidots_formula}), we have 
\begin{equation} \label{zotimes_formula}
\zotimes (z, x_1, \cdots, x_r; p_1, p_2, \cdots) = \mbox{exp} \left(\sum_{l \geq 1} \alpha_l(z, x_1, \cdots, x_r) \frac{p_l}{l} \right),
\end{equation}
where 
\begin{equation} \label{alphal}
\alpha_l(z, x_1, \cdots, x_r) = \sum_{i=1}^r (-1)^{m_i(l-1)} x_i^l z^{m_il}.
\end{equation}
For computations of the Euler characteristics we will need
\begin{equation} \label{alphal_chi}
\alpha_l(-1) = \sum_{i=1}^r (-1)^{m_i} x_i^l .
\end{equation}

For any symmetric sequence $M(\bullet)$, we get 
\begin{equation} \label{observation_proof_genius0}
\begin{array}{lll}
  \mbox{dim hom}_{\Sigma} (V^{\otimes \bullet}, M(\bullet)) & = & \left.\left\{\zotimes (\frac{1}{z}; p_l \leftarrow l\frac{\partial}{ \partial p_l}, l \in \mathbb{N}) Z_{M(\bullet)} (z; p_1,p_2,\ldots)\right\} \right|_{ p_l =0} \\
	& = & Z_{M(\bullet)}\left(z; p_l \leftarrow \alpha_l(\frac{1}{z}, x_1, \cdots, x_r), l \in \mathbb{N}\right).
  \end{array}
\end{equation}
Notation \lq\lq{}$\mathrm{dim}$\rq\rq{}
 stays for the generating function of dimensions that takes into account both homological degree (with $z$ responsible for it) and the Hodge  degrees ($x_1,\ldots,x_r$ are the responsible variables).

The hairy graph-complex $\empi$ we recalled at the beginning of Subsection~\ref{gen_function_htpy_genus_less1} has exactly this form:
\begin{equation}\label{eq_empi_as_sigma_map}
\empi \cong \mbox{hom}_\Sigma(V^{\otimes \bullet}, M(P_d^\bullet)),
\end{equation}
where $M(P_d^k)$ is the graph-complex from Subsection~\ref{graph_complex_mpdk}. 


\begin{thm}\label{t:super_tr1}
The supercharacter of the symmetric group action on the graph-complexes $\{M(P_d^k)\}_{k\geq 1}$ is described by the cycle index sum
\begin{align*}
Z_{\mathcal{X}M(P_d^\bullet)}(u;p_1,p_2,p_3,\ldots) =  & \sum_{k, l, j \geq 1} \frac{\mu(k)}{kj} S_j \left( -\frac 1l\sum_{a|l}\mu\left(\frac la\right)
p_{ak} \right) 
\left(\frac{(-1)^{d-1} l u^{kl}}{F_l(u^k)} \right)^j + \\
 &\sum_{k, l \geq 1}  \frac{\mu(k)}{kl}\left( \sum_{a|l} \mu\left(\frac la\right)p_{ak}\right)  \ln (F_l(u^k)),
\end{align*}
where the variable $u$ is as usual responsible for complexity.
\end{thm}

\begin{proof}
Applying \eqref{observation_proof_genius0} to $M(\bullet)=M(P_d^\bullet)$ and taking $z=-1$, we get
\begin{equation}\label{eq:F_Z}
F^\pi_{\vec{m}, d}(x_1,\ldots,x_r,u)=Z_{\mathcal{X}(M(P_d^\bullet))}(u, p_l\leftarrow\alpha_l(-1)).
\end{equation}
Thus to get $Z_{\mathcal{X}M(P_d^\bullet)}$ from $F^\pi_{\vec{m},d}$ we need to replace each occurence of $\alpha_l(-1)=\sum_{i=1}^r(-1)^{m_i}x_i^l$ back to $p_l$. Using the result of Theorem~\ref{gen_function_htpy_thm} and the fact that $E_l(x)=\frac 1l\sum_{a|l}\mu(\frac la)x^a$ we get the result.
\end{proof}

\subsection{Proof of Theorem~\ref{t:Z_L_infty}}\label{ss:proofLinfty}

\subsubsection{Cyclic and modular operads}\label{ss:cyc_mod}

All the operads that we are going to consider are ones in chain complexes.
Cyclic and modular operads were introduced by E.~Getzler and M.~Kapranov~\cite{getzler_kapranov95,
getzler_kapranov98}. In short a {\it cyclic operad} $O=\{O(n),\, n\geq 0\}$ is a usual symmetric operad 
for which the output of its elements has the same role as the inputs. In particular, each component
$O(n)$ has an action of $\Sigma_{n+1}$. To distinguish the cyclic arity with the usual one, one
writes $O((n+1))$ for $O(n)$.  For usual operads, the category that encodes the ways elements
can be composed is the category of rooted trees, while for cyclic operads, it is the category of unrooted trees
\cite[Section~1]{getzler_kapranov95}. 

 A {\it modular operad} is a {\it stable collection} $M=\{M((g,n));\, g\geq 0, n\geq 0\}$. {\it Stable}
 means $M((g,n))=0$ if $2g+n-2\leq 0$.\footnote{Below we also allow $M((0,2))$ to be one-dimensional being spanned by the identity element. It  appears in our graph-complexes, that\rq{}s why we add it.}  
 The compositions for $M$ are encoded by the categories of stable graphs $\Gamma((g,n))$, $g\geq 0$,
 $n\geq 0$, $2g+n-2>0$.  An element in $\Gamma((g,n))$ is a connected graph $G$ that has
 $n$ {\it external vertices} of valence one and labeled bijectively by $\{1\ldots n\}$, and some set $V(G)$
 of  non-labeled {\it internal vertices}, together with a map $g\colon V(G)\to {\mathbb N}$ 
 to the set of non-negative integers. We will denote by $|v|$ the valence of $v\in V(G)$. One also requires
 $
 \sum_{v\in V(G)}g(v)+\beta_1(G)=g,
 $
 where $\beta_1(G)$ is the first Betti number of $G$. One has a morphism $\rho\colon G_1\to G_2$ in 
 $\Gamma((g,n))$ if $G_2$ is obtained from $G_1$ by a contraction of some subset of internal edges also
 assuming that for any $v\in V(G_2)$, 
 $
 g(v)=\sum_{v'\in V(\rho^{-1}(v))}g(v') + \beta_1(\rho^{-1}(v)).
 $
 Here $\rho^{-1}(v)$ is the preimage of the vertex $v\in V(G_2)$ under the edge contraction $\rho$
 taken together with its small neighborhood, so that $\rho^{-1}(v)$ can be viewed as a stable
 graph in $\Gamma((g(v),|v|))$. (This notation will be used below in the definition of a {\it cocycle}
 or {\it twist} of the modular operadic structure.) 
 
 The terminal element in $\Gamma((g,n))$ is $c_{g,n}$ -- the $n$-corolla with the only vertex of genus $g$.
 The structure of a modular operad is determined by the composition maps
 \[
 M((G)):=\bigotimes_{v\in V(G)}M((g(v),|v|))\to M((g,n))=:M((c_{g,n})),
 \]
 corresponding to the morphisms $G\to c_{g,n}$ in $\Gamma((g,n))$ $g\geq 0$, $n\geq 0$, that should satisfy natural associativity properties. Here  we implicitly assume that each $M((g,n))$ has a $\Sigma_n$
 action. 
 
 One has an adjunction
 $
 \mathbf{Mod}\colon \mathrm{CycOp}\rightleftarrows\mathrm{ModOp}\colon\mathbf{Cyc}
 $
 between the categories of cyclic and modular operads, where $\mathbf{Cyc}$ assigns
  to a modular operad its $g=0$ part. Its left adjoint functor $\mathbf{Mod}$ assigns 
  to a cyclic operad its modular envelope~\cite{hin_vaintrob02}. As relevant to us examples,
  ${\mathcal L}ie$ and its Koszul resulution ${\mathcal L}_\infty$ are cyclic operads, for which one defines $\mathbf{Mod}({\mathcal L}ie)$
  and $\MODL$, cf. loc. cit. 

The notion of a modular operad is more subtle than it might appear at the first sight. It comes
in different {\it twisted} versions. Denote by $\mathrm{Iso}\Gamma((g,n))$, $g\geq 0$, $n\geq 0$,
the groupoids of isomorphisms of stable graphs. For any stable graph $G$, let $Aut_G$ denote
the endomorphism set of the graph $G$ in one of these categories. It is the group of symmetries of $G$
that can also permute its external vertices (if symmetries allow). A {\it cocycle} $\mathfrak{D}$ is a
family of functors
$
\mathfrak{D}\colon \mathrm{Iso}\Gamma((g,n))\to \mathrm{grVect},\, g\geq 0,\, n\geq 0,
$
to the category of graded vector spaces, that always assigns a one-dimensional vector space, and that
has in addition the structure of a {\it hyperoperad}: to each morphism $\rho\colon G_1\to G_2$
in $\Gamma((g,n))$, it is assigned a map
$
\nu_\rho\colon\mathfrak{D}(G_2)\otimes\bigotimes_{v\in V(G_2)}\mathfrak{D}(\rho^{-1}(v))
\to \mathfrak{D}(G_1)
$
satisfying natural axioms~\cite[Section~4.1]{getzler_kapranov98}. For any cocycle $\mathfrak{D}$,
a $\mathfrak{D}$-twisted modular operad is a stable sequence $M=\{M((g,n)),\, g\geq 0,\, n\geq 0\}$
endowed with composition maps
\[
\mathfrak{D}(G)\otimes M((G))\to M((g,n)) = M((c_{g,n}))
\]
for any $G\in \Gamma((g,n))$, $g\geq 0$, $n\geq 0$.

The cocycle $\Det$ from \cite{getzler_kapranov98} is of a special interest to us. For a  vector space
$W$ of dimension $k$ define
$
\Det\, W=\Sigma^{-k}\Lambda^kW.
$
It is a one-dimensional vector space in degree $-k$. One has
\begin{equation}\label{eq:det_prod}
\Det(W_1\oplus W_2)=\Det\, W_1\otimes \Det\, W_2.
\end{equation}
Below, for a tensor product of one-dimensional vector spaces we will be using \lq\lq{}$\cdot$\rq\rq{} instead of
\lq\lq{}$\otimes$\rq\rq{}. Also the dual of a one-dimensional space $X$ will be denoted by $X^{-1}$. 

The cocycle $\Det$ is defined as
$
\Det(G):=\Det\, H_1(G),
$
for any stable graph $G\in\Gamma((g,n))$. A peculiar property of $\Det$ is that it restricts trivially on 
$\mathrm{Iso}\Gamma((0,n))$. This means that $\Det$-twisted cyclic operads are the usual cyclic ones. One gets
a similar adjunction
$
\mathbf{Mod}_\Det\colon\mathrm{CycOp}\rightleftarrows\mathrm{ModOp}_\Det\colon\mathbf{Cyc}
$
between the usual cyclic operads and the $\Det$-twisted modular ones. This in particular produces
$\MODLdet$ -- the $\Det$-twisted modular envelope of ${\mathcal L}_\infty$.

\subsubsection{$M(P_d^\bullet)$ as $\MODL$ and $\MODLdet$}\label{ss:lemmas_envelops}

For any stable collection $\{M((g,n))\}$ define a symmetric sequence $M((\bullet))=\{\oplus_gM(g,n),\, n\geq 0\}$.

Consider now the modular envelope  $\MODL$. 
  It is easy to notice that graph-complexes $M(P_d^\bullet)$ are closely related to the components of $\MODL((\bullet))$. 
  They are spanned by the same combinatorial graphs, where the grading genus corresponds to the first Betti number of the
  graphs.
  If we choose $d=3$ and also tensor each component $M(P_d^k)$ with the sign representation and take a shift in degree (desuspension) by $k$ we get $\MODL((k))$. 
  \begin{lem}\label{l:L_infty}
  For any $k\geq 1$,
\begin{equation}\label{eq:L_infty}
\MODL((k))\cong_{\Sigma_k} \Sigma^{-k}M(P_3^k)\otimes sign.
\end{equation}
\end{lem}
\begin{proof}
Combinatorially $\MODL((k))$ is a graph-complex consisting of exactly the same graphs as $M(P_3^k)$, so we only need to 
work out the signs and degrees properly. The operad ${\mathcal L}_\infty$ is cyclic and is freely generated by operations of 
cyclic arity $l$, $l\geq 3$, which have degree $l-3$ and the sign action of $\Sigma_{l}$.  Graphically such operations correspond to vertices of arity $l$. 
As a conclusion, to orient a graph $G\in \MODL((k))$ we need to order its vertices, where a vertex $v$  of valence $|v|$ is considered as element of degree $|v|-3$, and for each vertex to order edges adjacent to it. Changing the order of adjacent edges at any vertex gives the sign of permutation; changing the order of vertices gives the \textit{Koszul sign of permutation} (that is, the sign of permutation that takes  into account the degree of elements.) Now when we look at a graph $G\in M(P_3^k)$,
it is oriented by ordering the set of its vertices (considered as elements of degree $-3$) and edges
 (considered as elements of degree $2$, therefore their placement in the orientation set can be ignored), and by orienting
 all edges. Changing orientation of an edge gives a negative sign. Now, we replace each edge in the orientation set by its two 
 half-edges in the order -- first source, second target. Then we change the order of the elements in the orientation set so that
 the vertices and adjacent to it half-edges come in one block -- first the vertex than half-edges. The combined block 
 corresponding to any vertex $v$  has degree exactly $|v|-3$. Notice however, that $k$ half-edges, corresponding to the
 external vertices, don\rq{}t appear in any such block. These half-edges get annihilated with $\Sigma^{-k}sign$ in~\eqref{eq:L_infty}.
 
 For this argument one should consider separately the case of the identity element $id\in\MODL((0,2))$. It is described as
 a graph that has two external vetrices connected by an edge (no internal vertices). The same graph in $M_0(P_3^2)$ has degree
 two and enjoys the sign action of $\Sigma_2$. The degree and sign shift exactly correspond to the statement of the lemma.
 
 The fact that the differentials agree, which in both cases are sums of expansions of vertices, is a tedious, but straightforward check.
\end{proof}

One has a similar description for the components of $\MODLdet((\bullet))$ in terms of $M(P_2^\bullet)$.

\begin{lem}\label{l:L_infty_det}
 For any $k\geq 1$, 
$\MODLdet((k))\cong_{\Sigma_k} \Sigma^{-1}M(P_2^k).$
\end{lem}

\begin{proof}
Similarly to the previous lemma, we need to check that the signs and gradings agree. 

Forgetting the differential, ${\mathcal L}_\infty$ is a free cyclic operad generated by a sequence of one-dimensional vector spaces. As a consequence, $\MODL$ and its twisted version $\MODLdet$ are also free (twisted) modular operads
generated by the same sequence viewed as a stable collection concentrated in genus $g=0$. Thus for any stable graph $G$,
for which $g|_{V(G)}\equiv 0$, there corresponds exactly one graph in $\MODL$ and $\MODLdet$, that, abusing notation,
we also denote by~$G$. Recall that $Aut_G$ is the group of symmetries of~$G$. Denote by $Aut_G^{int}$ its subgroup
of elements fixing external vertices of $G$ pointwise. 

Denote by $or(G)$, respectively $or_\Det(G)$, the one dimensional sign representation of $Aut_G$ that describes how the sign
of $G$ in $\MODL$, respectively $\MODLdet$, changes when the symmetries get applied. One obviously has that 
$G=0$ in $\MODL$, respectively $\MODLdet$, if the restriction of this sign representation on $Aut_G^{int}$ is 
non-trivial. We consider $or(G)$ and $or_\Det(G)$ as one-dimensional graded vector spaces concentrated in the degree
of~$G$, which is $|G|=\sum_{v\in V(G)}(|v|-3)$ for $or(G)$ and $|G|=\sum_{v\in V(G)}(|v|-3)-\beta_1(G)$ for
$or_\Det(G)$. 

One can view $G$ as an element of $M(P_2^k)$, respectively $M(P_3^k)$. One similarly defines $or_2(G)$ and $or_3(G)$
-- the corresponding graded one-dimensional $Aut_G$-modules.

To prove the lemma, we should show that $or_\Det(G)=\Sigma^{-1}or_2(G)$. 

\begin{itemize}
\item Let $C_0^{int}(G)$ denote the vector space of cellular 0-chains spanned by the internal vertices of $G$.
\item Let $C_1^{int}(G)$ denote the vector space of cellular 1-chains spanned by the internal edges of $G$. Here as usual, 
one considers oriented edges; changing orientation implies change in sign: $\overrightarrow{e}=-\overleftarrow{e}$.
\item Let $C_1^{ext}(G)$ denote the vector space of cellular 1-chains spanned by the external edges of $G$.
\item We also set $C_1(G)=C_1^{int}(G)\oplus C_1^{ext}(G)$.
\end{itemize}
All these spaces are viewed as $Aut_G$-modules concentrated in degree zero. One has an exact sequence of $Aut_G$-modules:

\begin{equation} \label{eq:exact}
0\leftarrow H_0(G)\leftarrow C_0^{int}(G) \stackrel{\partial}{\leftarrow} C_1^{int}(G)\leftarrow H_1(G)\leftarrow 0.
\end{equation}

From Lemma~\ref{l:L_infty}, one has
$
or(G)=or_3(G)\cdot\Det\, C_1^{ext}(G).
$
From the definition of $M(P_d^k)$, one can easily get
$
or_2(G)=or_3(G)\cdot \Det\, C_1(G)\cdot (\Det\, C_0^{int}(G))^{-1}.
$
From the definition of the $\Det$-cocycle, and using the fact that $\MODL$ and $\MODLdet$ are free, one has
\[
or_\Det(G)=or(G)\cdot\Det\, H_1(G)=or(G)\cdot\Det\, C_1^{int}(G)\cdot(\Det\, C_0^{int}(G))^{-1}\cdot\Det\, H_0(G).
\]
The last equation follows from~\eqref{eq:exact} and~\eqref{eq:det_prod}. 

From the three identities above, we get
\begin{multline*}
or_\Det(G)=\Sigma^{-1} or(G)\cdot\Det\, C_1^{int}(G)\cdot(\Det\, C_0^{int}(G))^{-1}
= \\
\Sigma^{-1}or_3(G)\cdot\Det\, C_1(G)\cdot(\Det\, C_0^{int}(G))^{-1}=
\Sigma^{-1}or_2(G).
\end{multline*}
The case of the identity element $id\in\MODLdet((0,2))$ should be considered separately. The corresponding
graph in $M_0(P_2^2)$ has degree~1 and trivial action of $\Sigma_2$, which is compatible with the statement of the lemma.
\end{proof}

For $d>3$, the components $M_g(P_d^k)$ can also be expressed in terms of $\MODL((g,k))$ and $\MODLdet((g,k))$ using some regrading, which follows from the isomorphism
$M_g(P_{d+2}^k)\simeq_{\Sigma_k}\Sigma^{2(g-1)+2k}M_g(P_d^k).$

The lemma below is aimed to the readers more familiar with the Feynman transforms rather than with the modular envelope. 

For a stable collection $M=\{M((g,n)),\, g\geq 0,\, n\geq 0\}$, one denotes by $\Sigma M$ the objectwise suspension 
$\Sigma M=\{\Sigma M((g,n)),\, g\geq 0,\, n\geq 0\}$, and one denotes by $\mathfrak{s}M$ its modular operadic suspension
\[
\mathfrak{s}M=\{\mathfrak{s}M((g,n)),\,g\geq 0,\, n\geq 0\}=\{\Sigma^{-2(g-1)-n}M((g,n))\otimes sign,\,g\geq 0,\, n\geq 0\}.
\]

\begin{lem}\label{l:fcom}
 The modular envelopes of ${\mathcal L}ie$ and the Feynman transforms  of ${\mathcal C}om$ are related to each other by the following regrading: $\FCOMdet=\Sigma\mathfrak{s}\MODL$,
$\FCOM=\Sigma\mathfrak{s}\MODLdet$.
\end{lem}

\begin{proof}
To the regradings $\Sigma$ and $\mathfrak{s}$ one assigns the twisting cocycles $\mathfrak{D}_\Sigma$ and 
$\mathfrak{D}_{\mathfrak s}$, respectively, see~\cite{getzler_kapranov98}. One has that $\MODLdet$ is a free $\Det$-twisted modular operad
generated by $\Sigma^{-1}\mathfrak{s}^{-1}{\mathcal C}om$. Thus $\Sigma\mathfrak{s}\MODLdet$ is
a free $(\Det\cdot \mathfrak{D}_\Sigma \cdot\mathfrak{D}_{\mathfrak s}=:\mathfrak{K})$-twisted modular operad generated
by ${\mathcal C}om((\bullet))$, which is exactly the definition of $\FCOM$. Similarly, $\Sigma\mathfrak{s}\MODL$
is a free $(\mathfrak{D}_\Sigma \cdot\mathfrak{D}_{\mathfrak s}=\mathfrak{K}\cdot\Det^{-1})$-twisted modular operad
generated by ${\mathcal C}om((\bullet))$,
which is the definition of $\FCOMdet$.
\end{proof}

\subsubsection{Proof of the theorem}
We concentrate on positive arities. The case of arity zero and a connection to Willwacher-\v{Z}ivkovi\'c\rq{}
computations~\cite{wil_zh} have been explained in the introduction.  
For a symmetric sequence $\{M(n),\, n\geq 0\}$, we denote by $Z^{>0}_M$ the cycle index sum of $\{M(n),\, n\geq 1\}$.

Now we have all ingredients to prove Theorem~\ref{t:Z_L_infty}: Lemmas~\ref{l:L_infty}-\ref{l:L_infty_det} and Theorem~\ref{t:super_tr1}.  In the cycle index sum of the latter result, there is  variable $u$, which is responsible for complexity.   
It follows from~\eqref{eq:L_infty} that
\begin{equation}\label{eq:ch1}
Z^{>0}_{\mathcal{X}\MODL}(u;p_1,p_2,p_3,\ldots)=Z_{\mathcal{X}M(P_3^\bullet)}(u;-p_1,-p_2,-p_3,\ldots).
\end{equation}
Since $\MODL$ is a modular operad, it is more natural to consider the splitting by genus rather than by complexity. We use the variable $\hbar$ as the one responsible for the genus.  
 It follows from~\eqref{eq:compl_gen} 
that we need to make the change of variables $u\leftarrow \hbar$, $p_l\leftarrow \frac{p_l}{\hbar^l}$, and in addition to it, multiply  the result by $\hbar$. (Compare with~\eqref{eq:F_pi_genus}.)  Combining it with~\eqref{eq:ch1}, we get
\[
Z^{>0}_{\mathcal{X}\MODL}(\hbar;p_1,p_2,p_3,\ldots)=\hbar Z_{\mathcal{X}M(P_3^\bullet)}(u\leftarrow \hbar;p_l\leftarrow -\frac{p_l}{\hbar^l},\, l\in \nbb).
\]
To finish the proof of~\eqref{eq:ch_odd} we apply Theorem~\ref{t:super_tr1} for $d=3$. 

To prove~\eqref{eq:ch_even}, we similarly get
\[
Z^{>0}_{\mathcal{X}\MODLdet}(\hbar;p_1,p_2,p_3,\ldots)=-\hbar Z_{\mathcal{X}M(P_2^\bullet)}(u\leftarrow \hbar;p_l\leftarrow \frac{p_l}{\hbar^l},\, l\in \nbb).
\]
And then apply Theorem~\ref{t:super_tr1} for $d=2$. 

%
%


\appendix

\section{Tables of Euler characteristics}  

Here we present results of computer calculations which where produced using Mathematica. Recall the splitting of the complex $\mathcal{E}^{\vec{m}, d}_{\pi}$ from (\ref{splitting_epi}). One can split it again into a direct sum 
$
\mathcal{E}^{\vec{m}, d}_{\pi} = \underset{g \geq 0}{\bigoplus} \underset{\vec{s}, t}{\bigoplus} \mathcal{E}^{\vec{m}, d}_{\pi g, \vec{s}, t}.
$
Let $\chi^{\pi g}_{\vec{s}, t}$ denote the  Euler characteristic of each summand in that splitting. The following tables furnish results of $\chi^{\pi g}_{s_1, s_2, t}$ (that is, in the case $r=2$) for genus $g \in \{0, 1, 2, 3\}$ with $m_1, m_2$ and $d$ odd.  Recall the formula $g+s_1+s_2 = t+1$ from (\ref{eq:compl_gen}). 

\tiny

\begin{table}[ht!]

{\tiny \begin{center}
\begin{tabular}{|c|c|c|c|c|c|c|c|c|c|c|c|c|c|c|c|c|c|c|c|c|c|c|c|c|}
\hline
$t$ &\multicolumn{23}{|c|}{Hodge degree $s_2$}& \\ \hline
  & 0 & 1 & 2 & 3 & 4 & 5 & 6 & 7 & 8 & 9 & 10 & 11 & 12 & 13 & 14 & 15 & 16 & 17 & 18 & 19 & 20 & 21 & 22 & 23 \\ \hline
 1 & \begin{turn}{80}{1}\end{turn} & \begin{turn}{80}{1}\end{turn} & \begin{turn}{80}{1}\end{turn} & \begin{turn}{80}{0}\end{turn} & \begin{turn}{80}{0}\end{turn} & \begin{turn}{80}{0}\end{turn} & \begin{turn}{80}{0}\end{turn} & \begin{turn}{80}{0}\end{turn} & \begin{turn}{80}{0}\end{turn} & \begin{turn}{80}{0}\end{turn} & \begin{turn}{80}{0}\end{turn} & \begin{turn}{80}{0}\end{turn} & \begin{turn}{80}{0}\end{turn} & \begin{turn}{80}{0}\end{turn} & \begin{turn}{80}{0}\end{turn} & \begin{turn}{80}{0}\end{turn} & \begin{turn}{80}{0}\end{turn} & \begin{turn}{80}{0}\end{turn} & \begin{turn}{80}{0}\end{turn} & \begin{turn}{80}{0}\end{turn} & \begin{turn}{80}{0}\end{turn} & \begin{turn}{80}{0}\end{turn} & \begin{turn}{80}{0}\end{turn} & \begin{turn}{80}{0}\end{turn} \\ 
\hline
 2 & \begin{turn}{80}{0}\end{turn} & \begin{turn}{80}{0}\end{turn} & \begin{turn}{80}{0}\end{turn} & \begin{turn}{80}{0}\end{turn} & \begin{turn}{80}{0}\end{turn} & \begin{turn}{80}{0}\end{turn} & \begin{turn}{80}{0}\end{turn} & \begin{turn}{80}{0}\end{turn} & \begin{turn}{80}{0}\end{turn} & \begin{turn}{80}{0}\end{turn} & \begin{turn}{80}{0}\end{turn} & \begin{turn}{80}{0}\end{turn} & \begin{turn}{80}{0}\end{turn} & \begin{turn}{80}{0}\end{turn} & \begin{turn}{80}{0}\end{turn} & \begin{turn}{80}{0}\end{turn} & \begin{turn}{80}{0}\end{turn} & \begin{turn}{80}{0}\end{turn} & \begin{turn}{80}{0}\end{turn} & \begin{turn}{80}{0}\end{turn} & \begin{turn}{80}{0}\end{turn} & \begin{turn}{80}{0}\end{turn} & \begin{turn}{80}{0}\end{turn} & \begin{turn}{80}{0}\end{turn} 
\\ \hline
 3 & \begin{turn}{80}{0}\end{turn} & \begin{turn}{80}{0}\end{turn} & \begin{turn}{80}{1}\end{turn} & \begin{turn}{80}{0}\end{turn} & \begin{turn}{80}{0}\end{turn} & \begin{turn}{80}{0}\end{turn} & \begin{turn}{80}{0}\end{turn} & \begin{turn}{80}{0}\end{turn} & \begin{turn}{80}{0}\end{turn} & \begin{turn}{80}{0}\end{turn} & \begin{turn}{80}{0}\end{turn} & \begin{turn}{80}{0}\end{turn} & \begin{turn}{80}{0}\end{turn} & \begin{turn}{80}{0}\end{turn} & \begin{turn}{80}{0}\end{turn} & \begin{turn}{80}{0}\end{turn} & \begin{turn}{80}{0}\end{turn} & \begin{turn}{80}{0}\end{turn} & \begin{turn}{80}{0}\end{turn} & \begin{turn}{80}{0}\end{turn} & \begin{turn}{80}{0}\end{turn} & \begin{turn}{80}{0}\end{turn} & \begin{turn}{80}{0}\end{turn} & \begin{turn}{80}{0}\end{turn} 
\\ \hline
 4 & \begin{turn}{80}{0}\end{turn} & \begin{turn}{80}{0}\end{turn} & \begin{turn}{80}{0}\end{turn} & \begin{turn}{80}{0}\end{turn} & \begin{turn}{80}{0}\end{turn} & \begin{turn}{80}{0}\end{turn} & \begin{turn}{80}{0}\end{turn} & \begin{turn}{80}{0}\end{turn} & \begin{turn}{80}{0}\end{turn} & \begin{turn}{80}{0}\end{turn} & \begin{turn}{80}{0}\end{turn} & \begin{turn}{80}{0}\end{turn} & \begin{turn}{80}{0}\end{turn} & \begin{turn}{80}{0}\end{turn} & \begin{turn}{80}{0}\end{turn} & \begin{turn}{80}{0}\end{turn} & \begin{turn}{80}{0}\end{turn} & \begin{turn}{80}{0}\end{turn} & \begin{turn}{80}{0}\end{turn} & \begin{turn}{80}{0}\end{turn} & \begin{turn}{80}{0}\end{turn} & \begin{turn}{80}{0}\end{turn} & \begin{turn}{80}{0}\end{turn} & \begin{turn}{80}{0}\end{turn} 
\\ \hline
 5 & \begin{turn}{80}{0}\end{turn} & \begin{turn}{80}{0}\end{turn} & \begin{turn}{80}{1}\end{turn} & \begin{turn}{80}{1}\end{turn} & \begin{turn}{80}{1}\end{turn} & \begin{turn}{80}{0}\end{turn} & \begin{turn}{80}{0}\end{turn} & \begin{turn}{80}{0}\end{turn} & \begin{turn}{80}{0}\end{turn} & \begin{turn}{80}{0}\end{turn} & \begin{turn}{80}{0}\end{turn} & \begin{turn}{80}{0}\end{turn} & \begin{turn}{80}{0}\end{turn} & \begin{turn}{80}{0}\end{turn} & \begin{turn}{80}{0}\end{turn} & \begin{turn}{80}{0}\end{turn} & \begin{turn}{80}{0}\end{turn} & \begin{turn}{80}{0}\end{turn} & \begin{turn}{80}{0}\end{turn} & \begin{turn}{80}{0}\end{turn} & \begin{turn}{80}{0}\end{turn} & \begin{turn}{80}{0}\end{turn} & \begin{turn}{80}{0}\end{turn} & \begin{turn}{80}{0}\end{turn} \\ 
\hline
 6 &\begin{turn}{80}{0}\end{turn} &  \begin{turn}{80}{0}\end{turn} & \begin{turn}{80}{0}\end{turn} & \begin{turn}{80}{0}\end{turn} & \begin{turn}{80}{0}\end{turn} & \begin{turn}{80}{0}\end{turn} & \begin{turn}{80}{0}\end{turn} & \begin{turn}{80}{0}\end{turn} & \begin{turn}{80}{0}\end{turn} & \begin{turn}{80}{0}\end{turn} & \begin{turn}{80}{0}\end{turn} & \begin{turn}{80}{0}\end{turn} & \begin{turn}{80}{0}\end{turn} & \begin{turn}{80}{0}\end{turn} & \begin{turn}{80}{0}\end{turn} & \begin{turn}{80}{0}\end{turn} & \begin{turn}{80}{0}\end{turn} & \begin{turn}{80}{0}\end{turn} & \begin{turn}{80}{0}\end{turn} & \begin{turn}{80}{0}\end{turn} & \begin{turn}{80}{0}\end{turn} & \begin{turn}{80}{0}\end{turn} & \begin{turn}{80}{0}\end{turn} & \begin{turn}{80}{0}\end{turn} \\ 
\hline
 7 & \begin{turn}{80}{0}\end{turn}& \begin{turn}{80}{0}\end{turn} & \begin{turn}{80}{1}\end{turn} & \begin{turn}{80}{1}\end{turn} & \begin{turn}{80}{2}\end{turn} & \begin{turn}{80}{1}\end{turn} & \begin{turn}{80}{1}\end{turn} & \begin{turn}{80}{0}\end{turn} & \begin{turn}{80}{0}\end{turn} & \begin{turn}{80}{0}\end{turn} & \begin{turn}{80}{0}\end{turn} & \begin{turn}{80}{0}\end{turn} & \begin{turn}{80}{0}\end{turn} & \begin{turn}{80}{0}\end{turn} & \begin{turn}{80}{0}\end{turn} & \begin{turn}{80}{0}\end{turn} & \begin{turn}{80}{0}\end{turn} & \begin{turn}{80}{0}\end{turn} & \begin{turn}{80}{0}\end{turn} & \begin{turn}{80}{0}\end{turn} & \begin{turn}{80}{0}\end{turn} & \begin{turn}{80}{0}\end{turn} & \begin{turn}{80}{0}\end{turn} & \begin{turn}{80}{0}\end{turn} \\ 
\hline
 8 & \begin{turn}{80}{0}\end{turn} & \begin{turn}{80}{0}\end{turn} & \begin{turn}{80}{0}\end{turn} & \begin{turn}{80}{1}\end{turn} & \begin{turn}{80}{1}\end{turn} & \begin{turn}{80}{1}\end{turn} & \begin{turn}{80}{1}\end{turn} & \begin{turn}{80}{0}\end{turn} & \begin{turn}{80}{0}\end{turn} & \begin{turn}{80}{0}\end{turn} & \begin{turn}{80}{0}\end{turn} & \begin{turn}{80}{0}\end{turn} & \begin{turn}{80}{0}\end{turn} & \begin{turn}{80}{0}\end{turn} & \begin{turn}{80}{0}\end{turn} & \begin{turn}{80}{0}\end{turn} & \begin{turn}{80}{0}\end{turn} & \begin{turn}{80}{0}\end{turn} & \begin{turn}{80}{0}\end{turn} & \begin{turn}{80}{0}\end{turn} & \begin{turn}{80}{0}\end{turn} & \begin{turn}{80}{0}\end{turn} & \begin{turn}{80}{0}\end{turn} & \begin{turn}{80}{0}\end{turn}  \\ 
\hline
 9 &\begin{turn}{80}{0}\end{turn} & \begin{turn}{80}{0}\end{turn} & \begin{turn}{80}{1}\end{turn} & \begin{turn}{80}{1}\end{turn} & \begin{turn}{80}{3}\end{turn} & \begin{turn}{80}{3}\end{turn} & \begin{turn}{80}{3}\end{turn} & \begin{turn}{80}{1}\end{turn} & \begin{turn}{80}{1}\end{turn} & \begin{turn}{80}{0}\end{turn} & \begin{turn}{80}{0}\end{turn} & \begin{turn}{80}{0}\end{turn} & \begin{turn}{80}{0}\end{turn} & \begin{turn}{80}{0}\end{turn} & \begin{turn}{80}{0}\end{turn} & \begin{turn}{80}{0}\end{turn} & \begin{turn}{80}{0}\end{turn} & \begin{turn}{80}{0}\end{turn} & \begin{turn}{80}{0}\end{turn} & \begin{turn}{80}{0}\end{turn} & \begin{turn}{80}{0}\end{turn} & \begin{turn}{80}{0}\end{turn} & \begin{turn}{80}{0}\end{turn} & \begin{turn}{80}{0}\end{turn} \\ 
\hline
 10 & \begin{turn}{80}{0}\end{turn} & \begin{turn}{80}{0}\end{turn} & \begin{turn}{80}{0}\end{turn} & \begin{turn}{80}{1}\end{turn} & \begin{turn}{80}{2}\end{turn} & \begin{turn}{80}{3}\end{turn} & \begin{turn}{80}{3}\end{turn} & \begin{turn}{80}{2}\end{turn} & \begin{turn}{80}{1}\end{turn} & \begin{turn}{80}{0}\end{turn} & \begin{turn}{80}{0}\end{turn} & \begin{turn}{80}{0}\end{turn} & \begin{turn}{80}{0}\end{turn} & \begin{turn}{80}{0}\end{turn} & \begin{turn}{80}{0}\end{turn} & \begin{turn}{80}{0}\end{turn} & \begin{turn}{80}{0}\end{turn} & \begin{turn}{80}{0}\end{turn} & \begin{turn}{80}{0}\end{turn} & \begin{turn}{80}{0}\end{turn} & \begin{turn}{80}{0}\end{turn} & \begin{turn}{80}{0}\end{turn} & \begin{turn}{80}{0}\end{turn} & \begin{turn}{80}{0}\end{turn}  \\ 
\hline
 11 &  \begin{turn}{80}{0}\end{turn} & \begin{turn}{80}{0}\end{turn} & \begin{turn}{80}{1}\end{turn} & \begin{turn}{80}{2}\end{turn} & \begin{turn}{80}{5}\end{turn} & \begin{turn}{80}{6}\end{turn} & \begin{turn}{80}{9}\end{turn} & \begin{turn}{80}{6}\end{turn} & \begin{turn}{80}{5}\end{turn} & \begin{turn}{80}{2}\end{turn} & \begin{turn}{80}{1}\end{turn} & \begin{turn}{80}{0}\end{turn} & \begin{turn}{80}{0}\end{turn} & \begin{turn}{80}{0}\end{turn} & \begin{turn}{80}{0}\end{turn} & \begin{turn}{80}{0}\end{turn} & \begin{turn}{80}{0}\end{turn} & \begin{turn}{80}{0}\end{turn} & \begin{turn}{80}{0}\end{turn} & \begin{turn}{80}{0}\end{turn} & \begin{turn}{80}{0}\end{turn} & \begin{turn}{80}{0}\end{turn} & \begin{turn}{80}{0}\end{turn} & \begin{turn}{80}{0}\end{turn}  \\ 
\hline
 12 & \begin{turn}{80}{0}\end{turn} & \begin{turn}{80}{0}\end{turn} & \begin{turn}{80}{0}\end{turn} & \begin{turn}{80}{1}\end{turn} & \begin{turn}{80}{3}\end{turn} & \begin{turn}{80}{7}\end{turn} & \begin{turn}{80}{9}\end{turn} & \begin{turn}{80}{9}\end{turn} & \begin{turn}{80}{7}\end{turn} & \begin{turn}{80}{3}\end{turn} & \begin{turn}{80}{1}\end{turn} & \begin{turn}{80}{0}\end{turn} & \begin{turn}{80}{0}\end{turn} & \begin{turn}{80}{0}\end{turn} & \begin{turn}{80}{0}\end{turn} & \begin{turn}{80}{0}\end{turn} & \begin{turn}{80}{0}\end{turn} & \begin{turn}{80}{0}\end{turn} & \begin{turn}{80}{0}\end{turn} & \begin{turn}{80}{0}\end{turn} & \begin{turn}{80}{0}\end{turn} & \begin{turn}{80}{0}\end{turn} & \begin{turn}{80}{0}\end{turn} & \begin{turn}{80}{0}\end{turn}  \\ 
\hline
 13 &\begin{turn}{80}{0}\end{turn} & \begin{turn}{80}{0}\end{turn} & \begin{turn}{80}{1}\end{turn} & \begin{turn}{80}{2}\end{turn} & \begin{turn}{80}{7}\end{turn} & \begin{turn}{80}{11}\end{turn} & \begin{turn}{80}{19}\end{turn} & \begin{turn}{80}{19}\end{turn} & \begin{turn}{80}{19}\end{turn} & \begin{turn}{80}{11}\end{turn} & \begin{turn}{80}{7}\end{turn} & \begin{turn}{80}{2}\end{turn} & \begin{turn}{80}{1}\end{turn} & \begin{turn}{80}{0}\end{turn} & \begin{turn}{80}{0}\end{turn} & \begin{turn}{80}{0}\end{turn} & \begin{turn}{80}{0}\end{turn} & \begin{turn}{80}{0}\end{turn} & \begin{turn}{80}{0}\end{turn} & \begin{turn}{80}{0}\end{turn} & \begin{turn}{80}{0}\end{turn} & \begin{turn}{80}{0}\end{turn} & \begin{turn}{80}{0}\end{turn} & \begin{turn}{80}{0}\end{turn}  \\ 
 \hline
 14  & \begin{turn}{80}{0}\end{turn} & \begin{turn}{80}{0}\end{turn} & \begin{turn}{80}{0}\end{turn} & \begin{turn}{80}{2}\end{turn} & \begin{turn}{80}{5}\end{turn} & \begin{turn}{80}{13}\end{turn} & \begin{turn}{80}{22}\end{turn} & \begin{turn}{80}{28}\end{turn} & \begin{turn}{80}{28}\end{turn} & \begin{turn}{80}{22}\end{turn} & \begin{turn}{80}{13}\end{turn} & \begin{turn}{80}{5}\end{turn} & \begin{turn}{80}{2}\end{turn} & \begin{turn}{80}{0}\end{turn} & \begin{turn}{80}{0}\end{turn} & \begin{turn}{80}{0}\end{turn} & \begin{turn}{80}{0}\end{turn} & \begin{turn}{80}{0}\end{turn} & \begin{turn}{80}{0}\end{turn} & \begin{turn}{80}{0}\end{turn} & \begin{turn}{80}{0}\end{turn} & \begin{turn}{80}{0}\end{turn} & \begin{turn}{80}{0}\end{turn} & \begin{turn}{80}{0}\end{turn}  \\ 
\hline 
 15 & \begin{turn}{80}{0}\end{turn}   &  \begin{turn}{80}{0}\end{turn} & \begin{turn}{80}{1}\end{turn} & \begin{turn}{80}{2}\end{turn} & \begin{turn}{80}{9}\end{turn} & \begin{turn}{80}{18}\end{turn} & \begin{turn}{80}{36}\end{turn} & \begin{turn}{80}{47}\end{turn} & \begin{turn}{80}{58}\end{turn} & \begin{turn}{80}{47}\end{turn} & \begin{turn}{80}{36}\end{turn} & \begin{turn}{80}{18}\end{turn} & \begin{turn}{80}{9}\end{turn} & \begin{turn}{80}{2}\end{turn} & \begin{turn}{80}{1}\end{turn} & \begin{turn}{80}{0}\end{turn} & \begin{turn}{80}{0}\end{turn} & \begin{turn}{80}{0}\end{turn} & \begin{turn}{80}{0}\end{turn} & \begin{turn}{80}{0}\end{turn} & \begin{turn}{80}{0}\end{turn} & \begin{turn}{80}{0}\end{turn} & \begin{turn}{80}{0}\end{turn} & \begin{turn}{80}{0}\end{turn} \\ 
\hline
 16 & \begin{turn}{80}{0}\end{turn} & \begin{turn}{80}{0}\end{turn} & \begin{turn}{80}{0}\end{turn} & \begin{turn}{80}{2}\end{turn} & \begin{turn}{80}{7}\end{turn} & \begin{turn}{80}{21}\end{turn} & \begin{turn}{80}{42}\end{turn} & \begin{turn}{80}{68}\end{turn} & \begin{turn}{80}{85}\end{turn} & \begin{turn}{80}{85}\end{turn} & \begin{turn}{80}{68}\end{turn} & \begin{turn}{80}{42}\end{turn} & \begin{turn}{80}{21}\end{turn} & \begin{turn}{80}{7}\end{turn} & \begin{turn}{80}{2}\end{turn} & \begin{turn}{80}{0}\end{turn} & \begin{turn}{80}{0}\end{turn} & \begin{turn}{80}{0}\end{turn} & \begin{turn}{80}{0}\end{turn} & \begin{turn}{80}{0}\end{turn} & \begin{turn}{80}{0}\end{turn} & \begin{turn}{80}{0}\end{turn} & \begin{turn}{80}{0}\end{turn} & \begin{turn}{80}{0}\end{turn} \\ 
\hline
 17 & \begin{turn}{80}{0}\end{turn} &  \begin{turn}{80}{0}\end{turn} & \begin{turn}{80}{1}\end{turn} & \begin{turn}{80}{3}\end{turn} & \begin{turn}{80}{12}\end{turn} & \begin{turn}{80}{28}\end{turn} & \begin{turn}{80}{66}\end{turn} & \begin{turn}{80}{104}\end{turn} & \begin{turn}{80}{150}\end{turn} & \begin{turn}{80}{160}\end{turn} & \begin{turn}{80}{150}\end{turn} & \begin{turn}{80}{104}\end{turn} & \begin{turn}{80}{66}\end{turn} & \begin{turn}{80}{28}\end{turn} & \begin{turn}{80}{12}\end{turn} & \begin{turn}{80}{3}\end{turn} & \begin{turn}{80}{1}\end{turn} & \begin{turn}{80}{0}\end{turn} & \begin{turn}{80}{0}\end{turn} & \begin{turn}{80}{0}\end{turn} & \begin{turn}{80}{0}\end{turn} & \begin{turn}{80}{0}\end{turn} & \begin{turn}{80}{0}\end{turn} & \begin{turn}{80}{0}\end{turn} \\ \hline
 18 & \begin{turn}{80}{0}\end{turn} &  \begin{turn}{80}{0}\end{turn} & \begin{turn}{80}{0}\end{turn} & \begin{turn}{80}{2}\end{turn} & \begin{turn}{80}{9}\end{turn} & \begin{turn}{80}{32}\end{turn} & \begin{turn}{80}{74}\end{turn} & \begin{turn}{80}{142}\end{turn} & \begin{turn}{80}{214}\end{turn} & \begin{turn}{80}{262}\end{turn} & \begin{turn}{80}{262}\end{turn} & \begin{turn}{80}{214}\end{turn} & \begin{turn}{80}{142}\end{turn} & \begin{turn}{80}{74}\end{turn} & \begin{turn}{80}{32}\end{turn} & \begin{turn}{80}{9}\end{turn} & \begin{turn}{80}{2}\end{turn} & \begin{turn}{80}{0}\end{turn} & \begin{turn}{80}{0}\end{turn} & \begin{turn}{80}{0}\end{turn} & \begin{turn}{80}{0}\end{turn} & \begin{turn}{80}{0}\end{turn} & \begin{turn}{80}{0}\end{turn} & \begin{turn}{80}{0}\end{turn}  \\ \hline
 19 & \begin{turn}{80}{0}\end{turn} &   \begin{turn}{80}{0}\end{turn} & \begin{turn}{80}{1}\end{turn} & \begin{turn}{80}{3}\end{turn} & \begin{turn}{80}{15}\end{turn} & \begin{turn}{80}{41}\end{turn} & \begin{turn}{80}{108}\end{turn} &  \begin{turn}{80}{204}\end{turn}& \begin{turn}{80}{342}\end{turn} & \begin{turn}{80}{442}\end{turn} & \begin{turn}{80}{499}\end{turn} & \begin{turn}{80}{442}\end{turn} & \begin{turn}{80}{342}\end{turn} & \begin{turn}{80}{204}\end{turn} & \begin{turn}{80}{108}\end{turn} & \begin{turn}{80}{41}\end{turn} & \begin{turn}{80}{15}\end{turn} & \begin{turn}{80}{3}\end{turn} & \begin{turn}{80}{1}\end{turn} & \begin{turn}{80}{0}\end{turn} & \begin{turn}{80}{0}\end{turn} & \begin{turn}{80}{0}\end{turn} & \begin{turn}{80}{0}\end{turn} & \begin{turn}{80}{0}\end{turn}  \\ \hline
 20 & \begin{turn}{80}{0}\end{turn} & \begin{turn}{80}{0}\end{turn} & \begin{turn}{80}{0}\end{turn} & \begin{turn}{80}{3}\end{turn} & \begin{turn}{80}{12}\end{turn} & \begin{turn}{80}{46}\end{turn} & \begin{turn}{80}{124}\end{turn} &\begin{turn}{80}{271}\end{turn} & \begin{turn}{80}{474}\end{turn} &\begin{turn}{80}{691}\end{turn} & \begin{turn}{80}{827}\end{turn}&\begin{turn}{80}{827}\end{turn} &\begin{turn}{80}{691}\end{turn} & \begin{turn}{80}{474}\end{turn} & \begin{turn}{80}{271}\end{turn} & \begin{turn}{80}{124}\end{turn} & \begin{turn}{80}{46}\end{turn} & \begin{turn}{80}{12}\end{turn} & \begin{turn}{80}{3}\end{turn} & \begin{turn}{80}{0}\end{turn} & \begin{turn}{80}{0}\end{turn} & \begin{turn}{80}{0}\end{turn} & \begin{turn}{80}{0}\end{turn} & \begin{turn}{80}{0}\end{turn}  \\ \hline
 21  & \begin{turn}{80}{0}\end{turn} & \begin{turn}{80}{0}\end{turn} & \begin{turn}{80}{1}\end{turn}& \begin{turn}{80}{3}\end{turn} & \begin{turn}{80}{18}\end{turn} & \begin{turn}{80}{57}\end{turn} &  \begin{turn}{80}{168}\end{turn} & \begin{turn}{80}{368}\end{turn} & \begin{turn}{80}{707}\end{turn} & \begin{turn}{80}{1075}\end{turn} & \begin{turn}{80}{1419}\end{turn}& \begin{turn}{80}{1527}\end{turn} & \begin{turn}{80}{1419}\end{turn}& \begin{turn}{80}{1075}\end{turn} & \begin{turn}{80}{707}\end{turn} & \begin{turn}{80}{368}\end{turn} & \begin{turn}{80}{168}\end{turn} & \begin{turn}{80}{57}\end{turn} & \begin{turn}{80}{18}\end{turn} & \begin{turn}{80}{3}\end{turn} & \begin{turn}{80}{1}\end{turn} & \begin{turn}{80}{0}\end{turn} & \begin{turn}{80}{0}\end{turn} & \begin{turn}{80}{0}\end{turn}  \\ \hline
 22 & \begin{turn}{80}{0}\end{turn} &  \begin{turn}{80}{0}\end{turn} & \begin{turn}{80}{0}\end{turn} & \begin{turn}{80}{3}\end{turn} & \begin{turn}{80}{15}\end{turn} & \begin{turn}{80}{64}\end{turn} & \begin{turn}{80}{192}\end{turn} & \begin{turn}{80}{477}\end{turn}& \begin{turn}{80}{954}\end{turn}& \begin{turn}{80}{1600}\end{turn} & \begin{turn}{80}{2240}\end{turn}& \begin{turn}{80}{2651}\end{turn}& \begin{turn}{80}{2651}\end{turn} & \begin{turn}{80}{2240}\end{turn} & \begin{turn}{80}{1600}\end{turn} & \begin{turn}{80}{954}\end{turn}& \begin{turn}{80}{477}\end{turn}& \begin{turn}{80}{192}\end{turn} & \begin{turn}{80}{64}\end{turn} & \begin{turn}{80}{15}\end{turn} & \begin{turn}{80}{3}\end{turn} & \begin{turn}{80}{0}\end{turn} & \begin{turn}{80}{0}\end{turn} & \begin{turn}{80}{0}\end{turn} \\ \hline
 23 & \begin{turn}{80}{0}\end{turn} &  \begin{turn}{80}{0}\end{turn} & \begin{turn}{80}{1}\end{turn} & \begin{turn}{80}{4}\end{turn} & \begin{turn}{80}{22}\end{turn}& \begin{turn}{80}{77}\end{turn} &  \begin{turn}{80}{254}\end{turn} & \begin{turn}{80}{627}\end{turn} & \begin{turn}{80}{1353}\end{turn}& \begin{turn}{80}{2371}\end{turn} & \begin{turn}{80}{3586}\end{turn}& \begin{turn}{80}{4522}\end{turn} &  \begin{turn}{80}{4940}\end{turn}&  \begin{turn}{80}{4522}\end{turn} &  \begin{turn}{80}{3586}\end{turn} & \begin{turn}{80}{2371}\end{turn}&  \begin{turn}{80}{1353}\end{turn} & \begin{turn}{80}{627}\end{turn} & \begin{turn}{80}{254}\end{turn} & \begin{turn}{80}{77}\end{turn} & \begin{turn}{80}{22}\end{turn} & \begin{turn}{80}{4}\end{turn} & \begin{turn}{80}{1}\end{turn} & \begin{turn}{80}{0}\end{turn} \\
\hline
\end{tabular}
\end{center}}
\caption{\small Table of Euler characteristics $\chi_{s_1,s_2,t}^{\pi g}$ by genus $g=0$, complexity $t$ and Hodge degree $s_2$ of $\pi_*\overline{\mbox{Emb}}_c(\coprod_{i=1}^2 \rbb^{m_i}, \rdbb) \otimes \qbb$ for $m_1$, $m_2$ and $d$ odd $(s_1=t-s_2+1)$.}
\end{table}

\begin{table}[ht!]

{\tiny \begin{center}
\begin{tabular}{|c|c|c|c|c|c|c|c|c|c|c|c|c|c|c|c|c|c|c|c|c|c|c|c|c|}
\hline
$t$ &\multicolumn{23}{|c|}{Hodge degree $s_2$}& \\ \hline
  & 0 & 1 & 2 & 3 & 4 & 5 & 6 & 7 & 8 & 9 & 10 & 11 & 12 & 13 & 14 & 15 & 16 & 17 & 18 & 19 & 20 & 21 & 22 & 23 \\ \hline
 1 & \begin{turn}{80}{0}\end{turn} & \begin{turn}{80}{0}\end{turn} & \begin{turn}{80}{0}\end{turn} & \begin{turn}{80}{0}\end{turn} & \begin{turn}{80}{0}\end{turn} & \begin{turn}{80}{0}\end{turn} & \begin{turn}{80}{0}\end{turn} & \begin{turn}{80}{0}\end{turn} & \begin{turn}{80}{0}\end{turn} & \begin{turn}{80}{0}\end{turn} & \begin{turn}{80}{0}\end{turn} & \begin{turn}{80}{0}\end{turn} & \begin{turn}{80}{0}\end{turn} & \begin{turn}{80}{0}\end{turn} & \begin{turn}{80}{0}\end{turn} & \begin{turn}{80}{0}\end{turn} & \begin{turn}{80}{0}\end{turn} & \begin{turn}{80}{0}\end{turn} & \begin{turn}{80}{0}\end{turn} & \begin{turn}{80}{0}\end{turn} & \begin{turn}{80}{0}\end{turn} & \begin{turn}{80}{0}\end{turn} & \begin{turn}{80}{0}\end{turn} & \begin{turn}{80}{0}\end{turn}   \\  
   \hline
 2 & \begin{turn}{80}{1}\end{turn} & \begin{turn}{80}{1}\end{turn} & \begin{turn}{80}{1}\end{turn} & \begin{turn}{80}{0}\end{turn} & \begin{turn}{80}{0}\end{turn} & \begin{turn}{80}{0}\end{turn} & \begin{turn}{80}{0}\end{turn} & \begin{turn}{80}{0}\end{turn} & \begin{turn}{80}{0}\end{turn} & \begin{turn}{80}{0}\end{turn} & \begin{turn}{80}{0}\end{turn} & \begin{turn}{80}{0}\end{turn} & \begin{turn}{80}{0}\end{turn} & \begin{turn}{80}{0}\end{turn} & \begin{turn}{80}{0}\end{turn} & \begin{turn}{80}{0}\end{turn} & \begin{turn}{80}{0}\end{turn} & \begin{turn}{80}{0}\end{turn} & \begin{turn}{80}{0}\end{turn} & \begin{turn}{80}{0}\end{turn} & \begin{turn}{80}{0}\end{turn} & \begin{turn}{80}{0}\end{turn} & \begin{turn}{80}{0}\end{turn} & \begin{turn}{80}{0}\end{turn}  \\  
\hline
 3 & \begin{turn}{80}{0}\end{turn} & \begin{turn}{80}{0}\end{turn} & \begin{turn}{80}{0}\end{turn} & \begin{turn}{80}{0}\end{turn} & \begin{turn}{80}{0}\end{turn} & \begin{turn}{80}{0}\end{turn} & \begin{turn}{80}{0}\end{turn} & \begin{turn}{80}{0}\end{turn} & \begin{turn}{80}{0}\end{turn} & \begin{turn}{80}{0}\end{turn} & \begin{turn}{80}{0}\end{turn} & \begin{turn}{80}{0}\end{turn} & \begin{turn}{80}{0}\end{turn} & \begin{turn}{80}{0}\end{turn} & \begin{turn}{80}{0}\end{turn} & \begin{turn}{80}{0}\end{turn} & \begin{turn}{80}{0}\end{turn} & \begin{turn}{80}{0}\end{turn} & \begin{turn}{80}{0}\end{turn} & \begin{turn}{80}{0}\end{turn} & \begin{turn}{80}{0}\end{turn} & \begin{turn}{80}{0}\end{turn} & \begin{turn}{80}{0}\end{turn} & \begin{turn}{80}{0}\end{turn}  \\  
   \hline
 4 & \begin{turn}{80}{1}\end{turn} & \begin{turn}{80}{1}\end{turn}  & \begin{turn}{80}{2}\end{turn}  & \begin{turn}{80}{1}\end{turn} & \begin{turn}{80}{1}\end{turn} & \begin{turn}{80}{0}\end{turn} & \begin{turn}{80}{0}\end{turn} & \begin{turn}{80}{0}\end{turn} & \begin{turn}{80}{0}\end{turn} & \begin{turn}{80}{0}\end{turn} & \begin{turn}{80}{0}\end{turn} & \begin{turn}{80}{0}\end{turn} & \begin{turn}{80}{0}\end{turn} & \begin{turn}{80}{0}\end{turn} & \begin{turn}{80}{0}\end{turn} & \begin{turn}{80}{0}\end{turn} & \begin{turn}{80}{0}\end{turn} & \begin{turn}{80}{0}\end{turn} & \begin{turn}{80}{0}\end{turn} & \begin{turn}{80}{0}\end{turn} & \begin{turn}{80}{0}\end{turn} & \begin{turn}{80}{0}\end{turn} & \begin{turn}{80}{0}\end{turn} & \begin{turn}{80}{0}\end{turn}  \\  
   \hline
 5 & \begin{turn}{80}{0}\end{turn} & \begin{turn}{80}{0}\end{turn} & \begin{turn}{80}{0}\end{turn} & \begin{turn}{80}{0}\end{turn} & \begin{turn}{80}{0}\end{turn} & \begin{turn}{80}{0}\end{turn} & \begin{turn}{80}{0}\end{turn} & \begin{turn}{80}{0}\end{turn} & \begin{turn}{80}{0}\end{turn} & \begin{turn}{80}{0}\end{turn} & \begin{turn}{80}{0}\end{turn} & \begin{turn}{80}{0}\end{turn} & \begin{turn}{80}{0}\end{turn} & \begin{turn}{80}{0}\end{turn} & \begin{turn}{80}{0}\end{turn} & \begin{turn}{80}{0}\end{turn} & \begin{turn}{80}{0}\end{turn} & \begin{turn}{80}{0}\end{turn} & \begin{turn}{80}{0}\end{turn} & \begin{turn}{80}{0}\end{turn} & \begin{turn}{80}{0}\end{turn} & \begin{turn}{80}{0}\end{turn} & \begin{turn}{80}{0}\end{turn} & \begin{turn}{80}{0}\end{turn}  \\ 
    \hline
 6 & \begin{turn}{80}{1}\end{turn} & \begin{turn}{80}{1}\end{turn} & \begin{turn}{80}{3}\end{turn} & \begin{turn}{80}{3}\end{turn} & \begin{turn}{80}{3}\end{turn} & \begin{turn}{80}{1}\end{turn} & \begin{turn}{80}{1}\end{turn} & \begin{turn}{80}{0}\end{turn} & \begin{turn}{80}{0}\end{turn} & \begin{turn}{80}{0}\end{turn} & \begin{turn}{80}{0}\end{turn} & \begin{turn}{80}{0}\end{turn} & \begin{turn}{80}{0}\end{turn} & \begin{turn}{80}{0}\end{turn} & \begin{turn}{80}{0}\end{turn} & \begin{turn}{80}{0}\end{turn} & \begin{turn}{80}{0}\end{turn} & \begin{turn}{80}{0}\end{turn} & \begin{turn}{80}{0}\end{turn} & \begin{turn}{80}{0}\end{turn} & \begin{turn}{80}{0}\end{turn} & \begin{turn}{80}{0}\end{turn} & \begin{turn}{80}{0}\end{turn} & \begin{turn}{80}{0}\end{turn} \\  
   \hline
 7 & \begin{turn}{80}{0}\end{turn} & \begin{turn}{80}{0}\end{turn} & \begin{turn}{80}{0}\end{turn} & \begin{turn}{80}{1}\end{turn} & \begin{turn}{80}{1}\end{turn} & \begin{turn}{80}{0}\end{turn} & \begin{turn}{80}{0}\end{turn} & \begin{turn}{80}{0}\end{turn} & \begin{turn}{80}{0}\end{turn} & \begin{turn}{80}{0}\end{turn} & \begin{turn}{80}{0}\end{turn} & \begin{turn}{80}{0}\end{turn} & \begin{turn}{80}{0}\end{turn} & \begin{turn}{80}{0}\end{turn} & \begin{turn}{80}{0}\end{turn} & \begin{turn}{80}{0}\end{turn} & \begin{turn}{80}{0}\end{turn} & \begin{turn}{80}{0}\end{turn} & \begin{turn}{80}{0}\end{turn} & \begin{turn}{80}{0}\end{turn} & \begin{turn}{80}{0}\end{turn} & \begin{turn}{80}{0}\end{turn} & \begin{turn}{80}{0}\end{turn} & \begin{turn}{80}{0}\end{turn} \\  
   \hline
 8 & \begin{turn}{80}{1}\end{turn} & \begin{turn}{80}{1}\end{turn} & \begin{turn}{80}{4}\end{turn} & \begin{turn}{80}{5}\end{turn} & \begin{turn}{80}{8}\end{turn} & \begin{turn}{80}{5}\end{turn} & \begin{turn}{80}{4}\end{turn} & \begin{turn}{80}{1}\end{turn} & \begin{turn}{80}{1}\end{turn} & \begin{turn}{80}{0}\end{turn} & \begin{turn}{80}{0}\end{turn} & \begin{turn}{80}{0}\end{turn} & \begin{turn}{80}{0}\end{turn} & \begin{turn}{80}{0}\end{turn} & \begin{turn}{80}{0}\end{turn} & \begin{turn}{80}{0}\end{turn} & \begin{turn}{80}{0}\end{turn} & \begin{turn}{80}{0}\end{turn} & \begin{turn}{80}{0}\end{turn} & \begin{turn}{80}{0}\end{turn} & \begin{turn}{80}{0}\end{turn} & \begin{turn}{80}{0}\end{turn} & \begin{turn}{80}{0}\end{turn} & \begin{turn}{80}{0}\end{turn} \\  
     \hline
 9 & \begin{turn}{80}{0}\end{turn} & \begin{turn}{80}{0}\end{turn} & \begin{turn}{80}{0}\end{turn} & \begin{turn}{80}{3}\end{turn} & \begin{turn}{80}{4}\end{turn} & \begin{turn}{80}{4}\end{turn} & \begin{turn}{80}{3}\end{turn} & \begin{turn}{80}{0}\end{turn} & \begin{turn}{80}{0}\end{turn} & \begin{turn}{80}{0}\end{turn} & \begin{turn}{80}{0}\end{turn} & \begin{turn}{80}{0}\end{turn} & \begin{turn}{80}{0}\end{turn} & \begin{turn}{80}{0}\end{turn} & \begin{turn}{80}{0}\end{turn} & \begin{turn}{80}{0}\end{turn} & \begin{turn}{80}{0}\end{turn} & \begin{turn}{80}{0}\end{turn} & \begin{turn}{80}{0}\end{turn} & \begin{turn}{80}{0}\end{turn} & \begin{turn}{80}{0}\end{turn} & \begin{turn}{80}{0}\end{turn} & \begin{turn}{80}{0}\end{turn} & \begin{turn}{80}{0}\end{turn} \\  
    \hline
 10 & \begin{turn}{80}{1}\end{turn} &\begin{turn}{80}{1}\end{turn} & \begin{turn}{80}{5}\end{turn} & \begin{turn}{80}{8}\end{turn} & \begin{turn}{80}{16}\end{turn} & \begin{turn}{80}{16}\end{turn} & \begin{turn}{80}{16}\end{turn} & \begin{turn}{80}{8}\end{turn} & \begin{turn}{80}{5}\end{turn} & \begin{turn}{80}{1}\end{turn} & \begin{turn}{80}{1}\end{turn} & \begin{turn}{80}{0}\end{turn} & \begin{turn}{80}{0}\end{turn} & \begin{turn}{80}{0}\end{turn} & \begin{turn}{80}{0}\end{turn} & \begin{turn}{80}{0}\end{turn} & \begin{turn}{80}{0}\end{turn} & \begin{turn}{80}{0}\end{turn} & \begin{turn}{80}{0}\end{turn} & \begin{turn}{80}{0}\end{turn} & \begin{turn}{80}{0}\end{turn} & \begin{turn}{80}{0}\end{turn} & \begin{turn}{80}{0}\end{turn} & \begin{turn}{80}{0}\end{turn} \\ 
       \hline
 11 & \begin{turn}{80}{0}\end{turn} & \begin{turn}{80}{0}\end{turn} & \begin{turn}{80}{0}\end{turn} & \begin{turn}{80}{5}\end{turn} & \begin{turn}{80}{10}\end{turn} & \begin{turn}{80}{16}\end{turn} & \begin{turn}{80}{16}\end{turn} & \begin{turn}{80}{10}\end{turn} & \begin{turn}{80}{5}\end{turn} & \begin{turn}{80}{0}\end{turn} & \begin{turn}{80}{0}\end{turn} & \begin{turn}{80}{0}\end{turn} & \begin{turn}{80}{0}\end{turn} & \begin{turn}{80}{0}\end{turn} & \begin{turn}{80}{0}\end{turn} & \begin{turn}{80}{0}\end{turn} & \begin{turn}{80}{0}\end{turn} & \begin{turn}{80}{0}\end{turn} & \begin{turn}{80}{0}\end{turn} & \begin{turn}{80}{0}\end{turn} & \begin{turn}{80}{0}\end{turn} & \begin{turn}{80}{0}\end{turn} & \begin{turn}{80}{0}\end{turn} & \begin{turn}{80}{0}\end{turn} \\  
    \hline
 12 & \begin{turn}{80}{1}\end{turn} & \begin{turn}{80}{1}\end{turn} & \begin{turn}{80}{6}\end{turn} & \begin{turn}{80}{12}\end{turn} & \begin{turn}{80}{29}\end{turn} & \begin{turn}{80}{38}\end{turn} & \begin{turn}{80}{50}\end{turn} & \begin{turn}{80}{38}\end{turn} & \begin{turn}{80}{29}\end{turn} & \begin{turn}{80}{12}\end{turn} & \begin{turn}{80}{6}\end{turn} & \begin{turn}{80}{1}\end{turn} & \begin{turn}{80}{1}\end{turn} & \begin{turn}{80}{0}\end{turn} & \begin{turn}{80}{0}\end{turn} & \begin{turn}{80}{0}\end{turn} & \begin{turn}{80}{0}\end{turn} & \begin{turn}{80}{0}\end{turn} & \begin{turn}{80}{0}\end{turn} & \begin{turn}{80}{0}\end{turn} & \begin{turn}{80}{0}\end{turn} & \begin{turn}{80}{0}\end{turn} & \begin{turn}{80}{0}\end{turn} & \begin{turn}{80}{0}\end{turn} \\       \hline
 13 & \begin{turn}{80}{0}\end{turn} & \begin{turn}{80}{0}\end{turn} & \begin{turn}{80}{0}\end{turn} & \begin{turn}{80}{8}\end{turn} & \begin{turn}{80}{20}\end{turn} & \begin{turn}{80}{42}\end{turn} & \begin{turn}{80}{56}\end{turn} & \begin{turn}{80}{56}\end{turn} & \begin{turn}{80}{42}\end{turn} & \begin{turn}{80}{20}\end{turn} & \begin{turn}{80}{8}\end{turn} & \begin{turn}{80}{0}\end{turn} & \begin{turn}{80}{0}\end{turn} & \begin{turn}{80}{0}\end{turn} & \begin{turn}{80}{0}\end{turn} & \begin{turn}{80}{0}\end{turn} & \begin{turn}{80}{0}\end{turn} & \begin{turn}{80}{0}\end{turn} & \begin{turn}{80}{0}\end{turn} & \begin{turn}{80}{0}\end{turn} & \begin{turn}{80}{0}\end{turn} & \begin{turn}{80}{0}\end{turn} & \begin{turn}{80}{0}\end{turn} & \begin{turn}{80}{0}\end{turn} \\     \hline
 14 & \begin{turn}{80}{1}\end{turn} & \begin{turn}{80}{1}\end{turn} & \begin{turn}{80}{7}\end{turn} & \begin{turn}{80}{16}\end{turn} & \begin{turn}{80}{47}\end{turn} & \begin{turn}{80}{79}\end{turn} & \begin{turn}{80}{126}\end{turn} & \begin{turn}{80}{133}\end{turn} & \begin{turn}{80}{126}\end{turn} & \begin{turn}{80}{79}\end{turn} & \begin{turn}{80}{47}\end{turn} & \begin{turn}{80}{16}\end{turn} & \begin{turn}{80}{7}\end{turn} & \begin{turn}{80}{1}\end{turn} & \begin{turn}{80}{1}\end{turn} & \begin{turn}{80}{0}\end{turn} & \begin{turn}{80}{0}\end{turn} & \begin{turn}{80}{0}\end{turn} & \begin{turn}{80}{0}\end{turn} & \begin{turn}{80}{0}\end{turn} & \begin{turn}{80}{0}\end{turn} & \begin{turn}{80}{0}\end{turn} & \begin{turn}{80}{0}\end{turn} & \begin{turn}{80}{0}\end{turn} \\  \hline
 15 & \begin{turn}{80}{0}\end{turn} & \begin{turn}{80}{0}\end{turn} & \begin{turn}{80}{0}\end{turn} & \begin{turn}{80}{12}\end{turn} & \begin{turn}{80}{35}\end{turn} & \begin{turn}{80}{90}\end{turn} & \begin{turn}{80}{150}\end{turn} & \begin{turn}{80}{197}\end{turn} & \begin{turn}{80}{197}\end{turn} & \begin{turn}{80}{150}\end{turn} & \begin{turn}{80}{90}\end{turn} & \begin{turn}{80}{35}\end{turn} & \begin{turn}{80}{12}\end{turn} & \begin{turn}{80}{0}\end{turn} & \begin{turn}{80}{0}\end{turn} & \begin{turn}{80}{0}\end{turn} & \begin{turn}{80}{0}\end{turn} & \begin{turn}{80}{0}\end{turn} & \begin{turn}{80}{0}\end{turn} & \begin{turn}{80}{0}\end{turn} & \begin{turn}{80}{0}\end{turn} & \begin{turn}{80}{0}\end{turn} & \begin{turn}{80}{0}\end{turn} & \begin{turn}{80}{0}\end{turn} \\  \hline
 16 & \begin{turn}{80}{1}\end{turn} & \begin{turn}{80}{1}\end{turn} & \begin{turn}{80}{8}\end{turn} & \begin{turn}{80}{21}\end{turn} & \begin{turn}{80}{72}\end{turn} & \begin{turn}{80}{147}\end{turn} & \begin{turn}{80}{280}\end{turn} & \begin{turn}{80}{375}\end{turn} & \begin{turn}{80}{440}\end{turn} & \begin{turn}{80}{375}\end{turn} & \begin{turn}{80}{280}\end{turn} & \begin{turn}{80}{147}\end{turn} & \begin{turn}{80}{72}\end{turn} & \begin{turn}{80}{21}\end{turn} & \begin{turn}{80}{8}\end{turn} & \begin{turn}{80}{1}\end{turn} & \begin{turn}{80}{1}\end{turn} & \begin{turn}{80}{0}\end{turn} & \begin{turn}{80}{0}\end{turn} & \begin{turn}{80}{0}\end{turn} & \begin{turn}{80}{0}\end{turn} & \begin{turn}{80}{0}\end{turn} & \begin{turn}{80}{0}\end{turn} & \begin{turn}{80}{0}\end{turn} \\  \hline
 17 & \begin{turn}{80}{0}\end{turn} & \begin{turn}{80}{0}\end{turn} & \begin{turn}{80}{0}\end{turn} & \begin{turn}{80}{16}\end{turn} & \begin{turn}{80}{56}\end{turn} & \begin{turn}{80}{168}\end{turn} & \begin{turn}{80}{336}\end{turn} & \begin{turn}{80}{544}\end{turn} & \begin{turn}{80}{680}\end{turn} & \begin{turn}{80}{680}\end{turn} & \begin{turn}{80}{544}\end{turn} & \begin{turn}{80}{336}\end{turn} & \begin{turn}{80}{168}\end{turn} & \begin{turn}{80}{56}\end{turn} & \begin{turn}{80}{16}\end{turn} & \begin{turn}{80}{0}\end{turn} & \begin{turn}{80}{0}\end{turn} & \begin{turn}{80}{0}\end{turn} & \begin{turn}{80}{0}\end{turn} & \begin{turn}{80}{0}\end{turn} & \begin{turn}{80}{0}\end{turn} & \begin{turn}{80}{0}\end{turn} & \begin{turn}{80}{0}\end{turn} & \begin{turn}{80}{0}\end{turn}  \\  \hline
 18 & \begin{turn}{80}{1}\end{turn} & \begin{turn}{80}{1}\end{turn} & \begin{turn}{80}{9}\end{turn} & \begin{turn}{80}{27}\end{turn} & \begin{turn}{80}{104}\end{turn} & \begin{turn}{80}{252}\end{turn} & \begin{turn}{80}{561}\end{turn}&  \begin{turn}{80}{912}\end{turn}& \begin{turn}{80}{1282}\end{turn}& \begin{turn}{80}{1387}\end{turn} & \begin{turn}{80}{1282}\end{turn}& \begin{turn}{80}{912}\end{turn} & \begin{turn}{80}{561}\end{turn} & \begin{turn}{80}{252}\end{turn} & \begin{turn}{80}{104}\end{turn} & \begin{turn}{80}{27}\end{turn} & \begin{turn}{80}{9}\end{turn} & \begin{turn}{80}{1}\end{turn} & \begin{turn}{80}{1}\end{turn} & \begin{turn}{80}{0}\end{turn} & \begin{turn}{80}{0}\end{turn} & \begin{turn}{80}{0}\end{turn} & \begin{turn}{80}{0}\end{turn} & \begin{turn}{80}{0}\end{turn} \\  \hline
 19 & \begin{turn}{80}{0}\end{turn} & \begin{turn}{80}{0}\end{turn} & \begin{turn}{80}{0}\end{turn} & \begin{turn}{80}{21}\end{turn} & \begin{turn}{80}{84}\end{turn} & \begin{turn}{80}{288}\end{turn}& \begin{turn}{80}{672}\end{turn} & \begin{turn}{80}{1284}\end{turn} & \begin{turn}{80}{1926}\end{turn} & \begin{turn}{80}{2368}\end{turn} & \begin{turn}{80}{2368}\end{turn} & \begin{turn}{80}{1926}\end{turn} & \begin{turn}{80}{1284}\end{turn} & \begin{turn}{80}{672}\end{turn} & \begin{turn}{80}{288}\end{turn} & \begin{turn}{80}{84}\end{turn} & \begin{turn}{80}{21}\end{turn} & \begin{turn}{80}{0}\end{turn} & \begin{turn}{80}{0}\end{turn} & \begin{turn}{80}{0}\end{turn} & \begin{turn}{80}{0}\end{turn} & \begin{turn}{80}{0}\end{turn} & \begin{turn}{80}{0}\end{turn} & \begin{turn}{80}{0}\end{turn} \\  \hline
 20 & \begin{turn}{80}{1}\end{turn} & \begin{turn}{80}{1}\end{turn} & \begin{turn}{80}{10}\end{turn} &  \begin{turn}{80}{33}\end{turn}&  \begin{turn}{80}{145}\end{turn} &  \begin{turn}{80}{406}\end{turn}& \begin{turn}{80}{1032}\end{turn} &  \begin{turn}{80}{1980}\end{turn}&  \begin{turn}{80}{3260}\end{turn} & \begin{turn}{80}{4262}\end{turn} &  \begin{turn}{80}{4752}\end{turn} &  \begin{turn}{80}{4262}\end{turn}& \begin{turn}{80}{3260}\end{turn}& \begin{turn}{80}{1980}\end{turn} & \begin{turn}{80}{1032}\end{turn} & \begin{turn}{80}{406}\end{turn} & \begin{turn}{80}{145}\end{turn} & \begin{turn}{80}{33}\end{turn} & \begin{turn}{80}{10}\end{turn} & \begin{turn}{80}{1}\end{turn} & \begin{turn}{80}{1}\end{turn} & \begin{turn}{80}{0}\end{turn} & \begin{turn}{80}{0}\end{turn} & \begin{turn}{80}{0}\end{turn}  \\  \hline
 21 & \begin{turn}{80}{0}\end{turn} &  \begin{turn}{80}{0}\end{turn} & \begin{turn}{80}{0}\end{turn}&  \begin{turn}{80}{27}\end{turn}&  \begin{turn}{80}{120}\end{turn} &  \begin{turn}{80}{462}\end{turn} &  \begin{turn}{80}{1233}\end{turn}&  \begin{turn}{80}{2709}\end{turn}&  \begin{turn}{80}{4740}\end{turn}& \begin{turn}{80}{6895}\end{turn}&  \begin{turn}{80}{8272}\end{turn} &  \begin{turn}{80}{8272}\end{turn}&\begin{turn}{80}{6895}\end{turn}& \begin{turn}{80}{4740}\end{turn}&\begin{turn}{80}{2709}\end{turn}& \begin{turn}{80}{1233}\end{turn} & \begin{turn}{80}{462}\end{turn} & \begin{turn}{80}{120}\end{turn} & \begin{turn}{80}{27}\end{turn} & \begin{turn}{80}{0}\end{turn} & \begin{turn}{80}{0}\end{turn} & \begin{turn}{80}{0}\end{turn} & \begin{turn}{80}{0}\end{turn} & \begin{turn}{80}{0}\end{turn}  \\  
\hline
 22 & \begin{turn}{80}{1}\end{turn} &\begin{turn}{80}{1}\end{turn}& \begin{turn}{80}{11}\end{turn}& \begin{turn}{80}{40}\end{turn}&  \begin{turn}{80}{195}\end{turn}&  \begin{turn}{80}{621}\end{turn} & \begin{turn}{80}{1782}\end{turn}& \begin{turn}{80}{3936}\end{turn}&  \begin{turn}{80}{7440}\end{turn}&  \begin{turn}{80}{11410}\end{turn}& \begin{turn}{80}{14938}\end{turn}&  \begin{turn}{80}{16159}\end{turn} &  \begin{turn}{80}{14938}\end{turn}&  \begin{turn}{80}{11410}\end{turn}& \begin{turn}{80}{7440}\end{turn}&  \begin{turn}{80}{3936}\end{turn}& \begin{turn}{80}{1782}\end{turn}&  \begin{turn}{80}{621}\end{turn} & \begin{turn}{80}{195}\end{turn}& \begin{turn}{80}{40}\end{turn} & \begin{turn}{80}{11}\end{turn} & \begin{turn}{80}{1}\end{turn} & \begin{turn}{80}{1}\end{turn} & \begin{turn}{80}{0}\end{turn} \\  \hline
 23 & \begin{turn}{80}{0}\end{turn} & \begin{turn}{80}{0}\end{turn} & \begin{turn}{80}{0}\end{turn}&  \begin{turn}{80}{33}\end{turn}& \begin{turn}{80}{165}\end{turn}&  \begin{turn}{80}{704}\end{turn}&  \begin{turn}{80}{2112}\end{turn}& \begin{turn}{80}{5247}\end{turn}&  \begin{turn}{80}{10494}\end{turn}&  \begin{turn}{80}{17600}\end{turn}&  \begin{turn}{80}{24640}\end{turn}& \begin{turn}{80}{29162}\end{turn}&  \begin{turn}{80}{29162}\end{turn}& \begin{turn}{80}{24640}\end{turn}& \begin{turn}{80}{17600}\end{turn}& \begin{turn}{80}{10494}\end{turn}&  \begin{turn}{80}{5247}\end{turn}& \begin{turn}{80}{2112}\end{turn}& \begin{turn}{80}{704}\end{turn}& \begin{turn}{80}{165}\end{turn}& \begin{turn}{80}{33}\end{turn}& \begin{turn}{80}{0}\end{turn} & \begin{turn}{80}{0}\end{turn} & \begin{turn}{80}{0}\end{turn} \\  \hline
\end{tabular}
\end{center}}
\caption{\small Table of Euler characteristics $\chi_{s_1,s_2,t}^{\pi g}$ by genus $g=1$, complexity $t$ and Hodge degree $s_2$ of $\pi_*\overline{\mbox{Emb}}_c(\coprod_{i=1}^2 \rbb^{m_i}, \rdbb) \otimes \qbb$ for $m_1$, $m_2$ and $d$ odd $(s_1=t-s_2)$.}
\end{table}

\begin{table}[ht!]

{\tiny \begin{center}
\begin{tabular}{|c|c|c|c|c|c|c|c|c|c|c|c|c|c|c|c|c|c|c|c|c|c|c|c|c|}
\hline
$t$ &\multicolumn{23}{|c|}{Hodge degree $s_2$}& \\ \hline
  & 0 & 1 & 2 & 3 & 4 & 5 & 6 & 7 & 8 & 9 & 10 & 11 & 12 & 13 & 14 & 15 & 16 & 17 & 18 & 19 & 20 & 21 & 22 & 23 \\ \hline
 1 & \begin{turn}{80}{0}\end{turn} & \begin{turn}{80}{0}\end{turn} & \begin{turn}{80}{0}\end{turn} & \begin{turn}{80}{0}\end{turn} & \begin{turn}{80}{0}\end{turn} & \begin{turn}{80}{0}\end{turn} & \begin{turn}{80}{0}\end{turn} & \begin{turn}{80}{0}\end{turn} & \begin{turn}{80}{0}\end{turn} & \begin{turn}{80}{0}\end{turn} & \begin{turn}{80}{0}\end{turn} & \begin{turn}{80}{0}\end{turn} & \begin{turn}{80}{0}\end{turn} & \begin{turn}{80}{0}\end{turn} & \begin{turn}{80}{0}\end{turn} & \begin{turn}{80}{0}\end{turn} & \begin{turn}{80}{0}\end{turn} & \begin{turn}{80}{0}\end{turn} & \begin{turn}{80}{0}\end{turn} & \begin{turn}{80}{0}\end{turn} & \begin{turn}{80}{0}\end{turn} & \begin{turn}{80}{0}\end{turn} & \begin{turn}{80}{0}\end{turn} & \begin{turn}{80}{0}\end{turn}   \\ \hline
 2 & \begin{turn}{80}{-1}\end{turn} & \begin{turn}{80}{-1}\end{turn} & \begin{turn}{80}{0}\end{turn} &\begin{turn}{80}{0}\end{turn}& \begin{turn}{80}{0}\end{turn} & \begin{turn}{80}{0}\end{turn} & \begin{turn}{80}{0}\end{turn}  & \begin{turn}{80}{0}\end{turn}  & \begin{turn}{80}{0}\end{turn}  & \begin{turn}{80}{0}\end{turn}  & \begin{turn}{80}{0}\end{turn}  & \begin{turn}{80}{0}\end{turn}  & \begin{turn}{80}{0}\end{turn}  & \begin{turn}{80}{0}\end{turn}  & \begin{turn}{80}{0}\end{turn}  & \begin{turn}{80}{0}\end{turn}  & \begin{turn}{80}{0}\end{turn}  & \begin{turn}{80}{0}\end{turn}  & \begin{turn}{80}{0}\end{turn}  & \begin{turn}{80}{0}\end{turn}  & \begin{turn}{80}{0}\end{turn}  & \begin{turn}{80}{0}\end{turn}  & \begin{turn}{80}{0}\end{turn}  & \begin{turn}{80}{0}\end{turn}  \\ \hline
 3 & \begin{turn}{80}{1}\end{turn} & \begin{turn}{80}{1}\end{turn}  & \begin{turn}{80}{1}\end{turn}  & \begin{turn}{80}{0}\end{turn}  & \begin{turn}{80}{0}\end{turn}  & \begin{turn}{80}{0}\end{turn}  & \begin{turn}{80}{0}\end{turn}  & \begin{turn}{80}{0}\end{turn}  & \begin{turn}{80}{0}\end{turn}  & \begin{turn}{80}{0}\end{turn}  & \begin{turn}{80}{0}\end{turn}  & \begin{turn}{80}{0}\end{turn}  & \begin{turn}{80}{0}\end{turn}  & \begin{turn}{80}{0}\end{turn}  & \begin{turn}{80}{0}\end{turn}  & \begin{turn}{80}{0}\end{turn}  & \begin{turn}{80}{0}\end{turn}  & \begin{turn}{80}{0}\end{turn}  & \begin{turn}{80}{0}\end{turn}  & \begin{turn}{80}{0}\end{turn}  & \begin{turn}{80}{0}\end{turn}  & \begin{turn}{80}{0}\end{turn}  & \begin{turn}{80}{0}\end{turn}  & \begin{turn}{80}{0}\end{turn} \\   \hline
 4 & \begin{turn}{80}{-1}\end{turn} & \begin{turn}{80}{-1}\end{turn} & \begin{turn}{80}{-1}\end{turn}  & \begin{turn}{80}{-1}\end{turn}  & \begin{turn}{80}{0}\end{turn}  & \begin{turn}{80}{0}\end{turn}  & \begin{turn}{80}{0}\end{turn}  & \begin{turn}{80}{0}\end{turn}  & \begin{turn}{80}{0}\end{turn}  & \begin{turn}{80}{0}\end{turn}  & \begin{turn}{80}{0}\end{turn}  & \begin{turn}{80}{0}\end{turn}  & \begin{turn}{80}{0}\end{turn}  & \begin{turn}{80}{0}\end{turn}  & \begin{turn}{80}{0}\end{turn}  & \begin{turn}{80}{0}\end{turn}  & \begin{turn}{80}{0}\end{turn}  & \begin{turn}{80}{0}\end{turn}  & \begin{turn}{80}{0}\end{turn}  & \begin{turn}{80}{0}\end{turn}  & \begin{turn}{80}{0}\end{turn}  & \begin{turn}{80}{0}\end{turn}  & \begin{turn}{80}{0}\end{turn}  & \begin{turn}{80}{0}\end{turn}   \\ \hline
 5 & \begin{turn}{80}{1}\end{turn} &  \begin{turn}{80}{1}\end{turn} & \begin{turn}{80}{3}\end{turn}  & \begin{turn}{80}{1}\end{turn} & \begin{turn}{80}{1}\end{turn}  & \begin{turn}{80}{0}\end{turn}  & \begin{turn}{80}{0}\end{turn}  & \begin{turn}{80}{0}\end{turn}  & \begin{turn}{80}{0}\end{turn}  & \begin{turn}{80}{0}\end{turn}  & \begin{turn}{80}{0}\end{turn}  & \begin{turn}{80}{0}\end{turn}  & \begin{turn}{80}{0}\end{turn}  & \begin{turn}{80}{0}\end{turn}  & \begin{turn}{80}{0}\end{turn}  & \begin{turn}{80}{0}\end{turn}  & \begin{turn}{80}{0}\end{turn}  & \begin{turn}{80}{0}\end{turn}  & \begin{turn}{80}{0}\end{turn}  & \begin{turn}{80}{0}\end{turn}  & \begin{turn}{80}{0}\end{turn}  & \begin{turn}{80}{0}\end{turn}  & \begin{turn}{80}{0}\end{turn}  & \begin{turn}{80}{0}\end{turn}  \\ 
    \hline
 6 & \begin{turn}{80}{-1}\end{turn} &  \begin{turn}{80}{-1}\end{turn} &  \begin{turn}{80}{-2}\end{turn} &  \begin{turn}{80}{-2}\end{turn} & \begin{turn}{80}{-1}\end{turn}  & \begin{turn}{80}{-1}\end{turn}  & \begin{turn}{80}{0}\end{turn}  & \begin{turn}{80}{0}\end{turn}  & \begin{turn}{80}{0}\end{turn}  & \begin{turn}{80}{0}\end{turn}  & \begin{turn}{80}{0}\end{turn}  & \begin{turn}{80}{0}\end{turn}  & \begin{turn}{80}{0}\end{turn}  & \begin{turn}{80}{0}\end{turn}  & \begin{turn}{80}{0}\end{turn}  & \begin{turn}{80}{0}\end{turn}  & \begin{turn}{80}{0}\end{turn}  & \begin{turn}{80}{0}\end{turn}  & \begin{turn}{80}{0}\end{turn}  & \begin{turn}{80}{0}\end{turn}  & \begin{turn}{80}{0}\end{turn}  & \begin{turn}{80}{0}\end{turn}  & \begin{turn}{80}{0}\end{turn}  & \begin{turn}{80}{0}\end{turn}   \\ \hline
 7 & \begin{turn}{80}{2}\end{turn} &  \begin{turn}{80}{2}\end{turn} &  \begin{turn}{80}{6}\end{turn} &  \begin{turn}{80}{5}\end{turn} &  \begin{turn}{80}{6}\end{turn} &  \begin{turn}{80}{2}\end{turn} & \begin{turn}{80}{2}\end{turn}  & \begin{turn}{80}{0}\end{turn}  & \begin{turn}{80}{0}\end{turn}  & \begin{turn}{80}{0}\end{turn}  & \begin{turn}{80}{0}\end{turn}  & \begin{turn}{80}{0}\end{turn}  & \begin{turn}{80}{0}\end{turn}  & \begin{turn}{80}{0}\end{turn}  & \begin{turn}{80}{0}\end{turn}  & \begin{turn}{80}{0}\end{turn}  & \begin{turn}{80}{0}\end{turn}  & \begin{turn}{80}{0}\end{turn}  & \begin{turn}{80}{0}\end{turn}  & \begin{turn}{80}{0}\end{turn}  & \begin{turn}{80}{0}\end{turn}  & \begin{turn}{80}{0}\end{turn}  & \begin{turn}{80}{0}\end{turn}  & \begin{turn}{80}{0}\end{turn}  \\ \hline
 8 & \begin{turn}{80}{-2}\end{turn} &  \begin{turn}{80}{-2}\end{turn} &  \begin{turn}{80}{-4}\end{turn} &  \begin{turn}{80}{-4}\end{turn} &  \begin{turn}{80}{-4}\end{turn} &  \begin{turn}{80}{-4}\end{turn} &  \begin{turn}{80}{-2}\end{turn} & \begin{turn}{80}{-2}\end{turn}  & \begin{turn}{80}{0}\end{turn}  & \begin{turn}{80}{0}\end{turn}  & \begin{turn}{80}{0}\end{turn}  & \begin{turn}{80}{0}\end{turn}  & \begin{turn}{80}{0}\end{turn}  & \begin{turn}{80}{0}\end{turn}  & \begin{turn}{80}{0}\end{turn}  & \begin{turn}{80}{0}\end{turn}  & \begin{turn}{80}{0}\end{turn}  & \begin{turn}{80}{0}\end{turn}  & \begin{turn}{80}{0}\end{turn}  & \begin{turn}{80}{0}\end{turn}  & \begin{turn}{80}{0}\end{turn}  & \begin{turn}{80}{0}\end{turn}  & \begin{turn}{80}{0}\end{turn}  & \begin{turn}{80}{0}\end{turn}  \\ 
   \hline
 9 & \begin{turn}{80}{2}\end{turn} &  \begin{turn}{80}{3}\end{turn} &  \begin{turn}{80}{10}\end{turn} &  \begin{turn}{80}{11}\end{turn} &  \begin{turn}{80}{18}\end{turn} &  \begin{turn}{80}{11}\end{turn} &  \begin{turn}{80}{10}\end{turn} &  \begin{turn}{80}{3}\end{turn} & \begin{turn}{80}{2}\end{turn}  & \begin{turn}{80}{0}\end{turn}  & \begin{turn}{80}{0}\end{turn}  & \begin{turn}{80}{0}\end{turn}  & \begin{turn}{80}{0}\end{turn}  & \begin{turn}{80}{0}\end{turn}  & \begin{turn}{80}{0}\end{turn}  & \begin{turn}{80}{0}\end{turn}  & \begin{turn}{80}{0}\end{turn}  & \begin{turn}{80}{0}\end{turn}  & \begin{turn}{80}{0}\end{turn}  & \begin{turn}{80}{0}\end{turn}  & \begin{turn}{80}{0}\end{turn}  & \begin{turn}{80}{0}\end{turn}  & \begin{turn}{80}{0}\end{turn}  & \begin{turn}{80}{0}\end{turn}  \\ 
   \hline
 10 & \begin{turn}{80}{-2}\end{turn} &  \begin{turn}{80}{-2}\end{turn} &  \begin{turn}{80}{-6}\end{turn} &  \begin{turn}{80}{-3}\end{turn} &  \begin{turn}{80}{-4}\end{turn} &  \begin{turn}{80}{-4}\end{turn} &  \begin{turn}{80}{-3}\end{turn} &  \begin{turn}{80}{-6}\end{turn} &  \begin{turn}{80}{-2}\end{turn} & \begin{turn}{80}{-2}\end{turn}  & \begin{turn}{80}{0}\end{turn}  & \begin{turn}{80}{0}\end{turn}  & \begin{turn}{80}{0}\end{turn}  & \begin{turn}{80}{0}\end{turn}  & \begin{turn}{80}{0}\end{turn}  & \begin{turn}{80}{0}\end{turn}  & \begin{turn}{80}{0}\end{turn}  & \begin{turn}{80}{0}\end{turn}  & \begin{turn}{80}{0}\end{turn}  & \begin{turn}{80}{0}\end{turn}  & \begin{turn}{80}{0}\end{turn}  & \begin{turn}{80}{0}\end{turn}  & \begin{turn}{80}{0}\end{turn}  & \begin{turn}{80}{0}\end{turn}   \\ 
      \hline
 11 & \begin{turn}{80}{2}\end{turn} &  \begin{turn}{80}{3}\end{turn} &  \begin{turn}{80}{15}\end{turn} &  \begin{turn}{80}{19}\end{turn} &  \begin{turn}{80}{39}\end{turn} &  \begin{turn}{80}{36}\end{turn} &  \begin{turn}{80}{39}\end{turn} &  \begin{turn}{80}{19}\end{turn} &  \begin{turn}{80}{15}\end{turn} &  \begin{turn}{80}{3}\end{turn} & \begin{turn}{80}{2}\end{turn}  & \begin{turn}{80}{0}\end{turn}  & \begin{turn}{80}{0}\end{turn}  & \begin{turn}{80}{0}\end{turn}  & \begin{turn}{80}{0}\end{turn}  & \begin{turn}{80}{0}\end{turn}  & \begin{turn}{80}{0}\end{turn}  & \begin{turn}{80}{0}\end{turn}  & \begin{turn}{80}{0}\end{turn}  & \begin{turn}{80}{0}\end{turn}  & \begin{turn}{80}{0}\end{turn}  & \begin{turn}{80}{0}\end{turn}  & \begin{turn}{80}{0}\end{turn}  & \begin{turn}{80}{0}\end{turn}  \\ 
     \hline
 12 & \begin{turn}{80}{-2}\end{turn} &  \begin{turn}{80}{-2}\end{turn} &  \begin{turn}{80}{-8}\end{turn} &  \begin{turn}{80}{-1}\end{turn} &  \begin{turn}{80}{1}\end{turn} &  \begin{turn}{80}{10}\end{turn} &  \begin{turn}{80}{10}\end{turn} &  \begin{turn}{80}{1}\end{turn} &  \begin{turn}{80}{-1}\end{turn} &  \begin{turn}{80}{-8}\end{turn} &  \begin{turn}{80}{-2}\end{turn} & \begin{turn}{80}{-2}\end{turn}  & \begin{turn}{80}{0}\end{turn}  & \begin{turn}{80}{0}\end{turn}  & \begin{turn}{80}{0}\end{turn}  & \begin{turn}{80}{0}\end{turn}  & \begin{turn}{80}{0}\end{turn}  & \begin{turn}{80}{0}\end{turn}  & \begin{turn}{80}{0}\end{turn}  & \begin{turn}{80}{0}\end{turn}  & \begin{turn}{80}{0}\end{turn}  & \begin{turn}{80}{0}\end{turn}  & \begin{turn}{80}{0}\end{turn}  & \begin{turn}{80}{0}\end{turn} \\ 
     \hline
 13 & \begin{turn}{80}{3}\end{turn} &  \begin{turn}{80}{4}\end{turn} &  \begin{turn}{80}{21}\end{turn} &  \begin{turn}{80}{34}\end{turn} &  \begin{turn}{80}{80}\end{turn} &  \begin{turn}{80}{96}\end{turn} &  \begin{turn}{80}{130}\end{turn} &  \begin{turn}{80}{96}\end{turn} &  \begin{turn}{80}{80}\end{turn} &  \begin{turn}{80}{34}\end{turn} &  \begin{turn}{80}{21}\end{turn} &  \begin{turn}{80}{4}\end{turn} & \begin{turn}{80}{3}\end{turn}  & \begin{turn}{80}{0}\end{turn}  & \begin{turn}{80}{0}\end{turn}  & \begin{turn}{80}{0}\end{turn}  & \begin{turn}{80}{0}\end{turn}  & \begin{turn}{80}{0}\end{turn}  & \begin{turn}{80}{0}\end{turn}  & \begin{turn}{80}{0}\end{turn}  & \begin{turn}{80}{0}\end{turn}  & \begin{turn}{80}{0}\end{turn}  & \begin{turn}{80}{0}\end{turn}  & \begin{turn}{80}{0}\end{turn} \\ 
      \hline
 14 & \begin{turn}{80}{-3}\end{turn} &  \begin{turn}{80}{-3}\end{turn} &  \begin{turn}{80}{-11}\end{turn} &  \begin{turn}{80}{1}\end{turn} &  \begin{turn}{80}{12}\end{turn} &  \begin{turn}{80}{56}\end{turn} &  \begin{turn}{80}{76}\end{turn} &  \begin{turn}{80}{76}\end{turn} &  \begin{turn}{80}{56}\end{turn} &  \begin{turn}{80}{12}\end{turn} &  \begin{turn}{80}{1}\end{turn} &  \begin{turn}{80}{-11}\end{turn} &  \begin{turn}{80}{-3}\end{turn} & \begin{turn}{80}{-3}\end{turn}  & \begin{turn}{80}{0}\end{turn}  & \begin{turn}{80}{0}\end{turn}  & \begin{turn}{80}{0}\end{turn}  & \begin{turn}{80}{0}\end{turn}  & \begin{turn}{80}{0}\end{turn}  & \begin{turn}{80}{0}\end{turn}  & \begin{turn}{80}{0}\end{turn}  & \begin{turn}{80}{0}\end{turn}  & \begin{turn}{80}{0}\end{turn}  & \begin{turn}{80}{0}\end{turn}   \\ 
     \hline
 15 & \begin{turn}{80}{3}\end{turn} &  \begin{turn}{80}{5}\end{turn} &  \begin{turn}{80}{28}\end{turn} &  \begin{turn}{80}{52}\end{turn} &  \begin{turn}{80}{146}\end{turn} &  \begin{turn}{80}{220}\end{turn} &  \begin{turn}{80}{353}\end{turn} &  \begin{turn}{80}{358}\end{turn} &  \begin{turn}{80}{353}\end{turn} &  \begin{turn}{80}{220}\end{turn} &  \begin{turn}{80}{146}\end{turn} &  \begin{turn}{80}{52}\end{turn} &  \begin{turn}{80}{28}\end{turn} &  \begin{turn}{80}{5}\end{turn} & \begin{turn}{80}{3}\end{turn}  & \begin{turn}{80}{0}\end{turn}  & \begin{turn}{80}{0}\end{turn}  & \begin{turn}{80}{0}\end{turn}  & \begin{turn}{80}{0}\end{turn}  & \begin{turn}{80}{0}\end{turn}  & \begin{turn}{80}{0}\end{turn}  & \begin{turn}{80}{0}\end{turn}  & \begin{turn}{80}{0}\end{turn}  & \begin{turn}{80}{0}\end{turn}   \\ 
   \hline
 16 & \begin{turn}{80}{-3}\end{turn} &  \begin{turn}{80}{-3}\end{turn} &  \begin{turn}{80}{-14}\end{turn} &  \begin{turn}{80}{10}\end{turn} &  \begin{turn}{80}{42}\end{turn} &  \begin{turn}{80}{168}\end{turn} &  \begin{turn}{80}{293}\end{turn} &  \begin{turn}{80}{405}\end{turn} &  \begin{turn}{80}{405}\end{turn} &  \begin{turn}{80}{293}\end{turn} &  \begin{turn}{80}{168}\end{turn} &  \begin{turn}{80}{42}\end{turn} &  \begin{turn}{80}{10}\end{turn} &  \begin{turn}{80}{-14}\end{turn} &  \begin{turn}{80}{-3}\end{turn} & \begin{turn}{80}{-3}\end{turn}  & \begin{turn}{80}{0}\end{turn}  & \begin{turn}{80}{0}\end{turn}  & \begin{turn}{80}{0}\end{turn}  & \begin{turn}{80}{0}\end{turn}  & \begin{turn}{80}{0}\end{turn}  & \begin{turn}{80}{0}\end{turn}  & \begin{turn}{80}{0}\end{turn}  & \begin{turn}{80}{0}\end{turn}  \\ 
   \hline
 17 & \begin{turn}{80}{3}\end{turn} &  \begin{turn}{80}{5}\end{turn} &  \begin{turn}{80}{36}\end{turn} &  \begin{turn}{80}{73}\end{turn} &  \begin{turn}{80}{241}\end{turn} &  \begin{turn}{80}{448}\end{turn} &  \begin{turn}{80}{846}\end{turn} &  \begin{turn}{80}{1090}\end{turn} &  \begin{turn}{80}{1300}\end{turn} &  \begin{turn}{80}{1090}\end{turn} &  \begin{turn}{80}{846}\end{turn} &  \begin{turn}{80}{448}\end{turn} &  \begin{turn}{80}{241}\end{turn} &  \begin{turn}{80}{73}\end{turn} &  \begin{turn}{80}{36}\end{turn} &  \begin{turn}{80}{5}\end{turn} & \begin{turn}{80}{3}\end{turn}  & \begin{turn}{80}{0}\end{turn}  & \begin{turn}{80}{0}\end{turn}  & \begin{turn}{80}{0}\end{turn}  & \begin{turn}{80}{0}\end{turn}  & \begin{turn}{80}{0}\end{turn}  & \begin{turn}{80}{0}\end{turn}  & \begin{turn}{80}{0}\end{turn}  \\ 
      \hline
 18 & \begin{turn}{80}{-3}\end{turn} &  \begin{turn}{80}{-3}\end{turn} &  \begin{turn}{80}{-17}\end{turn} &  \begin{turn}{80}{21}\end{turn} &  \begin{turn}{80}{95}\end{turn} &  \begin{turn}{80}{393}\end{turn} &  \begin{turn}{80}{814}\end{turn} &  \begin{turn}{80}{1386}\end{turn} &  \begin{turn}{80}{1742}\end{turn} &  \begin{turn}{80}{1742}\end{turn} &  \begin{turn}{80}{1386}\end{turn} &  \begin{turn}{80}{814}\end{turn} &  \begin{turn}{80}{393}\end{turn} &  \begin{turn}{80}{95}\end{turn} &  \begin{turn}{80}{21}\end{turn} &  \begin{turn}{80}{-17}\end{turn} &  \begin{turn}{80}{-3}\end{turn} & \begin{turn}{80}{-3}\end{turn}  & \begin{turn}{80}{0}\end{turn}  & \begin{turn}{80}{0}\end{turn}  & \begin{turn}{80}{0}\end{turn}  & \begin{turn}{80}{0}\end{turn}  & \begin{turn}{80}{0}\end{turn}  & \begin{turn}{80}{0}\end{turn}  \\ 
    \hline
 19 & \begin{turn}{80}{4}\end{turn} &  \begin{turn}{80}{6}\end{turn} &  \begin{turn}{80}{45}\end{turn} &  \begin{turn}{80}{105}\end{turn} &  \begin{turn}{80}{384}\end{turn} & \begin{turn}{80}{840}\end{turn} & \begin{turn}{80}{1851}\end{turn} & \begin{turn}{80}{2904}\end{turn}&  \begin{turn}{80}{4098}\end{turn} &  \begin{turn}{80}{4370}\end{turn} &  \begin{turn}{80}{4098}\end{turn} &  \begin{turn}{80}{2904}\end{turn}&  \begin{turn}{80}{1851}\end{turn} &  \begin{turn}{80}{840}\end{turn} &  \begin{turn}{80}{384}\end{turn} &  \begin{turn}{80}{105}\end{turn} &  \begin{turn}{80}{45}\end{turn} &  \begin{turn}{80}{6}\end{turn} & \begin{turn}{80}{4}\end{turn}  & \begin{turn}{80}{0}\end{turn}  & \begin{turn}{80}{0}\end{turn}  & \begin{turn}{80}{0}\end{turn}  & \begin{turn}{80}{0}\end{turn}  & \begin{turn}{80}{0}\end{turn}   \\ 
   \hline
 20 & \begin{turn}{80}{-4}\end{turn} &  \begin{turn}{80}{-4}\end{turn} &  \begin{turn}{80}{-21}\end{turn} &  \begin{turn}{80}{33}\end{turn} &  \begin{turn}{80}{174}\end{turn} &  \begin{turn}{80}{792}\end{turn} &  \begin{turn}{80}{1899}\end{turn} & \begin{turn}{80}{3795}\end{turn}& \begin{turn}{80}{5742}\end{turn}&  \begin{turn}{80}{7114}\end{turn} &  \begin{turn}{80}{7114}\end{turn} & \begin{turn}{80}{5742}\end{turn} & \begin{turn}{80}{3795}\end{turn} & \begin{turn}{80}{1899}\end{turn} &  \begin{turn}{80}{792}\end{turn}&  \begin{turn}{80}{174}\end{turn} &  \begin{turn}{80}{33}\end{turn} &  \begin{turn}{80}{-21}\end{turn} &  \begin{turn}{80}{-4}\end{turn} & \begin{turn}{80}{-4}\end{turn}  & \begin{turn}{80}{0}\end{turn}  & \begin{turn}{80}{0}\end{turn}  & \begin{turn}{80}{0}\end{turn}  & \begin{turn}{80}{0}\end{turn}   
   \\ \hline
 21 & \begin{turn}{80}{4}\end{turn} &  \begin{turn}{80}{7}\end{turn} &  \begin{turn}{80}{55}\end{turn} &  \begin{turn}{80}{141}\end{turn} &  \begin{turn}{80}{582}\end{turn} &  \begin{turn}{80}{1473}\end{turn} & \begin{turn}{80}{3702}\end{turn}& \begin{turn}{80}{6885}\end{turn} &  \begin{turn}{80}{11322}\end{turn}&  \begin{turn}{80}{14630}\end{turn} & \begin{turn}{80}{16390}\end{turn}& \begin{turn}{80}{14630}\end{turn}& \begin{turn}{80}{11322}\end{turn}& \begin{turn}{80}{6885}\end{turn}&  \begin{turn}{80}{3702}\end{turn} &  \begin{turn}{80}{1473}\end{turn} &  \begin{turn}{80}{582}\end{turn} &  \begin{turn}{80}{141}\end{turn} &  \begin{turn}{80}{55}\end{turn} &  \begin{turn}{80}{7}\end{turn} & \begin{turn}{80}{4}\end{turn}  & \begin{turn}{80}{0}\end{turn}  & \begin{turn}{80}{0}\end{turn}  & \begin{turn}{80}{0}\end{turn}  \\ 
  \hline
 22 & \begin{turn}{80}{-4}\end{turn} &  \begin{turn}{80}{-4}\end{turn} &  \begin{turn}{80}{-25}\end{turn} &  \begin{turn}{80}{56}\end{turn} & \begin{turn}{80}{300}\end{turn}& \begin{turn}{80}{1452}\end{turn} &  \begin{turn}{80}{3975}\end{turn} & \begin{turn}{80}{9060}\end{turn} & \begin{turn}{80}{15960}\end{turn} & \begin{turn}{80}{23432}\end{turn}& \begin{turn}{80}{28154}\end{turn}& \begin{turn}{80}{28154}\end{turn} & \begin{turn}{80}{23432}\end{turn}& \begin{turn}{80}{15960}\end{turn}&  \begin{turn}{80}{9060}\end{turn} &  \begin{turn}{80}{3975}\end{turn} &  \begin{turn}{80}{1452}\end{turn} &  \begin{turn}{80}{300}\end{turn} &  \begin{turn}{80}{56}\end{turn} &  \begin{turn}{80}{-25}\end{turn} &  \begin{turn}{80}{-4}\end{turn} & \begin{turn}{80}{-4}\end{turn}  & \begin{turn}{80}{0}\end{turn}  & \begin{turn}{80}{0}\end{turn}  \\ 
     \hline
 23  & \begin{turn}{80}{4}\end{turn} &  \begin{turn}{80}{7}\end{turn} &  \begin{turn}{80}{66}\end{turn}& \begin{turn}{80}{181}\end{turn}& \begin{turn}{80}{841}\end{turn} & \begin{turn}{80}{2442}\end{turn}& \begin{turn}{80}{6912}\end{turn} & \begin{turn}{80}{14865}\end{turn}& \begin{turn}{80}{28050}\end{turn} & \begin{turn}{80}{42595}\end{turn}& \begin{turn}{80}{55839}\end{turn}&  \begin{turn}{80}{60172}\end{turn}& \begin{turn}{80}{55839}\end{turn}&  \begin{turn}{80}{42595}\end{turn} &  \begin{turn}{80}{28050}\end{turn}&  \begin{turn}{80}{14865}\end{turn} & \begin{turn}{80}{6912}\end{turn} &  \begin{turn}{80}{2442}\end{turn}&  \begin{turn}{80}{841}\end{turn} &  \begin{turn}{80}{181}\end{turn} &  \begin{turn}{80}{66}\end{turn} &  \begin{turn}{80}{7}\end{turn} & \begin{turn}{80}{4}\end{turn}  & \begin{turn}{80}{0}\end{turn} \\ \hline
\end{tabular}
\end{center}}
\caption{\small Table of Euler characteristics $\chi_{s_1,s_2,t}^{\pi g}$ by genus $g=2$, complexity $t$ and Hodge degree $s_2$ of $\pi_*\overline{\mbox{Emb}}_c(\coprod_{i=1}^2 \rbb^{m_i}, \rdbb) \otimes \qbb$ for $m_1$, $m_2$ and $d$ odd $(s_1=t-s_2-1)$.}
\end{table}

\begin{table}[ht!]

{\tiny \begin{center}
\begin{tabular}{|c|c|c|c|c|c|c|c|c|c|c|c|c|c|c|c|c|c|c|c|c|c|c|c|c|}
\hline
$t$ &\multicolumn{23}{|c|}{Hodge degree $s_2$}& \\ \hline
  & 0 & 1 & 2 & 3 & 4 & 5 & 6 & 7 & 8 & 9 & 10 & 11 & 12 & 13 & 14 & 15 & 16 & 17 & 18 & 19 & 20 & 21 & 22 & 23 \\ \hline
1 & \begin{turn}{80}{0}\end{turn} & \begin{turn}{80}{0}\end{turn} & \begin{turn}{80}{0}\end{turn} & \begin{turn}{80}{0}\end{turn} & \begin{turn}{80}{0}\end{turn} & \begin{turn}{80}{0}\end{turn} & \begin{turn}{80}{0}\end{turn} & \begin{turn}{80}{0}\end{turn} & \begin{turn}{80}{0}\end{turn} & \begin{turn}{80}{0}\end{turn} & \begin{turn}{80}{0}\end{turn} & \begin{turn}{80}{0}\end{turn} & \begin{turn}{80}{0}\end{turn} & \begin{turn}{80}{0}\end{turn} & \begin{turn}{80}{0}\end{turn} & \begin{turn}{80}{0}\end{turn} & \begin{turn}{80}{0}\end{turn} & \begin{turn}{80}{0}\end{turn} & \begin{turn}{80}{0}\end{turn} & \begin{turn}{80}{0}\end{turn} & \begin{turn}{80}{0}\end{turn} & \begin{turn}{80}{0}\end{turn} & \begin{turn}{80}{0}\end{turn} & \begin{turn}{80}{0}\end{turn} \\ 
      \hline
 2 & \begin{turn}{80}{0}\end{turn} & \begin{turn}{80}{0}\end{turn} & \begin{turn}{80}{0}\end{turn} & \begin{turn}{80}{0}\end{turn} & \begin{turn}{80}{0}\end{turn} & \begin{turn}{80}{0}\end{turn} & \begin{turn}{80}{0}\end{turn} & \begin{turn}{80}{0}\end{turn} & \begin{turn}{80}{0}\end{turn} & \begin{turn}{80}{0}\end{turn} & \begin{turn}{80}{0}\end{turn} & \begin{turn}{80}{0}\end{turn} & \begin{turn}{80}{0}\end{turn} & \begin{turn}{80}{0}\end{turn} & \begin{turn}{80}{0}\end{turn} & \begin{turn}{80}{0}\end{turn} & \begin{turn}{80}{0}\end{turn} & \begin{turn}{80}{0}\end{turn} & \begin{turn}{80}{0}\end{turn} & \begin{turn}{80}{0}\end{turn} & \begin{turn}{80}{0}\end{turn} & \begin{turn}{80}{0}\end{turn} & \begin{turn}{80}{0}\end{turn} & \begin{turn}{80}{0}\end{turn}  \\ 
     \hline
 3 & \begin{turn}{80}{-1}\end{turn} &\begin{turn}{80}{-1}\end{turn} & \begin{turn}{80}{0}\end{turn} & \begin{turn}{80}{0}\end{turn} & \begin{turn}{80}{0}\end{turn} & \begin{turn}{80}{0}\end{turn} & \begin{turn}{80}{0}\end{turn} & \begin{turn}{80}{0}\end{turn} & \begin{turn}{80}{0}\end{turn} & \begin{turn}{80}{0}\end{turn} & \begin{turn}{80}{0}\end{turn} & \begin{turn}{80}{0}\end{turn} & \begin{turn}{80}{0}\end{turn} & \begin{turn}{80}{0}\end{turn} & \begin{turn}{80}{0}\end{turn} & \begin{turn}{80}{0}\end{turn} & \begin{turn}{80}{0}\end{turn} & \begin{turn}{80}{0}\end{turn} & \begin{turn}{80}{0}\end{turn} & \begin{turn}{80}{0}\end{turn} & \begin{turn}{80}{0}\end{turn} & \begin{turn}{80}{0}\end{turn} & \begin{turn}{80}{0}\end{turn} & \begin{turn}{80}{0}\end{turn}  \\ 
    \hline
4 & \begin{turn}{80}{1}\end{turn} &  \begin{turn}{80}{1}\end{turn} & \begin{turn}{80}{1}\end{turn} & \begin{turn}{80}{0}\end{turn} & \begin{turn}{80}{0}\end{turn} & \begin{turn}{80}{0}\end{turn} & \begin{turn}{80}{0}\end{turn} & \begin{turn}{80}{0}\end{turn} & \begin{turn}{80}{0}\end{turn} & \begin{turn}{80}{0}\end{turn} & \begin{turn}{80}{0}\end{turn} & \begin{turn}{80}{0}\end{turn} & \begin{turn}{80}{0}\end{turn} & \begin{turn}{80}{0}\end{turn} & \begin{turn}{80}{0}\end{turn} & \begin{turn}{80}{0}\end{turn} & \begin{turn}{80}{0}\end{turn} & \begin{turn}{80}{0}\end{turn} & \begin{turn}{80}{0}\end{turn} & \begin{turn}{80}{0}\end{turn} & \begin{turn}{80}{0}\end{turn} & \begin{turn}{80}{0}\end{turn} & \begin{turn}{80}{0}\end{turn} & \begin{turn}{80}{0}\end{turn}  \\ 
     \hline
 5 & \begin{turn}{80}{-1}\end{turn} & \begin{turn}{80}{-2}\end{turn} & \begin{turn}{80}{-2}\end{turn} & \begin{turn}{80}{-1}\end{turn} & \begin{turn}{80}{0}\end{turn} & \begin{turn}{80}{0}\end{turn} & \begin{turn}{80}{0}\end{turn} & \begin{turn}{80}{0}\end{turn} & \begin{turn}{80}{0}\end{turn} & \begin{turn}{80}{0}\end{turn} & \begin{turn}{80}{0}\end{turn} & \begin{turn}{80}{0}\end{turn} & \begin{turn}{80}{0}\end{turn} & \begin{turn}{80}{0}\end{turn} & \begin{turn}{80}{0}\end{turn} & \begin{turn}{80}{0}\end{turn} & \begin{turn}{80}{0}\end{turn} & \begin{turn}{80}{0}\end{turn} & \begin{turn}{80}{0}\end{turn} & \begin{turn}{80}{0}\end{turn} & \begin{turn}{80}{0}\end{turn} & \begin{turn}{80}{0}\end{turn} & \begin{turn}{80}{0}\end{turn} & \begin{turn}{80}{0}\end{turn} \\ 
     \hline 
 6 & \begin{turn}{80}{2}\end{turn} & \begin{turn}{80}{2}\end{turn} & \begin{turn}{80}{5}\end{turn} & \begin{turn}{80}{2}\end{turn} & \begin{turn}{80}{2}\end{turn} & \begin{turn}{80}{0}\end{turn} & \begin{turn}{80}{0}\end{turn} & \begin{turn}{80}{0}\end{turn} & \begin{turn}{80}{0}\end{turn} & \begin{turn}{80}{0}\end{turn} & \begin{turn}{80}{0}\end{turn} & \begin{turn}{80}{0}\end{turn} & \begin{turn}{80}{0}\end{turn} & \begin{turn}{80}{0}\end{turn} & \begin{turn}{80}{0}\end{turn} & \begin{turn}{80}{0}\end{turn} & \begin{turn}{80}{0}\end{turn} & \begin{turn}{80}{0}\end{turn} & \begin{turn}{80}{0}\end{turn} & \begin{turn}{80}{0}\end{turn} & \begin{turn}{80}{0}\end{turn} & \begin{turn}{80}{0}\end{turn} & \begin{turn}{80}{0}\end{turn} & \begin{turn}{80}{0}\end{turn} \\ 
   \hline 
 7 & \begin{turn}{80}{-3}\end{turn} & \begin{turn}{80}{-4}\end{turn} & \begin{turn}{80}{-7}\end{turn} & \begin{turn}{80}{-7}\end{turn} & \begin{turn}{80}{-4}\end{turn} & \begin{turn}{80}{-3}\end{turn} & \begin{turn}{80}{0}\end{turn} & \begin{turn}{80}{0}\end{turn} & \begin{turn}{80}{0}\end{turn} & \begin{turn}{80}{0}\end{turn} & \begin{turn}{80}{0}\end{turn} & \begin{turn}{80}{0}\end{turn} & \begin{turn}{80}{0}\end{turn} & \begin{turn}{80}{0}\end{turn} & \begin{turn}{80}{0}\end{turn} & \begin{turn}{80}{0}\end{turn} & \begin{turn}{80}{0}\end{turn} & \begin{turn}{80}{0}\end{turn} & \begin{turn}{80}{0}\end{turn} & \begin{turn}{80}{0}\end{turn} & \begin{turn}{80}{0}\end{turn} & \begin{turn}{80}{0}\end{turn} & \begin{turn}{80}{0}\end{turn} & \begin{turn}{80}{0}\end{turn} \\ 
     \hline 
 8  & \begin{turn}{80}{4}\end{turn} &  \begin{turn}{80}{5}\end{turn} & \begin{turn}{80}{12}\end{turn} & \begin{turn}{80}{11}\end{turn} & \begin{turn}{80}{12}\end{turn} & \begin{turn}{80}{5}\end{turn} & \begin{turn}{80}{4}\end{turn} & \begin{turn}{80}{0}\end{turn} & \begin{turn}{80}{0}\end{turn} & \begin{turn}{80}{0}\end{turn} & \begin{turn}{80}{0}\end{turn} & \begin{turn}{80}{0}\end{turn} & \begin{turn}{80}{0}\end{turn} & \begin{turn}{80}{0}\end{turn} & \begin{turn}{80}{0}\end{turn} & \begin{turn}{80}{0}\end{turn} & \begin{turn}{80}{0}\end{turn} & \begin{turn}{80}{0}\end{turn} & \begin{turn}{80}{0}\end{turn} & \begin{turn}{80}{0}\end{turn} & \begin{turn}{80}{0}\end{turn} & \begin{turn}{80}{0}\end{turn} & \begin{turn}{80}{0}\end{turn} & \begin{turn}{80}{0}\end{turn} \\ 
       \hline 
 9 & \begin{turn}{80}{-5}\end{turn} & \begin{turn}{80}{-7}\end{turn} & \begin{turn}{80}{-16}\end{turn} & \begin{turn}{80}{-19}\end{turn} & \begin{turn}{80}{-19}\end{turn} & \begin{turn}{80}{-16}\end{turn} & \begin{turn}{80}{-7}\end{turn} & \begin{turn}{80}{-5}\end{turn} & \begin{turn}{80}{0}\end{turn} & \begin{turn}{80}{0}\end{turn} & \begin{turn}{80}{0}\end{turn} & \begin{turn}{80}{0}\end{turn} & \begin{turn}{80}{0}\end{turn} & \begin{turn}{80}{0}\end{turn} & \begin{turn}{80}{0}\end{turn} & \begin{turn}{80}{0}\end{turn} & \begin{turn}{80}{0}\end{turn} & \begin{turn}{80}{0}\end{turn} & \begin{turn}{80}{0}\end{turn} & \begin{turn}{80}{0}\end{turn} & \begin{turn}{80}{0}\end{turn} & \begin{turn}{80}{0}\end{turn} & \begin{turn}{80}{0}\end{turn} & \begin{turn}{80}{0}\end{turn} \\ 
      \hline
 10 & \begin{turn}{80}{6}\end{turn} & \begin{turn}{80}{8}\end{turn} & \begin{turn}{80}{24}\end{turn} & \begin{turn}{80}{24}\end{turn} & \begin{turn}{80}{37}\end{turn} & \begin{turn}{80}{24}\end{turn} & \begin{turn}{80}{24}\end{turn} & \begin{turn}{80}{8}\end{turn} & \begin{turn}{80}{6}\end{turn} & \begin{turn}{80}{0}\end{turn} & \begin{turn}{80}{0}\end{turn} & \begin{turn}{80}{0}\end{turn} & \begin{turn}{80}{0}\end{turn} & \begin{turn}{80}{0}\end{turn} & \begin{turn}{80}{0}\end{turn} & \begin{turn}{80}{0}\end{turn} & \begin{turn}{80}{0}\end{turn} & \begin{turn}{80}{0}\end{turn} & \begin{turn}{80}{0}\end{turn} & \begin{turn}{80}{0}\end{turn} & \begin{turn}{80}{0}\end{turn} & \begin{turn}{80}{0}\end{turn} & \begin{turn}{80}{0}\end{turn} & \begin{turn}{80}{0}\end{turn} \\ 
     \hline
 11 & \begin{turn}{80}{-6}\end{turn} & \begin{turn}{80}{-10}\end{turn} & \begin{turn}{80}{-30}\end{turn} & \begin{turn}{80}{-35}\end{turn} & \begin{turn}{80}{-50}\end{turn} & \begin{turn}{80}{-50}\end{turn} & \begin{turn}{80}{-35}\end{turn} & \begin{turn}{80}{-30}\end{turn} & \begin{turn}{80}{-10}\end{turn} & \begin{turn}{80}{-6}\end{turn} & \begin{turn}{80}{0}\end{turn} & \begin{turn}{80}{0}\end{turn} & \begin{turn}{80}{0}\end{turn} & \begin{turn}{80}{0}\end{turn} & \begin{turn}{80}{0}\end{turn} & \begin{turn}{80}{0}\end{turn} & \begin{turn}{80}{0}\end{turn} & \begin{turn}{80}{0}\end{turn} & \begin{turn}{80}{0}\end{turn} & \begin{turn}{80}{0}\end{turn} & \begin{turn}{80}{0}\end{turn} & \begin{turn}{80}{0}\end{turn} & \begin{turn}{80}{0}\end{turn} & \begin{turn}{80}{0}\end{turn} \\      \hline
 12 & \begin{turn}{80}{7}\end{turn} & \begin{turn}{80}{11}\end{turn} & \begin{turn}{80}{41}\end{turn} & \begin{turn}{80}{45}\end{turn} & \begin{turn}{80}{82}\end{turn} & \begin{turn}{80}{70}\end{turn} & \begin{turn}{80}{82}\end{turn} & \begin{turn}{80}{45}\end{turn} & \begin{turn}{80}{41}\end{turn} & \begin{turn}{80}{11}\end{turn} & \begin{turn}{80}{7}\end{turn} & \begin{turn}{80}{0}\end{turn} & \begin{turn}{80}{0}\end{turn} & \begin{turn}{80}{0}\end{turn} & \begin{turn}{80}{0}\end{turn} & \begin{turn}{80}{0}\end{turn} & \begin{turn}{80}{0}\end{turn} & \begin{turn}{80}{0}\end{turn} & \begin{turn}{80}{0}\end{turn} & \begin{turn}{80}{0}\end{turn} & \begin{turn}{80}{0}\end{turn} & \begin{turn}{80}{0}\end{turn} & \begin{turn}{80}{0}\end{turn} & \begin{turn}{80}{0}\end{turn} \\ 
                 \hline
 13 & \begin{turn}{80}{-9}\end{turn} & \begin{turn}{80}{-14}\end{turn} & \begin{turn}{80}{-50}\end{turn} & \begin{turn}{80}{-65}\end{turn} & \begin{turn}{80}{-110}\end{turn} & \begin{turn}{80}{-120}\end{turn} & \begin{turn}{80}{-120}\end{turn} & \begin{turn}{80}{-110}\end{turn} & \begin{turn}{80}{-65}\end{turn} & \begin{turn}{80}{-50}\end{turn} & \begin{turn}{80}{-14}\end{turn} & \begin{turn}{80}{-9}\end{turn} & \begin{turn}{80}{0}\end{turn} & \begin{turn}{80}{0}\end{turn} & \begin{turn}{80}{0}\end{turn} & \begin{turn}{80}{0}\end{turn} & \begin{turn}{80}{0}\end{turn} & \begin{turn}{80}{0}\end{turn} & \begin{turn}{80}{0}\end{turn} & \begin{turn}{80}{0}\end{turn} & \begin{turn}{80}{0}\end{turn} & \begin{turn}{80}{0}\end{turn} & \begin{turn}{80}{0}\end{turn} & \begin{turn}{80}{0}\end{turn} \\ 
            \hline
 14 & \begin{turn}{80}{11}\end{turn} & \begin{turn}{80}{16}\end{turn} & \begin{turn}{80}{65}\end{turn} & \begin{turn}{80}{82}\end{turn} & \begin{turn}{80}{167}\end{turn} & \begin{turn}{80}{160}\end{turn} & \begin{turn}{80}{226}\end{turn} & \begin{turn}{80}{160}\end{turn} & \begin{turn}{80}{167}\end{turn} & \begin{turn}{80}{82}\end{turn} & \begin{turn}{80}{65}\end{turn} & \begin{turn}{80}{16}\end{turn} & \begin{turn}{80}{11}\end{turn} & \begin{turn}{80}{0}\end{turn} & \begin{turn}{80}{0}\end{turn} & \begin{turn}{80}{0}\end{turn} & \begin{turn}{80}{0}\end{turn} & \begin{turn}{80}{0}\end{turn} & \begin{turn}{80}{0}\end{turn} & \begin{turn}{80}{0}\end{turn} & \begin{turn}{80}{0}\end{turn} & \begin{turn}{80}{0}\end{turn} & \begin{turn}{80}{0}\end{turn} & \begin{turn}{80}{0}\end{turn} \\ 
               \hline
 15 & \begin{turn}{80}{-12}\end{turn} & \begin{turn}{80}{-19}\end{turn} & \begin{turn}{80}{-77}\end{turn} & \begin{turn}{80}{-107}\end{turn} & \begin{turn}{80}{-212}\end{turn} & \begin{turn}{80}{-245}\end{turn} & \begin{turn}{80}{-306}\end{turn} & \begin{turn}{80}{-306}\end{turn} & \begin{turn}{80}{-245}\end{turn} & \begin{turn}{80}{-212}\end{turn} & \begin{turn}{80}{-107}\end{turn} & \begin{turn}{80}{-77}\end{turn} & \begin{turn}{80}{-19}\end{turn} & \begin{turn}{80}{-12}\end{turn} & \begin{turn}{80}{0}\end{turn} & \begin{turn}{80}{0}\end{turn} & \begin{turn}{80}{0}\end{turn} & \begin{turn}{80}{0}\end{turn} & \begin{turn}{80}{0}\end{turn} & \begin{turn}{80}{0}\end{turn} & \begin{turn}{80}{0}\end{turn} & \begin{turn}{80}{0}\end{turn} & \begin{turn}{80}{0}\end{turn} & \begin{turn}{80}{0}\end{turn} \\ 
              \hline
 16 & \begin{turn}{80}{13}\end{turn} & \begin{turn}{80}{21}\end{turn} & \begin{turn}{80}{96}\end{turn} & \begin{turn}{80}{126}\end{turn} & \begin{turn}{80}{295}\end{turn} & \begin{turn}{80}{315}\end{turn} & \begin{turn}{80}{498}\end{turn} & \begin{turn}{80}{420}\end{turn} & \begin{turn}{80}{498}\end{turn} & \begin{turn}{80}{315}\end{turn} & \begin{turn}{80}{295}\end{turn} & \begin{turn}{80}{126}\end{turn} & \begin{turn}{80}{96}\end{turn} & \begin{turn}{80}{21}\end{turn} & \begin{turn}{80}{13}\end{turn} & \begin{turn}{80}{0}\end{turn} & \begin{turn}{80}{0}\end{turn} & \begin{turn}{80}{0}\end{turn} & \begin{turn}{80}{0}\end{turn} & \begin{turn}{80}{0}\end{turn} & \begin{turn}{80}{0}\end{turn} & \begin{turn}{80}{0}\end{turn} & \begin{turn}{80}{0}\end{turn} & \begin{turn}{80}{0}\end{turn} \\ 
             \hline
 17 & \begin{turn}{80}{-14}\end{turn} & \begin{turn}{80}{-24}\end{turn} & \begin{turn}{80}{-112}\end{turn} & \begin{turn}{80}{-156}\end{turn} & \begin{turn}{80}{-364}\end{turn} & \begin{turn}{80}{-448}\end{turn} & \begin{turn}{80}{-648}\end{turn} & \begin{turn}{80}{-700}\end{turn} & \begin{turn}{80}{-700}\end{turn} & \begin{turn}{80}{-648}\end{turn} & \begin{turn}{80}{-448}\end{turn} & \begin{turn}{80}{-364}\end{turn} & \begin{turn}{80}{-156}\end{turn} & \begin{turn}{80}{-112}\end{turn} & \begin{turn}{80}{-24}\end{turn} & \begin{turn}{80}{-14}\end{turn} & \begin{turn}{80}{0}\end{turn} & \begin{turn}{80}{0}\end{turn} & \begin{turn}{80}{0}\end{turn} & \begin{turn}{80}{0}\end{turn} & \begin{turn}{80}{0}\end{turn} & \begin{turn}{80}{0}\end{turn} & \begin{turn}{80}{0}\end{turn} & \begin{turn}{80}{0}\end{turn} \\ 
               \hline
 18 & \begin{turn}{80}{16}\end{turn} & \begin{turn}{80}{26}\end{turn} & \begin{turn}{80}{136}\end{turn} & \begin{turn}{80}{184}\end{turn} & \begin{turn}{80}{486}\end{turn} & \begin{turn}{80}{560}\end{turn} & \begin{turn}{80}{984}\end{turn}  & \begin{turn}{80}{928}\end{turn} & \begin{turn}{80}{1244}\end{turn} & \begin{turn}{80}{928}\end{turn} & \begin{turn}{80}{984}\end{turn} & \begin{turn}{80}{560}\end{turn} & \begin{turn}{80}{486}\end{turn} & \begin{turn}{80}{184}\end{turn} & \begin{turn}{80}{136}\end{turn} & \begin{turn}{80}{26}\end{turn} & \begin{turn}{80}{16}\end{turn} & \begin{turn}{80}{0}\end{turn} & \begin{turn}{80}{0}\end{turn} & \begin{turn}{80}{0}\end{turn} & \begin{turn}{80}{0}\end{turn} & \begin{turn}{80}{0}\end{turn} & \begin{turn}{80}{0}\end{turn} & \begin{turn}{80}{0}\end{turn} \\ 
            \hline
 19 & \begin{turn}{80}{-18}\end{turn} & \begin{turn}{80}{-30}\end{turn} & \begin{turn}{80}{-156}\end{turn} & \begin{turn}{80}{-228}\end{turn} & \begin{turn}{80}{-588}\end{turn} & \begin{turn}{80}{-756}\end{turn} & \begin{turn}{80}{-1260}\end{turn} & \begin{turn}{80}{-1428}\end{turn} & \begin{turn}{80}{-1680}\end{turn}& \begin{turn}{80}{-1680}\end{turn} & \begin{turn}{80}{-1428}\end{turn} & \begin{turn}{80}{-1260}\end{turn} & \begin{turn}{80}{-756}\end{turn} & \begin{turn}{80}{-588}\end{turn} & \begin{turn}{80}{-228}\end{turn} & \begin{turn}{80}{-156}\end{turn} & \begin{turn}{80}{-30}\end{turn} & \begin{turn}{80}{-18}\end{turn} & \begin{turn}{80}{0}\end{turn} & \begin{turn}{80}{0}\end{turn} & \begin{turn}{80}{0}\end{turn} & \begin{turn}{80}{0}\end{turn} & \begin{turn}{80}{0}\end{turn} & \begin{turn}{80}{0}\end{turn} \\ 
          \hline
 20 & \begin{turn}{80}{20}\end{turn} & \begin{turn}{80}{33}\end{turn} & \begin{turn}{80}{185}\end{turn} & \begin{turn}{80}{267}\end{turn} & \begin{turn}{80}{758}\end{turn} &  \begin{turn}{80}{924}\end{turn} &  \begin{turn}{80}{1805}\end{turn}&  \begin{turn}{80}{1848}\end{turn}&  \begin{turn}{80}{2718}\end{turn}&  \begin{turn}{80}{2320}\end{turn}& \begin{turn}{80}{2718}\end{turn} & \begin{turn}{80}{1848}\end{turn} & \begin{turn}{80}{1805}\end{turn}& \begin{turn}{80}{924}\end{turn} & \begin{turn}{80}{758}\end{turn} & \begin{turn}{80}{267}\end{turn} & \begin{turn}{80}{185}\end{turn} & \begin{turn}{80}{33}\end{turn} & \begin{turn}{80}{20}\end{turn} & \begin{turn}{80}{0}\end{turn} & \begin{turn}{80}{0}\end{turn} & \begin{turn}{80}{0}\end{turn} & \begin{turn}{80}{0}\end{turn} & \begin{turn}{80}{0}\end{turn} \\ 
          \hline
 21 & \begin{turn}{80}{-22}\end{turn} & \begin{turn}{80}{-37}\end{turn} & \begin{turn}{80}{-210}\end{turn} & \begin{turn}{80}{-378}\end{turn} &  \begin{turn}{80}{-903}\end{turn} & \begin{turn}{80}{-1200}\end{turn} &  \begin{turn}{80}{-2247}\end{turn}&  \begin{turn}{80}{-2667}\end{turn}& \begin{turn}{80}{-3570}\end{turn} &  \begin{turn}{80}{-3790}\end{turn} &  \begin{turn}{80}{-3790}\end{turn}& \begin{turn}{80}{-3570}\end{turn}&  \begin{turn}{80}{-2667}\end{turn}& \begin{turn}{80}{-2247}\end{turn}& \begin{turn}{80}{-1200}\end{turn} & \begin{turn}{80}{-903}\end{turn} & \begin{turn}{80}{-318}\end{turn} & \begin{turn}{80}{-210}\end{turn} & \begin{turn}{80}{-37}\end{turn} & \begin{turn}{80}{-22}\end{turn} & \begin{turn}{80}{0}\end{turn} & \begin{turn}{80}{0}\end{turn} & \begin{turn}{80}{0}\end{turn} & \begin{turn}{80}{0}\end{turn} \\ 
                \hline
 22 & \begin{turn}{80}{24}\end{turn} &\begin{turn}{80}{40}\end{turn} &  \begin{turn}{80}{245}\end{turn} & \begin{turn}{80}{360}\end{turn} &  \begin{turn}{80}{1130}\end{turn} & \begin{turn}{80}{1440}\end{turn} & \begin{turn}{80}{3060}\end{turn} & \begin{turn}{80}{3360}\end{turn}& \begin{turn}{80}{5410}\end{turn} & \begin{turn}{80}{5040}\end{turn} &  \begin{turn}{80}{6510}\end{turn}& \begin{turn}{80}{5040}\end{turn} & \begin{turn}{80}{5410}\end{turn}& \begin{turn}{80}{3360}\end{turn}&  \begin{turn}{80}{3060}\end{turn}&  \begin{turn}{80}{1440}\end{turn}& \begin{turn}{80}{1130}\end{turn} & \begin{turn}{80}{360}\end{turn} & \begin{turn}{80}{245}\end{turn} & \begin{turn}{80}{40}\end{turn} & \begin{turn}{80}{24}\end{turn} & \begin{turn}{80}{0}\end{turn} & \begin{turn}{80}{0}\end{turn} & \begin{turn}{80}{0}\end{turn} \\ 
                \hline
 23 & \begin{turn}{80}{-25}\end{turn} & \begin{turn}{80}{-44}\end{turn} &  \begin{turn}{80}{-275}\end{turn}&  \begin{turn}{80}{-418}\end{turn}& \begin{turn}{80}{-1320}\end{turn}& \begin{turn}{80}{-1815}\end{turn} &  \begin{turn}{80}{-3729}\end{turn} & \begin{turn}{80}{-4620}\end{turn}& \begin{turn}{80}{-6930}\end{turn}& \begin{turn}{80}{-7682}\end{turn} & \begin{turn}{80}{-8778}\end{turn}&  \begin{turn}{80}{-8778}\end{turn}&  \begin{turn}{80}{-7682}\end{turn}& \begin{turn}{80}{-6930}\end{turn}& \begin{turn}{80}{-4620}\end{turn}& \begin{turn}{80}{-3729}\end{turn} & \begin{turn}{80}{-1815}\end{turn}& \begin{turn}{80}{-1320}\end{turn} & \begin{turn}{80}{-418}\end{turn} & \begin{turn}{80}{-275}\end{turn} & \begin{turn}{80}{-44}\end{turn} & \begin{turn}{80}{-25}\end{turn} & \begin{turn}{80}{0}\end{turn} & \begin{turn}{80}{0}\end{turn} \\ \hline
\end{tabular}
\end{center}}
\caption{\small Table of Euler characteristics $\chi_{s_1,s_2,t}^{\pi g}$ by genus $g=3$, complexity $t$ and Hodge degree $s_2$ of $\pi_*\overline{\mbox{Emb}}_c(\coprod_{i=1}^2 \rbb^{m_i}, \rdbb) \otimes \qbb$ for $m_1$, $m_2$ and $d$ odd $(s_1=t-s_2-2)$.}
\end{table}

\textsf{Department of Mathematics and Statistics, University of Regina}\\
 3737 Wascana Pkwy, Regina, SK S4S 0A2, Canada\\
\textit{E-mail address: pso748@uregina.ca}

\textsf{Department of Mathematics, Kansas State University\\
        138 Cardwell Hall, Manhattan, KS 66506, USA \\}
\textit{E-mail address: turchin@ksu.edu}

\end{document}